\documentclass[12pt,reqno]{amsart}
\usepackage[margin=1in]{geometry}
\usepackage{amsmath,amssymb,amsthm,graphicx,amsxtra, setspace}
\usepackage[utf8]{inputenc}
\usepackage{mathrsfs}
\usepackage{hyperref}
\usepackage{upgreek}
\usepackage{mathtools}
\usepackage{xcolor}

\usepackage[T1]{fontenc}  
\usepackage[utf8]{inputenc}
\usepackage[mathcal]{euscript}
\allowdisplaybreaks

\usepackage[pagewise]{lineno}

\DeclareMathAlphabet{\mathpzc}{OT1}{pzc}{m}{it}

\usepackage[cyr]{aeguill}

\colorlet{darkblue}{blue!50!black}

\hypersetup{
	colorlinks,%
	citecolor=blue,%
	filecolor=red,%
	linkcolor=red,%
	urlcolor=blue,%
	pdfnewwindow=true,%
	pdfstartview={FitH}
}

\newtheorem{theorem}{Theorem}[section]
\newtheorem{lemma}[theorem]{Lemma}
\newtheorem{proposition}[theorem]{Proposition}

\newtheorem{definition}[theorem]{Definition}
\newtheorem{problem}[theorem]{Problem}

\newtheorem{remark}[theorem]{Remark}

\newtheorem{hypothesis}[theorem]{Hypothesis}

\allowdisplaybreaks

\let\originalleft\left
\let\originalright\right
\renewcommand{\left}{\mathopen{}\mathclose\bgroup\originalleft}
\renewcommand{\right}{\aftergroup\egroup\originalright}

\DeclareMathOperator*{\esssup}{ess\,sup}

\renewcommand{\d}{\/\mathrm{d}\/}

\def\w{\textbf{W}^{\varepsilon}_{{\theta}^{\varepsilon}}}

\def\L{\mathbb{L}}
\def\A{\mathfrak{A}}
\def\B{\mathfrak{B}}
\def\B{\mathfrak{B}}
\def\U{\mathbb{U}}
\def\I{\mathrm{I}}
\def\F{\mathcal{F}}
\def\C{\mathrm{C}}
\def\f{\boldsymbol{f}}

\def\D{\mathrm{D}}

\def\Y{\mathbb{Y}}
\def\Z{\mathbb{Z}}
\def\E{\mathbb{E}}
\def\X{\mathbb{X}}
\def\x{\boldsymbol{x}}
\def\y{\boldsymbol{y}}
\def\z{\boldsymbol{z}}
\def\s{\boldsymbol{s}}

\def\g{\boldsymbol{g}}

\def\v{\boldsymbol{v}}
\def\u{\boldsymbol{u}}
\def\V{\mathbb{v}}
\def\w{\boldsymbol{w}}

\def\N{\mathbb{N}}

\def\V{\mathbb{V}}
\def\wi{\widetilde}

\def\H{\mathbb{H}}
\def\n{\boldsymbol{n}}

\newcommand{\R}{\mathbb{R}}

\renewcommand{\d}{\/\mathrm{d}\/}

\newcommand{\Addresses}{{
		\footnote{
			
			\noindent \textsuperscript{1,2,3}Department of Mathematics, Indian Institute of Technology Roorkee-IIT Roorkee,
			Haridwar Highway, Roorkee, Uttarakhand 247667, INDIA.\par\nopagebreak
			\noindent  \textit{e-mail:} \texttt{Manil T. Mohan: maniltmohan@ma.iitr.ac.in, maniltmohan@gmail.com.}
			
			\textit{e-mail:} \texttt{Jyoti Jindal: jyoti@ma.iitr.ac.in.}
			
			\textit{e-mail:} \texttt{Sagar Gautam: sagar\_g@ma.iitr.ac.in.}
			
			\noindent \textsuperscript{*}Corresponding author.
			
			\textit{Key words:} Navier-Stokes equations, convective Brinkman-Forchheimer equations, stationary and non-stationary, psuedomonotone operators, hemivariational inequality, Rothe method.
			
			Mathematics Subject Classification (2020): Primary 47J20, 49J52, 76M30; Secondary 35Q35, 76D03.
			
}}}

\begin{document}
	
	\title[Hemivariational inequality for 2D and 3D CBF equations]{Well-posedness of a boundary hemivariational inequality for stationary and non-stationary 2D and 3D convective Brinkman-Forchheimer equations
		\Addresses}
	\author[J. Jindal, S. Gautam and M. T. Mohan]
	{Jyoti Jindal\textsuperscript{1}, Sagar Gautam\textsuperscript{2}, and Manil T. Mohan\textsuperscript{3*}}
	
	\maketitle
\begin{abstract}
This paper investigates boundary hemivariational inequality problems associated with both stationary and non-stationary two and three-dimensional convective Brinkman-Forchheimer equations (or Navier-stokes equations with damping), which model the flow of viscous incompressible fluids through saturated porous media. The governing equations are nonlinear in both velocity and pressure and are subject to nonstandard boundary conditions. Specifically, we impose the no-slip condition along with a Clarke subdifferential relation between pressure and the normal velocity components. For the stationary case, we establish the existence and uniqueness of weak solutions using a surjectivity theorem for pseudomonotone operators. The existence of  weak solutions to the non-stationary hemivariational inequality is established via a limiting process applied to a temporally semi-discrete scheme, where the time derivative is approximated using the backward Euler method-commonly referred to as the Rothe method. It is demonstrated that the discrete problem admits solutions, which possess a weakly convergent subsequence as the time step tends to zero, and that any such weak limit satisfies the original hemivariational inequality.  A novel outcome of this paper  is that the existence results obtained in this work is applicable to 3D non-stationary Navier-Stokes equations also.  Moreover, under appropriate conditions on the absorption exponent, we show that Leray-Hopf weak solutions satisfies the energy equality,  the solution is shown to be unique and to depend continuously on the given data.
	\end{abstract}
	
\section{Introduction}\label{sec1}\setcounter{equation}{0}

Variational inequalities and hemivariational inequalities represent two significant classes of nonlinear problems that arise naturally in the modeling of equilibrium and optimization phenomena. Variational inequalities typically describe systems governed by convex energy functionals and monotone operators, and are widely used in contact mechanics (see \cite{VB}), fluid flow \cite{JTONK}, and economic equilibrium models \cite{ZNMN}. Hemivariational inequalities extend this framework to nonconvex and possibly nonsmooth settings, often involving nonmonotone, multivalued, or nonsmooth energy functionals, making them suitable for describing systems with friction, and material nonlinearity. Together, these frameworks provide powerful tools for analyzing a broad range of problems in mechanics, physics, engineering, and economics. One of the foundational developments in the theory of variational inequalities was Fichera’s seminal work \cite{GFi} on the frictionless contact problem between a linearly elastic body and a rigid foundation, a problem originally formulated by Signorini \cite{ASi}. This contribution provided a rigorous mathematical framework for describing contact conditions as inequalities rather than equalities, which was a key step towards the modern theory. Recent developments in both the mathematical theory and engineering applications-especially in contact mechanics have expanded the study of variational-hemivariational inequalities, with significant findings documented in \cite{SMAOMS}.

Interest in hemivariational inequalities, much like in the case of variational inequalities, initially arose from mechanical problems. The problems under consideration are governed by the hemivariational inequalities with Navier-Stokes equations (NSE) which describes the steady state flow of incompressible fluid. The first study of a stationary Navier-Stokes hemivariational inequality (NSHVI) in two and three-dimensions was carried out in \cite{SMAO}, where the existence of solutions was established using a surjectivity theorem for a pseudomonotone and coercive operators. Also, the authors in \cite{MLWH} investigates the well-posedness of a hemivariational inequality associated with the stationary NSE by employing techniques of convex minimization together with the Banach fixed point theorem. This approach was first introduced in \cite{Han2020}, where minimization principles were applied to certain classes of hemivariational inequalities (HVIs) to establish existence results. Subsequently, an evolutionary hemivariational inequality was investigated in \cite{MO2007}, where existence results were obtained via the Galerkin method applied to regularized problems, with solutions constructed as limits of sequences of solutions to these approximations. More recently, in \cite{Fang2016}, the existence of solutions to a NSHVI was proved using a temporally semi-discrete approximation scheme, known as the \emph{Rothe method}, where approximate solutions converge to a solution of the original problem. In particular, we should note that Rothe method has been applied to parabolic hemivariational inequalities in \cite{BZMS}, to parabolic variational-hemivariational inequalities in \cite{Bar2015}. For more literature in variational and hemivariational inequalities, we refer to \cite{KBCM, MGNK, ZPYZF, WWTZ}, and references therein.


\subsection{The Model} 
 In the mathematical modelling of fluid flows, damping effects play a vital role. As they reflect physical phenomena such as drag forces, viscous dissipation, and other mechanisms responsible for energy loss (see \cite{Hajduk2017, SGMTM}). These effects are typically incorporated into the Stokes and Navier-Stokes equations through additional damping terms. In particular, the modification of the classical NSE by the damping term gives rise to the \emph{convective Brinkman-Forchheimer (CBF) equations} (\cite{Hajduk2017}). Let us first provide the mathematical formulation of CBF equations. 
 
 Let $\mathcal{O} \subset \mathbb{R}^d$, $d \in \{2,3\}$, be a bounded domain with sufficiently smooth boundary $\Gamma$. The convective Brinkman-Forchheimer (CBF) system, which models the motion of an incompressible fluid through a saturated porous medium, is formulated as follows:
 \begin{equation}\label{eqn-cbf}
 	\left\{
 	\begin{aligned}
 		\frac{\partial \y}{\partial t}-\mu \Delta\y+(\y\cdot\nabla)\y+\alpha\y+\beta|\y|^{r-1}\y+\nabla p&=\boldsymbol{f}, \ \text{ in } \ \mathcal{O}\times(0,T), \\ \nabla\cdot\y&=0, \ \text{ in } \ \mathcal{O}\times[0,T),\\
 		\y&=0, \ \text{ in } \ \Gamma \times[0,T],\\
 		\y(0)&= \y_{0}\ \text{ in } \ \mathcal{O},
 	\end{aligned}
 	\right.
 \end{equation}
  where $\y(\cdot, \cdot):\mathcal{O}\times(0,T)\to\R^d$ is the velocity flow field,  $p(\cdot, \cdot):\mathcal{O}\times(0,T)\to\R$ denotes the pressure, and $\f (\cdot, \cdot):\mathcal{O}\times(0,T)\to\R^d$, the density of external forces. The positive constant $\mu$ represents the Brinkman coefficient (kinematic viscosity), while the positive constants $\alpha$ and $\beta$ corresponds to the Darcy coefficient (permeability of porous medium) and the Forchheimer coefficient (proportional to the porosity of the material), respectively. We refer to $r \in [1,\infty)$ as the \emph{absorption exponent}, with $r=3$ being known as the \emph{critical exponent}. Furthermore, we refer to the case $r < 3$ as subcritical, while $r > 3$ is termed supercritical (corresponding to fast-growing nonlinearities). For results concerning the Brinkman-Forchheimer equations with rapidly growing nonlinearities, we refer the reader to \cite{VKKS}. Notably, by setting $\alpha = \beta = 0$ in \eqref{eqn-cbf}, the system \eqref{eqn-cbf} reduces to the classical Navier-Stokes equations. The second equation in \eqref{eqn-cbf} enforces the incompressibility constraint.

In recent years, considerable attention has been devoted to the study of the global solvability of the system \eqref{eqn-cbf}. In \cite{ZZXW}, the authors established the existence of global strong solutions for any $r>3$, and further showed that such solutions are unique when $3<r\leq5$ in the three-dimensional setting. Later, in \cite{YZ}, it was demonstrated that strong solutions exist globally for all $r\geq3$, together with two regularity criteria for the range $1 \leq r < 3$. Moreover, for arbitrary $r \geq 1$, they proved uniqueness of the strong solution even within the class of weak solutions. More recently, in \cite{SGMTM}, the existence and uniqueness of weak solutions for the critical and supercritical  CBF equations (i.e., NSE with damping) were obtained in bounded or periodic domains, where the system includes an additional highly nonlinear absorption term. Their proof relied on the Faedo-Galerkin approximation combined with the Minty-Browder technique to ensure global solvability. Subsequently, in \cite{DsSz}, the authors investigated the regularity of the 3D Navier-Stokes-Brinkman-Forchheimer system in a bounded domain with Dirichlet boundary conditions and non-autonomous external forces. Furthermore, in \cite{KKM}, the asymptotic behavior of the CBF equations in both two and three dimensions was studied in unbounded domains. In another direction, \cite{PAES} considered a modified model in three dimensions, where the Darcy term $\alpha \y$ in \eqref{eqn-cbf} is replaced by $\beta_1|\y|^{r_1-1}\y$, with $\beta_1 < 0$ modeling pumping, in addition to the damping term. They proved the existence of weak solutions for $\beta > 0$ and $\beta_1 \in \R$ under the condition $r > r_1 \geq 1$, while continuous dependence on the data and existence of strong solutions were established for $r > 3$.

\subsection{Comparison with other works on hemivariational inequalities} Most studies on HVIs establish the existence of solutions using either the fixed-point method or the Galerkin approach. An alternative approach is based on applying a surjectivity result for pseudomonotone and coercive operators, which has been used in \cite{SMAO} to study hemivariational inequalities for the stationary NSE. In the present work, we extend these methodologies to the stationary CBF system (damped NSE) with hemivariational inequalities, thereby broadening the theoretical framework for such models.

On 2D bounded domains,  the authors in \cite{Qiu2019} investigated the solvability of the Stokes variational inequality with damping for viscous incompressible fluids, addressing both the analytical framework and numerical aspects. More recently, with a restriction of $1\leq r\leq 5$ in 3D,  the well-posedness of the Stokes hemivariational inequality for incompressible fluid flows with damping is examined  in \cite{MHHQLM} by reformulating the problem within a minimization framework. Furthermore,  with a restriction of $1\leq r\leq 5$ in 3D, W. Wang et al. \cite{WWXCWH} established a well-posedness result for the hemivariational inequality associated with the stationary Navier-Stokes equations involving a nonlinear damping term, where the analysis focused on nonsmooth slip boundary conditions of friction type. Despite these contributions, the theoretical study of CBF equations coupled with hemivariational inequalities for viscous incompressible fluids remains largely unexplored. To bridge this gap, the present work introduces and investigates Navier-Stokes hemivariational inequalities (NSHVIs) with damping effects, thereby providing new insights into the mathematical analysis of such fluid flow models.

On the other hand, the other approach so-called \emph{Rothe method} (known also as \emph{time approximation method}) provides the existence of solutions. It allows to extend any numerical method that is used to solve the stationary, elliptic inclusions with the multivalued term given as the Clarke subdifferential to time-dependent parabolic problems. The key idea is the replacement of time derivative with the backward difference scheme and at each time step, solving the corresponding elliptic equation to determine the solution at the next point on the time mesh. It is proved that the results obtained by such approach approximates the solution of the original problem. In contrast to Galerkin approximation method, this method does not require any smoothing or other additional regularizing terms in the inclusion. In \cite{Fang2016, PKa}, the authors study the existence of solutions of the hemivariational inequalities for the NSE, through a unified framework of an abstract problem. Also, explores the uniqueness, and continuous dependence on the data, and apply the results to the nonstationary hemivariational inequalities for the NSE that are of the boundary type, corresponding to nonlinear slip boundary conditions, and of the domain type, corresponding to hydraulic flow controls. Our results broaden the existing theory for the time-dependent hemivariational inequalities with CBF(or damped NSE) equations that are of the boundary type.

\subsection{Novelties, difficulties and approaches} 
This paper investigates the well-posedness of boundary hemivariational inequalities associated with the stationary and non-stationary 2D and 3D CBF equations for absorption exponents $1\leq r<\infty$. 
\begin{itemize}
	\item In the stationary case, we establish the existence of solutions for all values of  $r$ in this range by proving the pseudomonotonicity and coercivity of the linear and nonlinear operators involved in the CBF equations, and applying an abstract result from \cite[Theorem 1.3.70]{ZdSmP}. 
\end{itemize}
In our analysis, we deliberately avoid relying on the Sobolev embedding $\V\hookrightarrow\L^{p}$ for $1\leq p\leq\frac{2d}{d-2}.$ Instead, we explicitly work within the function space $\V\cap\L^p$ which we adopt as our solution space throughout the paper. This choice allows us to handle the nonlinearities more directly and ensures greater flexibility in dealing with the absorption exponent $r$.

In the non-stationary setting, the methodologies adopted for the cases $d\in\{2,3\}$ with $r\in[1,3]$ and those with  $r\in(3,\infty)$ differ substantially. 
\begin{itemize}
	\item In the supercritical regime, we successfully establish a pseudomonotonicity result for the Nemytskii operator  corresponding to the linear and nonlinear operators involved in the formulation (see Proposition \ref{prop-pseudo-time}). This result plays a central role in proving the existence of weak solutions.
	\item  In contrast, for the subcritical and critical cases, establishing pseudomonotonicity is significantly more challenging. Instead, we leverage the uniform boundedness of the Rothe scheme in the space $\mathcal{M}^{2,2}(0,T;\V;\H^2(\mathcal{O})^{\prime})$ to obtain the desired existence results. 
	 Notably, \emph{this approach allows us to prove the existence of weak solutions for the 3D NSE also, which, to the best of our knowledge, is explored for the first time in the literature. }
	 \item Furthermore, with the exception of the case $d\in\{2,3\}$ with $r\in[1,3)$, we show that all weak solutions satisfy the energy equality. 
\end{itemize}
Regarding the dominance of terms, for  $d\in\{2,3\}$ with $r\in[1,\infty)$, the diffusion term dominates the convection term. For $r\in(3,\infty)$, both the diffusion and damping terms dominate the convection term. This structural advantage enables us to establish uniqueness and continuous dependence on the data for solutions in these cases.

A key contribution of this work is the establishment of norm equivalence between $\|\y\|_{\H^1}$ and $\|\mathrm{curl \ }\y\|_{\H}$ on $\V$, as presented in Lemma \ref{lem-equiv}. To the best of our knowledge, such equivalence results within the context of boundary hemivariational problems involving the curl operator have not been systematically addressed in the literature. Building on ideas from \cite{Te}, we initiate a new line of investigation by formulating an analytical framework for boundary HVIs governed by vector-valued equations with curl operators. Within this framework, we rigorously demonstrate essential properties of the associated operators, such as boundedness and coercivity. This work thus provides a substantial generalization of the existing theory and introduces new tools for analyzing problems with boundary conditions.

\subsection{Organization of the paper}
The remainder of this paper is organized as follows. In the next section, we introduce the necessary mathematical foundations and establish the functional framework. Since the functional framework used for the CBF equations in this study differs from that employed for the NSE (as detailed in \cite{Te}), we provide an in-depth discussion of the function spaces involved (refer to Lemma \ref{lem-equiv} and Remark \ref{rem-equiv}). We then proceed with a thorough analysis of the linear, bilinear, and nonlinear operators that appear in our formulation. Subsection \ref{sub-sec-coe} is dedicated to exploring key properties of these operators, including monotonicity (see Lemma \ref{lem-mon} and Remark \ref{rem-mon}), coercivity, and pseudomonotonicity (see Lemma \ref{lem-pseudo}).

Section \ref{sec3} focuses on proving the existence and uniqueness of a weak solution to the boundary hemivariational inequality associated with the stationary 2D and 3D CBF equations. Using an abstract surjectivity result for operators that are both pseudomonotone and coercive (see \cite[Section 32.4]{EZ} and \cite[Theorem 2.6]{TRo}), we first establish the existence of a solution to an abstract hemivariational inequality, as shown in Theorem \ref{thm-station-exis} (also see Lemma \ref{lem-pseudo-mon}). The corresponding uniqueness results are presented in Theorem \ref{thm-station-unique} and further discussed in Remark \ref{rem-station-unique}. A concrete example of a boundary hemivariational inequality problem is provided in Subsection \ref{sub-boundary}, where, under suitable hypotheses, the main result is stated in Theorem \ref{thm-main-boundary}.

Section \ref{sec4} addresses the boundary hemivariational inequality for the non-stationary 2D and 3D CBF equations. After presenting some preliminaries on Bochner spaces, we introduce an abstract hemivariational inequality problem in Subsection \ref{sub-abstract}. For the supercritical absorption exponent ($r > 3$), we establish an important result on the pseudomonotonicity of the Nemytskii operator in Proposition \ref{prop-pseudo-time}. The existence of weak solutions to the non-stationary hemivariational inequality is then proved in Theorem \ref{thm-main-non-station} using a temporal discretization method known as the Rothe scheme. The existence of solutions to the Rothe scheme is analyzed in Theorem \ref{thm-rothe-existence}, and uniform energy estimates are derived in Lemmas \ref{lem-uniform-ener} and \ref{lem-uniform-ener-1}, which are crucial for the proof of Theorem \ref{thm-main-non-station}. For the cases $d = 2$ and $r \in [1, \infty)$, as well as $d = 3$ and $r \in [3, \infty)$, we demonstrate that the solution obtained in Theorem \ref{thm-rothe-existence} satisfies the energy equality (Lemma \ref{lem-ener-eq}), utilizing a result analogous to the Lions-Magenes lemma. By employing this energy equality, we establish the uniqueness and continuous dependence on the data in Theorem \ref{thm-unique-1}. For the critical case $d = r = 3$, by using a comparison technique, uniqueness results are obtained in Theorem \ref{thm-unique-2}. Subsection \ref{sub-non-boundary} considers an example of a boundary hemivariational inequality, and under appropriate assumptions, the main result is presented in Theorem \ref{thm-main-1}.

\section{Mathematical Formulation}\label{sec2}\setcounter{equation}{0}

\subsection{Preliminaries}

In this work, all function spaces are defined over the field of real numbers. Assume $\mathbb{E}$ be a normed space, with norm denoted by $\|\cdot\|_{\mathbb{E}}.$ The topological dual is denoted by $\mathbb{E}^{\prime}$, and the duality pairing between $\mathbb{E}^{\prime}$ and $\E$ is written as ${}_{\E^{\prime}}\langle\cdot,\cdot\rangle_{\E}$. We use $\E_w$ to represent the  space $\E$ equipped with the weak topology. The set $2^{\E^{\prime}}$ refers to the set of all subsets of $\E^{\prime}$. Unless otherwise stated, $\mathbb{E}$ will be assumed to be a Banach space.

We begin by recalling the definition of a locally Lipschitz function.
\begin{definition}[{\cite[Definition 3.20]{SMAOMS}}]
	A function $\psi:\E\to\mathbb{R}$ is called locally Lipschitz, if for every $\x\in\E$, there exists a neighborhood $N$ of $\x$ and a constant $L_N$ such that $$|\psi(\u)-\psi(\v)|\leq L_N\|\u-\v\|_{\E}\ \text{ for all }\ \u,\v\in N.$$ 
\end{definition}

Next, we review the definitions of the generalized directional derivative and the generalized gradient in the sense of Clarke for a locally Lipschitz function.
\begin{definition}[{\cite[Definition  5.6.3]{ZdSm1}}]
	Let $\psi :\E \to\R$ be a locally Lipschitz function. The generalized directional derivative of $\psi$ at a point $\x\in\E$ in the direction $\v\in\E$, denoted by $\psi^0(\x;\v)$, is defined by
	\begin{align*}
		\psi^0(\x;\v)=\lim_{\u\to\x}\sup_{\lambda\downarrow 0}\frac{\psi(\u+\lambda\v)-\psi(\u)}{\lambda}. 
	\end{align*}
	The generalized gradient or Clarke subdifferential of $\psi$ at $\x$, denoted by $\partial \psi(\x)$, is the subset of the dual space $\E^{\prime}$ given by
	\begin{align*}
		\partial \psi(\x)=\left\{\zeta\in\E^{\prime}:\psi^0(\x;\v)\geq {}_{\E^{\prime}}\langle\zeta,\v\rangle_{\E}\ \text{ for all }\ \v\in\E \right\}. 
	\end{align*}
	A locally Lipschitz function $\psi$ is said to be \emph{regular (in the sense of Clarke)} at a point $\x \in \E$  if, for every direction $\v\in\E$, the one-sided directional derivative $\psi'(\x;\v)$  exists and satisfies $$\psi^0(\x;\v) = \psi'(\x;\v).$$
\end{definition}
The following results are used in the sequel. 
\begin{proposition}[{\cite[Proposition 2.1.2]{FHC}, \cite[Proposition 5.6.9]{ZdSm1}}]\label{prop-clarke}
If  $\psi:\E\to\R$ is locally Lipschitz, then 
\begin{enumerate}
	\item $\partial \psi(\x)$ is a nonempty, convex, weak$^*$-compact subset of $\E^{\prime}$ and satisfies $$\|\boldsymbol{\zeta}\|_{\E^{\prime}}\leq L_N \ \text{ for all }\  \boldsymbol{\zeta}\in\partial \psi(\x).$$
	\item We have 
	$$	\psi^0(\x;\v)=\max\left\{\langle\boldsymbol{\zeta},\v\rangle:\boldsymbol{\zeta}\in\partial \psi(\x)\right\} \text{ for all } \v\in\E.$$
\end{enumerate}
\end{proposition}

\begin{proposition}[{\cite[Proposition 5.6.10]{ZdSm1}}]\label{prop-ups}
	If $\psi:\E \to\R$ is locally Lipschitz, then the multi-function $\x\mapsto\partial \psi(\x)$  is upper semicontinuous (abreviated as u.s.c) from $\E$ into $\E^{\prime}$. 
\end{proposition}
We now turn to a review of the concept of pseudomonotonicity in the setting of single-valued operators.
\begin{definition}[{\cite[Definition 1]{SMAO}}]\label{def-pseudo}
	A single-valued operator $\mathcal{R} : \E\to\E^{\prime}$  is said to be \emph{pseudomonotone}, if
	\begin{enumerate}
		\item $\mathcal{R}$ is  bounded (means that it maps bounded subsets of $\E$ into bounded subsets of its dual space $\E^{\prime}$).
		\item  Moreover, if $\y_n\xrightarrow{w}\y$ in $\E$ and $\limsup\limits_{n\to\infty} {}_{\E^{\prime}}\langle\mathcal{R}(\y_n),\y_n-\y\rangle_{\E}\leq 0,$ then it follows that
		\begin{align*}
			{}_{\E^{\prime}}\langle\mathcal{R}(\y),\y-\z\rangle_{\E}\leq\liminf_{n\to\infty}{}_{\E^{\prime}}\langle\mathcal{R}(\y_n),\y_n-\z\rangle_{\E},\ \text{ for all }\ \z\in\E.
			\end{align*}
	\end{enumerate}
\end{definition}
It is shown in \cite[Remark 2]{SMAO} that an operator $\mathcal{R} : \E\to\E^{\prime}$ is pseudomonotone if an only if it is bounded and whenever $\y_n\xrightarrow{w}\y$ in $\E$ together with $\limsup\limits_{n\to\infty} {}_{\E^{\prime}}\langle\mathcal{R}(\y_n),\y_n-\y\rangle_{\E}\leq 0,$ it implies that
\begin{align}\label{eqn-con-pseudo}
	\mathcal{R}(\y_n)\xrightarrow{w}\mathcal{R}(\y)\ \text{ in }\ \E^{\prime}\ \text{ and }\ \lim_{n\to\infty}{}_{\E^{\prime}}\langle\mathcal{R}(\y_n),\y_n-\y\rangle_{\E}=0. 
\end{align}
The subsequent definition can be found, for example, in  \cite[Definition 1]{FBPH} or \cite[Definition 3.57]{SMAOMS}. 
\begin{definition}
	Consider $\E$ be a reflexive Banach space. A multi-valued operator $\mathcal{R} : \E\to 2^{\E^{\prime}}$ is said to be \emph{pseudomonotone} if the following conditions are satisfied:
	\begin{enumerate}
		\item for each $\y \in \E,$ the set $\mathcal{R}(\y)$ is nonempty, bounded, closed and convex;
		\item the mapping  $\mathcal{R}$  is upper semicontinuous from each finite dimensional subspace of $\E$ into $\E^{\prime}_w$;
		\item if $\{\y_n\}_{n\geq1}\subset\E$ and $\{\y_n^*\}_{n\geq1}\subset\E^{\prime}$ satisfy $$\y_n\xrightarrow{w}\y \text{ in } \E, \y_n^*\in\mathcal{R}(\y_n)\text{ and } \limsup\limits_{n\to\infty}{}_{\E^{\prime}}\langle\y_n^*,\y_n-\y\rangle_{\E}\leq 0,$$ then for every $\z\in\E$, there exists $\y^*(\z) \in\mathcal{R}(\y)$  such that
			\begin{align*}
			{}_{\E^{\prime}}\langle\y^*(\z),\y-\z\rangle_{\E}\leq\liminf_{n\to\infty}{}_{\E^{\prime}}\langle\y_n^*(\z),\y_n-\z\rangle_{\E}.
		\end{align*}
	\end{enumerate}
\end{definition}
The following widely used  result gives a sufficient condition to verify the pseudomonotonicity of an operator. 
\begin{proposition}[{\cite[Proposition 1.3.66]{ZdSmP}}]\label{prop-suff-pseudo}
	Let $\E$ be a real reflexive Banach space, and assume that $\mathcal{R} : \E \to 2^{\E^{\prime}}$  satisfies the following conditions:
	\begin{enumerate}
		\item $\mathcal{R}$ is nonempty, convex and closed valued;
		\item $\mathcal{R}$  is bounded;
		\item if $\y_n\xrightarrow{w}\y$ in $\E$, $\y_n^*\xrightarrow{w}\y^*$ in $\E^{\prime}$ with $\y_n^*\in\mathcal{R}(\y_n)$, and $\limsup\limits_{n\to\infty}{}_{\E^{\prime}}\langle\y_n^*,\y_n-\y\rangle_{\E}\leq 0$, then $\y^*\in\mathcal{R}(\y)$ and ${}_{\E^{\prime}}\langle\y_n^*,\y_n\rangle_{\E}\to {}_{\E^{\prime}}\langle\y^*,\y\rangle_{\E}$.
	\end{enumerate}
	Then, the operator $\mathcal{R}$ is pseudomonotone.
\end{proposition}

\begin{proposition}[{\cite[Proposition 1.3.68]{ZdSmP}}]
	If $\E$ is a reflexive Banach space and $\mathcal{R}_1,\mathcal{R}_2:\E \to 2^{\E^{\prime}}$ are pseudomonotone operators, then $\mathcal{R}_1+\mathcal{R}_2$ is also pseudomonotone. 
\end{proposition}

We now introduce the notion of coercivity, which will be used later.
\begin{definition}
	An operator $\mathcal{R} : \E \to 2^{\E^{\prime}}$ is said to be coercive if either
	\begin{enumerate}
\item  its domain $\D(\mathcal{R})$ is bounded, or \item $\D(\mathcal{R})$ is unbounded and
	\begin{align*}
		\lim\limits_{\|\y\|_{\E}\to\infty,\ \y\in\D(\mathcal{R})}\frac{\inf\left\{ {}_{\E^{\prime}}\langle\y^*,\y\rangle_{\E} :\y^*\in\mathcal{R}(\y)\right\}}{\|\y\|_{\E}}=\infty. 
		\end{align*}
	\end{enumerate}
\end{definition}
We now present the main surjectivity result for operators that are both pseudomonotone and coercive (\cite[Section 32.4]{EZ}, \cite[Theorem 2.6]{TRo}).
\begin{theorem}[{\cite[Theorem 1.3.70]{ZdSmP}}]\label{thm-surjective}
	Consider $\E$ be a reflexive Banach space, and $\mathcal{R}: \E \to 2^{\E^{\prime}}$ be pseudomonotone and coercive. Then $\mathcal{R}$ is surjective, that is, its range satisfies $R(\mathcal{R}) = \E^{\prime}$.
\end{theorem}

The following  result will be used to establish the covergence of the nonlinear term. 
\begin{lemma}[{\cite[Lemma 1.3]{JLL}}]\label{Lem-Lions}
	Consider $\mathcal{O}_T \subset \R^d\times \R$ be a bounded open set, and let $\varphi_m$, $\varphi$ be functions in $\mathrm{L}^q(\mathcal{O}_T)$ with $1<q <\infty$ for $m\in\N$, such that
	\begin{align*}
		\|\varphi_m\|_{L^q(\mathcal{O}_T)} \leq C,\ \ \mbox{for all}\ \ m\in\N \; \; \mbox{and} \ \ \varphi_m \to \varphi\;\; \mbox{a.e. in}\ \ \mathcal{O}_T,\  \mbox{as}\ \ m \to \infty.
	\end{align*}
	Then, $\varphi_m \xrightarrow{w} \varphi$ in $\mathrm{L}^q(\mathcal{O}_T)$, as $m\to \infty$.
\end{lemma}

\subsection{Functional setting} Let us first describe the function spaces and operators needed for this paper.
A bounded domain $\mathcal{O}$ is said to be Lipschitz, if for any point $\x$ on the boundary $\Gamma$, there exists a system of orthogonal co-ordinates $(y_1,\ldots,y_d)$, a cube $U_{\x}$ containing $\x$, $U_{\x}=\prod_{i=1}^d(-a_i,a_i)$, and a Lipschitz-continuous mapping $\Phi_{\x}$ defined from $\prod_{i=1}^{d-1}(-a_i,a_i)\to (-a_d,a_d)$  such that 
\begin{align*}
	\mathcal{O}\cap U_{\x}&=\left\{(y_1,\ldots,y_d)\in U_{\x}:y_d>\Phi_{\x}(y_1,\ldots,y_{d-1})\right\},\\
	\Gamma\cap U_{\x}&=\left\{(y_1,\ldots,y_d)\in U_{\x}:y_d=\Phi_{\x}(y_1,\ldots,y_{d-1})\right\}.
\end{align*}
The domain $\mathcal{O}$ is said to be of class $\C^{m,1}$, for an integer $m\geq 1$, if the mappings $\Phi_{\x}$ can be chosen $m$-times differentiable with Lipschitz-continuous partial derivatives of order $m$ (\cite{RAAJJ,ABDG}).  Assuming that the boundary is of class $\C^{1,1}$,  the curvature tensor of the boundary will be denoted by $\mathcal{B}$.  We know that the function $\Phi_{\x}$ has a.e. second derivatives and that $\mathcal{B}$ coincides with the matrix  of these derivatives $\{\frac{\partial^2\Phi_{\x}}{\partial y_i\partial y_j}\}_{1\leq i,j\leq d-1}$ for appropriate co-ordinates. Let 
$\mathrm{Tr\ }\mathcal{B}$ denote the trace of this operator.

Let us first define the following vector-valued function spaces: 
\begin{align}
	\label{eqn-hcurl}
	\mathbb{H}(\mathrm{curl}) &:= \{ \y \in \L^2(\mathcal{O}) : \mathrm{curl\ } \y \in\L^2(\mathcal{O}) \},
	\\
	\label{eqn-hdiv}
	\mathbb{H}(\mathrm{div}) &:= \{ \y \in \L^2(\mathcal{O}) : \mathrm{div\ } \y \in\L^2(\mathcal{O}) \},
\end{align}
where both differential operators $\mathrm{curl\ }$ and $\mathrm{div\ }$ are understood in the weak sense.
The spaces $	\mathbb{H}(\mathrm{curl}) $ and $\mathbb{H}(\mathrm{div}) $  endowed with the following graph norms:
\begin{align}
	\label{eqn-hcurl-norm}
	\|\y\|_{\H(\mathrm{curl})}^2 &:= \|\y\|^2_{ \L^2(\mathcal{O})} + \|\mathrm{curl\ }\y\|^2_{\L^2(\mathcal{O}) },\ \text{ for all } \ \y \in \mathbb{H}(\mathrm{curl})\ \text{ and }
	\\
	\label{eqn-hdiv-norm}
	\|\y\|_{\H(\mathrm{div})}^2 &:= \|\y\|^2_{ \L^2(\mathcal{O})} + \|\mathrm{div\ }\y\|^2_{\L^2(\mathcal{O}) }, \ \text{ for all } \ \y \in \mathbb{H}(\mathrm{div}),
\end{align}
respectively, are  Hilbert spaces.
The following Green's formulae will be helpful in the sequel (\cite[Equations (2.17) and (2.22)]{VGPR}): 
\begin{enumerate}
	\item [(i)] For all $\z\in \H^1(\mathcal{O})$ and $\phi\in \mathrm{H}^1(\mathcal{O})$, we have 
	\begin{align}\label{eqn-green-1}
		(\z,\nabla\phi)+(\nabla\cdot\z,\phi)=\langle\z\cdot\n,\phi\rangle_{\mathrm{H}^{-1/2}(\Gamma)\times \mathrm{H}^{1/2}(\Gamma)}.	\end{align}
	\item [(ii)] For all $\z\in \H^1(\mathcal{O})$ and $\boldsymbol{\phi}\in\mathbb{H}^1(\mathcal{O})$, we infer 
	\begin{align}\label{eqn-green-2}
		(\mathrm{curl \ } \z,\boldsymbol{\phi})-(\z,\mathrm{curl \ } \boldsymbol{\phi})=\langle\z\times\n,\boldsymbol{\phi}\rangle_{\mathrm{H}^{-1/2}(\Gamma)\times \mathrm{H}^{1/2}(\Gamma)}. 
	\end{align} 
\end{enumerate}
Green's formulae \eqref{eqn-green-1} and \eqref{eqn-green-2} immediately yield $$\mathrm{curl \ } \z\cdot\n=0, \ \text{ for all }\ \z\in\V,$$ since for all $\phi\in \mathrm{H}^2(\mathcal{O})$, we have 
\begin{align}
	\langle\mathrm{curl \ }\z\cdot\n,\phi\rangle_{\mathrm{H}^{-1/2}(\Gamma)\times \mathrm{H}^{1/2}(\Gamma)}&=(\mathrm{curl \ }\z,\nabla\phi)+(\nabla\cdot\mathrm{curl \ }\z,\phi)=(\mathrm{curl \ }\z,\nabla\phi)\nonumber\\&= (\z,\mathrm{curl \ }\nabla\phi)+\langle\z\times\n,\nabla\phi\rangle_{\mathrm{H}^{-1/2}(\Gamma)\times \mathrm{H}^{1/2}(\Gamma)}=0,
\end{align}
where we have used the fact that $\nabla\cdot\mathrm{curl \ }\z=0$ and $\mathrm{curl \ }\nabla\phi=\boldsymbol{0}$, and the required result follows by using the density of $\mathrm{H}^2(\mathcal{O})$ in $\mathrm{H}^1(\mathcal{O})$. Moreover, for $\z\in\mathbb{H}^2(\mathcal{O})\cap\V$, by choosing $\boldsymbol{\phi}=\mathrm{curl\ }\z$ in \eqref{eqn-green-2}, we deduce 
\begin{align}\label{eqn-curl-curl}
\|\mathrm{curl\ }\z\|_{\L^2({\mathcal{O}})}^2
=(\z,\mathrm{curl\ }\mathrm{curl\ }\z),
\end{align}
since $\z\times\n=0$ on $\Gamma$. 
 
Let us first introduce the following function spaces:
\begin{align}
	\mathscr{M}=\left\{\z\in \C^{\infty}(\overline{\mathcal{O}};\mathbb{R}^d):\nabla\cdot\z=0\ \text{ in }\ \mathcal{O},\ \z_{\tau}=\boldsymbol{0}\ \text{ on }\ \Gamma\right\}. 
\end{align}
The condition $\z_{\tau}=\boldsymbol{0}$ on $\Gamma$ provides $\z=z_n\boldsymbol{n}$, where $z_n=\z\cdot\boldsymbol{n}$. Therefore, we immediately have 
\begin{align}
	 \z\times \boldsymbol{n}=\boldsymbol{0} \ \text{ on } \Gamma. 
\end{align}
  Let $\V$ and $\H$ be the closures of $\mathscr{M}$ with respect to the norms of $\mathrm{H}^1(\mathcal{O};\mathbb{R}^d)$ and $\mathrm{L}^2(\mathcal{O};\mathbb{R}^d)$, respectively. These spaces can be characterized as  
  \begin{align}
  	\mathbb{H}&=\left\{\z\in \mathrm{L}^2(\mathcal{O};\mathbb{R}^d):\nabla\cdot\z=0\ \text{ in }\ \mathcal{O},\ \z\times\n=\boldsymbol{0}\ \text{ on }\ \Gamma\right\},\\
  		\mathbb{V}&=\left\{\z\in \mathrm{H}^{1}(\mathcal{O};\mathbb{R}^d):\nabla\cdot\z=0\ \text{ in }\ \mathcal{O},\ \z\times\n=\boldsymbol{0}\ \text{ on }\ \Gamma\right\}.
  \end{align}
   For $p\in(2,\infty)$, we denote by ${\L}^p$, the closures of $\mathscr{M}$ with respect to the   $\mathrm{L}^p(\mathcal{O};\mathbb{R}^d)$ norm. The space ${\L}^p$ can also be  characterized in a similar way as above. Then, we have the following continuous and dense embedding:
\begin{align}
\V\cap\L^p\hookrightarrow	\V\hookrightarrow\H\equiv\H^{\prime}\hookrightarrow\V^{\prime}\hookrightarrow \V^{\prime}+\L^{\frac{p}{p-1}}.
\end{align}
Note that  the embedding mapping $\iota:\V\to\H$ is continuous and compact. We denote the norms in $\H$ and ${\L}^p, $ $p\in(2,\infty)$ by 
\begin{align*}
	\|\y\|_{\H}=\bigg(\int_{\mathcal{O}}|\y(x)|^2\d x\bigg)^{1/2} \ \text{ and }\ 	\|\y\|_{\L^{p}}=\bigg(\int_{\mathcal{O}}|\y(x)|^p\d x\bigg)^{1/p},
\end{align*}
respectively. The induced norm on $\V$ is the $\H^1$-norm, that is, $\|\y\|_{\H^1}=\big(\|\y\|_{\H}^2+\|\nabla\y\|_{\H}^2\big)^{1/2}$. Let us now show that the norm
\begin{align*}
	\|\y\|_{\V}=\bigg(\int_{\mathcal{O}}|\mathrm{curl \ }\y(x)|^2\d x\bigg)^{1/2}=\|\mathrm{curl \ }\y\|_{\H}
\end{align*}
is equivalent to the $\H^1$-norm on $\V$. In the sequel, the duality pairing between the spaces $\V$ and $\V^{\prime}$, $\L^{p}$ and $\L^{\frac{p}{p-1}}$, and $\V\cap\L^{p}$ and $\V^{\prime}+\L^{\frac{p}{p-1}}$, will be denoted by $\langle\cdot,\cdot\rangle$. From \cite[Subsection 2.1]{FKS} , we have that the sum space $\V^{\prime} + \L^{\frac{p}{p-1}}$ is a Banach space with respect to the norm
\begin{align}
	\| \y \|_{\V^{\prime} + \L^{\frac{p}{p-1}}} &:= \inf \left\{ \|  \y_1 \|_{\V^{\prime}} + \|  \y_2 \|_{\L^{\frac{p}{p-1}}} : \y = \y_1 + \y_2, \ \y_1 \in \V^{\prime}, \ \y_2 \in \L^{\frac{p}{p-1}} \right\} \notag \\
	&= \sup \left\{ \frac{|\langle \y_1 + \y_2, \f \rangle|}{\| \f \|_{\V \cap \L^p}} : 0 \neq \f \in \V \cap \L^p \right\},
\end{align}
where $\| \cdot \|_{\V \cap \L^p} := \max \{ \| \cdot \|_{\V}, \| \cdot \|_{\L^p} \}$ is a norm on the Banach space $\V \cap \L^p$. Also, the norm $\max \{ \| \y \|_{\V}, \| \y \|_{\L^p} \}$ is equivalent to the norms $\| \y \|_{\V} + \| \y \|_{\L^p}$ and $\sqrt{ \| \y \|_{\V}^2 + \| \y \|_{\L^p}^2 }$ on the space $\V \cap \L^p$.
\begin{lemma}\label{lem-equiv}
	There exist positive constants $C_1$ and $C_2$ such that 
	\begin{align}\label{eqn-equiv}
		 \|\y\|_{\H^1}\leq C_1\left(\|\y\|_{\H}+\|\mathrm{curl \ }\y\|_{\H}\right), \ \text{ for all }\ \y\in\V,
	\end{align}
	and 
		\begin{align}\label{eqn-equiv-1}
		\|\y\|_{\H}\leq C_2\|\mathrm{curl \ }\y\|_{\H}, \ \text{ for all }\ \y\in\V.
	\end{align}
\end{lemma}
\begin{proof}
	Let us first prove the inequality \eqref{eqn-equiv}. From \cite[Lemma 3.8]{VGPR} or \cite[Theorem 3.1.1.2]{GrP}, we infer for any function $\y\in\V$, 
	\begin{align}\label{eqn-curl-div}
		\|\y\|_{\H^1}^2=\|\mathrm{curl \ }\y\|_{\H}^2+\|\mathrm{div\ }\y\|_{\H}^2-\int_{\Gamma}(\mathrm{Tr}(\mathcal{B}))(\y(x)\cdot\n(x))^2\d S(x),
	\end{align}
where $\mathcal{B}$ is the curvature tensor of the boundary. By using \cite[Theorem 1.5.1.10]{GrP}, we estimate the final term in \eqref{eqn-curl-div} as 
	\begin{align*}
		\left|\int_{\Gamma}(\mathrm{Tr}(\mathcal{B}))(\y(x)\cdot\n(x))^2\d S(x)\right|&\leq C\int_{\Gamma}|\y(x)|^2\d S(x)\leq\frac{1}{2}\|\y\|_{\H^1}^2+C\|\y\|_{\H}^2.
	\end{align*}
Using the fact that $\mathrm{div\ }\y=0$ in the above inequality in  \eqref{eqn-curl-div}, we finally arrive at \eqref{eqn-equiv}. 

Let us now prove the  inequality in \eqref{eqn-equiv-1} by a contradiction argument. Suppose that  there exists a sequence $\{\y_m\}_{m\in\mathbb{N}}\in\V$  such that
\begin{align}\label{eqn-6.1.18}
	\|\y_m\|_{\H}>  m \|\mathrm{curl\ }\y_m\|_{\H}, \ \mbox{ for every }\ m\in\mathbb{N}.
	\end{align}
	Without loss of generality, we can assume that
	$\|\y_m\|_{\H}=1$. Using  the  inequality in \eqref{eqn-equiv} and \eqref{eqn-6.1.18}, we have
	\begin{align*}
		\|\y_m\|_{\mathbb{H}^1}\leq C_1\left(\|\y_m\|_{\H}+\|\mathrm{curl\ }\y_m\|_{\H}\right)\leq 2C_1,
	\end{align*}
for some positive constant $C_1$	so that the sequence $\{\y_m\}_{m\in\mathbb{N}}$ is bounded in $\V$, endowed with the norm induced by the space $\mathbb{H}^1$. By the Banach-Alaoglu theorem, we can extract a subsequence still denoted by  $\{\y_m\}_{m\in\mathbb{N}}$ which converges weakly in $\mathbb{V}$ to $\y\in\mathbb{V}$ as $m\to\infty$. Therefore, one can deduce
\begin{align*}
	\mathrm{curl \ }\y_m\xrightarrow{w}\mathrm{curl \ } \y\ \text{ in }\ \H\ \text{ as }\ m\to\infty. 
\end{align*}
Since the embedding of $\mathbb{V}\hookrightarrow \H$ is compact by the Rellich-Kondrachov theorem (\cite[Theorem 9.16]{Brezis2011}), there exists a  subsequence  denoted by $\{\y_m\}_{m\in\mathbb{N}}$ and an element $\y \in \H$ such that $\y_m \to \y$   strongly in $\H$ and using the continuity of the norm function, we get $\|\y\|_{\H}=1$.
On the other hand, it follows  from \eqref{eqn-6.1.18} that $\mathrm{curl \ }\y=0$  in the weak $\mathrm{L}^2$-sense, since by using the weakly lower semicontinuity property of the $\H$-norm, we have 
\begin{align*}
	\|	\mathrm{curl \ }\y\|_{\H}\leq\liminf_{m\to\infty}\|	\mathrm{curl \ }\y_m\|_{\H}<\liminf_{m\to \infty}\frac{1}{m}\|\y_m\|_{\H}=0. 
\end{align*}
 Thus we infer that $\y\in\ker(\mathrm{curl})\cap\H,$ where 
 \begin{align*}
 	\ker(\mathrm{curl})&:= \{ \y \in \mathbb{H}(\mathrm{curl}): \mathrm{curl \ }\y=\boldsymbol{0} \},
 \end{align*}
 If one shows that  $\ker(\mathrm{curl})\cap\H=\{\boldsymbol{0}\}$, then we infer that  $\y=\boldsymbol{0},$ which contradicts the previously proven assertion  that $\|\y\|_{\H}=1$. Therefore, one can conclude the proof of \eqref{eqn-equiv-1}.

It is now left to show that $\ker(\mathrm{curl})\cap\H=\{\boldsymbol{0}\}$. If $\y\in\ker(\mathrm{curl})\cap\H$, then $\mathrm{curl\ \y}=\boldsymbol{0}$, $\mathrm{div\ }\y=0$ and $\y\times\n=\boldsymbol{0}$. Using \cite[Theorem 3.4]{VGPR}, the condition $\mathrm{div\ }\y=0$  implies the  existence of potential $\boldsymbol{\phi}\in\mathbb{H}^1(\mathcal{O})$ such that $\y=\mathrm{curl\ }\boldsymbol{\phi}$ and $\mathrm{div\ }\boldsymbol{\phi}=0$. Since $\mathrm{curl\ }\boldsymbol{\phi}\times\n=\boldsymbol{0}$ and $\mathrm{curl \ }\mathrm{curl \ }\boldsymbol{\phi}\in\mathbb{L}^2(\mathcal{O})$, then by replacing $\z$ in \eqref{eqn-green-2}  with $\mathrm{curl\ }\boldsymbol{\phi}$, we find 
\begin{align*}
\|\y\|_{\H}^2=	\|\mathrm{curl\ }\boldsymbol{\phi}\|_{\H}^2=(\boldsymbol{\phi},\mathrm{curl \ }\mathrm{curl \ }\boldsymbol{\phi})=(\boldsymbol{\phi},\mathrm{curl \ }\y)=0,
\end{align*}
so that $\y=\boldsymbol{0}$ and hence $\ker(\mathrm{curl})\cap\H=\{\boldsymbol{0}\}$.
%
%
\end{proof}

\begin{remark}\label{rem-equiv}
	As $\|\y\|_{\H^1}=\big(\|\y\|_{\H}^2+\|\nabla\y\|_{\H}^2\big)^{1/2}$, it is immediate that 	there exists a constant $C_3>0$ such that 
	\begin{align*}C_3\|\mathrm{curl \ }\y\|_{\H}\leq \|\y\|_{\H^1}, \ \text{ for all }\ \y\in\V. 
	\end{align*}
	Using the relations \eqref{eqn-equiv} and \eqref{eqn-equiv-1}, we obtain the existence of a constant $C_4>0$ such that 
		\begin{align*} \|\y\|_{\H^1}\leq C_4\|\mathrm{curl \ }\y\|_{\H}, \ \text{ for all }\ \y\in\V. 	\end{align*}
		Therefore, the norms $\|\y\|_{\H^1}$ and $\|\mathrm{curl \ }\y\|_{\H}$ are equivalent on $\V$. 
\end{remark}

\subsection{Linear operator}
Let us define a bilinear form $a : \V \times \V\to\R$ by $$a(\y, \z) := (\mathrm{curl\ } \y, \mathrm{curl\ }\z),\  \text{ for }\ \y, \z \in\V.$$ From the definition of $a(\cdot,\cdot)$, it is clear that $a(\cdot,\cdot)$ is $\V$-continuous, that is, 
   \begin{align}\label{eqn-abound}
	|a(\y,\z)|\leq \|\y\|_{\V}\|\z\|_{\V},\ \text{ for all }\ \y,\z\in\V.
	\end{align} 
	Hence, by \emph{the Riesz representation theorem}, there exists a unique linear operator $\A:\V\to \V^{\prime}$, where $\V^{\prime}$ is the dual of $\V$, such that
	\begin{align}\label{eqn-def-A}a(\y, \z) = \langle\A\y, \z\rangle, \text{ for all }\y, \z \in\V, \end{align} where $\langle\cdot,\cdot\rangle$ denotes the duality pairing between $\V^{\prime}$ and $\V$.  Moreover, the form $a(\cdot,\cdot)$ is $\V$-coercive, that is, it satisfies \begin{align}\label{eqn-coercive}a(\y, \y) =\|\y\|_{\V}^2=\|\mathrm{curl\ } \y\|_{\H}^2, \text{ for all }\ \y \in\V.\end{align}Therefore, by means of the \emph{Lax-Milgram theorem} (see \eqref{eqn-abound} and \eqref{eqn-coercive}), the operator $\A : \V\to\V^{\prime}$ is an isomorphism.

\subsection{Bilinear operator}
Let us define the \emph{trilinear form} $b(\cdot,\cdot,\cdot):\V\times\V\times\V\to\R$ by $$b(\y,\z,\w)=\int_{\mathcal{O}}(\mathrm{curl\ }\y(x)\times\z(x))\cdot\w(x)\d x.$$
It can be easily seen that for all $\y,\z,\w\in\V$
\begin{align}\label{eqn-bbound}
	|b(\y,\z,\w)|\leq\|\mathrm{curl\ }\y\|_{\H}\|\z\|_{\L^4}\|\w\|_{\L^4}\leq  C\|\y\|_{\V}\|\z\|_{\H^1}\|\w\|_{\H^1}\leq C\|\y\|_{\V}\|\z\|_{\V}\|\w\|_{\V},
\end{align}
where we have used Sobolev's embedding also. Therefore, the map $b(\y, \z, \cdot) $ is linear, continuous on $\V$, the corresponding element of $\V^{\prime}$ is denoted by $\B(\y, \z),$ so that $$b(\y,\z,\w)=\langle\B(\y,\z),\w\rangle, \ \text{ for all }\ \y,\z,\w\in\V.$$ Note that $\B(\cdot,\cdot)$ is a bilinear operator. We also denote  $\B(\y) = \B(\y, \y)$. By fixing $\u=\mathrm{curl\ }\y$, we find
\begin{align}\label{eqn-b-est}
	b(\y,\z,\w)&=\int_{\mathcal{O}}(\u(x)\times\z(x))\cdot\w(x)\d x=\int_{\mathcal{O}}\varepsilon_{ijk}u_j(x)z_k(x)w_i(x)\d x \nonumber\\&=-\int_{\mathcal{O}}\varepsilon_{kji}u_j(x)w_i(x)z_k(x)\d x=-\int_{\mathcal{O}}(\u(x)\times\w(x))\cdot\z(x)\d x\nonumber\\&=-b(\y,\w,\z),
\end{align}
where $\varepsilon_{ijk}$ is the Levi-Civita symbol. Taking $\z=\w$ in \eqref{eqn-b-est}, we immediately have 
\begin{align}\label{eqn-b0}
	b(\y,\z,\z)=\langle\B(\y,\z),\z\rangle=0,\ \text{ for all }\ \y,\z\in\V. 
\end{align}
For $r>3$ and $\y,\z\in\V\cap\L^{r+1}$, by using H\"older's and interpolation inequalities, we estimate $|\langle\B(\y,\y),\z\rangle|$ as 
\begin{align}\label{eqn-b-est-lr}
	|\langle\B(\y,\y),\z\rangle|&\leq\|\y\|_{\V}\|\y\|_{\L^{\frac{2(r+1)}{r-1}}}\|\z\|_{\L^{r+1}}\leq\|\y\|_{\V}\|\y\|_{\H}^{\frac{r-3}{r-1}}\|\y\|_{\L^{r+1}}^{\frac{2}{r-1}}\|\z\|_{\L^{r+1}}.
\end{align}
Thus, we deduce that
\begin{align}
	\|\B(\y)\|_{\V^{\prime}+\L^{\frac{r+1}{r}}}\leq \|\y\|_{\V}\|\y\|_{\H}^{\frac{r-3}{r-1}}\|\y\|_{\L^{r+1}}^{\frac{2}{r-1}},
\end{align}
which implies that the operator $\B(\cdot):\V\cap\L^{r+1}\to\V^{\prime}+\L^{\frac{r+1}{r}}$ is well defined.

\subsection{Nonlinear operator}
Let us now consider the operator  $$\mathfrak{C}_r(\y):=|\y|^{r-1}\y, \ \text{ for }\ \y\in\L^{r+1}. $$ 
For notational convenience, we denote $\mathfrak{C}_r$ by $\mathfrak{C}$ throughout the remainder of the paper. It follows immediately that
\begin{align}\label{eqn-c}\langle\mathfrak{C}(\y),\y\rangle =\|\y\|_{\L^{r+1}}^{r+1}\end{align} and the map $\mathfrak{C}(\cdot):\L^{r+1}\to\L^{\frac{r+1}{r}}$ is G\^ateaux differentiable with its G\^ateaux derivative
\begin{align}\label{Gaetu}
	\mathfrak{C}'(\y)\z&=\left\{\begin{array}{cl}\z,&\text{ for }r=1,\\ \left\{\begin{array}{cc}|\y|^{r-1}\z+(r-1)\left(\frac{\y}{|\y|^{3-r}}(\y\cdot\z)\right),&\text{ if }\y\neq \mathbf{0},\\\mathbf{0},&\text{ if }\y=\mathbf{0},\end{array}\right.&\text{ for } 1<r<3,\\ |\y|^{r-1}\z+(r-1)\y|\y|^{r-3}(\y\cdot\z), &\text{ for }r\geq 3,\end{array}\right.
\end{align}
for all $\z\in\L^{r+1}$. 

\begin{lemma}[{\cite[Section 2.4]{MTMS}}]
	For all 	$\y,\z\in\L^{r+1}(\mathcal{O})$ and $r\geq 1$, we have 
	\begin{align}\label{2.23}
			&\langle\y|\y|^{r-1}-\z|\z|^{r-1},\y-\z\rangle\geq \frac{1}{2}\||\y|^{\frac{r-1}{2}}(\y-\z)\|_{\H}^2+\frac{1}{2}\||\z|^{\frac{r-1}{2}}(\y-\z)\|_{\H}^2\geq 0,
		\end{align}
		and 
	\begin{align}\label{Eqn-mon-lip}
		&\langle\y|\y|^{r-1}-\z|\z|^{r-1},\y-\z\rangle
		\geq\frac{1}{2^{r-1}}\|\y-\z\|_{\L^{r+1}}^{r+1}.
	\end{align}
\end{lemma}

\subsection{Coercivity and pseudomonotonicity}\label{sub-sec-coe}
The following lemma plays a crucial role in our analysis, a  proof of which can also be found in \cite[Theorem 2.5]{SGMTM}. For completeness, we provide a proof here also.
\begin{lemma}\label{lem-mon}
	For all 	$\y,\z\in\V\cap\L^{r+1}$ and $r>3$, the operator \begin{align}\label{eqn-op-f}\mathcal{F}(\y):=\mu\A\y+\B(\y)+\alpha\y+\beta\mathfrak{C}(\y)\end{align} satisfies 
	\begin{align}\label{eqn-fes}
		\langle\mathcal{F}(\y)-\mathcal{F}(\z),\y-\z\rangle+\frac{\varrho}{2\mu}\|\y-\z\|_{\H}^2\geq\frac{\mu}{2}\|\y-\z\|_{\V}^2+\alpha\|\y-\z\|_{\H}^2,
	\end{align}
	where 
	\begin{align}\label{eqn-rho}
		\varrho=\frac{r-3}{r-1}\left(\frac{2}{\beta\mu (r-1)}\right)^{\frac{2}{r-3}}.
	\end{align}
\end{lemma}
\begin{proof}
Using \eqref{eqn-def-A} and \eqref{eqn-coercive}, we infer for all $\y,\z\in\V$ that 
\begin{align}\label{eqn-aes}
	\mu\langle\A(\y-\z),\y-\z\rangle=\mu\|\mathrm{curl}(\y-\z)\|_{\H}^2=\mu\|\y-\z\|_{\V}^2. 
\end{align}
From \eqref{2.23}, we infer
	\begin{align}\label{eqn-ces}
	&\langle\mathfrak{C}(\y)-\mathfrak{C}(\z),\y-\z\rangle\geq \frac{1}{2}\||\y|^{\frac{r-1}{2}}(\y-\z)\|_{\H}^2+\frac{1}{2}\||\z|^{\frac{r-1}{2}}(\y-\z)\|_{\H}^2. 
\end{align}
Next, we estimate  $|\langle\B(\y)-\B(\z),\y-\z\rangle|$ using \eqref{eqn-b0} along with  H\"older's and Young's inequalities as 
\begin{align}\label{eqn-bdes}
|	\langle\B(\y)-\B(\z),\y-\z\rangle|&\leq |\langle \B(\y,\y-\z),\y-\z\rangle|+ |\langle \B(\y-\z,\z),\y-\z\rangle|\nonumber\\&=|\langle \B(\y-\z,\z),\y-\z\rangle|\leq \|\mathrm{curl}(\y-\z)\|_{\H}\|\z(\y-\z)\|_{\H}\nonumber\\&\leq\frac{\mu}{2}\|\y-\z\|_{\V}^2+\frac{1}{2\mu}\|\z(\y-\z)\|_{\H}^2. 
\end{align}
For $r>3$, we estimate the expression $\|\z(\y-\z)\|_{\H}^2$   by using H\"older's and Young's inequalities as mentioned in (\cite[Theorem 2.5]{SGMTM})
\begin{align}\label{eqn-des}
\|\z(\y-\z)\|_{\H}^2&=	\int_{\mathcal{O}}|\z(x)|^2|\y(x)-\z(x)|^2\d x\nonumber\\&\leq\left(\int_{\mathcal{O}}|\z(x)|^{r-1}|\y(x)-\z(x)|^2\d x\right)^{\frac{2}{r-1}}\left(\int_{\mathcal{O}}|\y(x)-\z(x)|^2\d x\right)^{\frac{r-3}{r-1}}\nonumber\\&\leq{\beta\mu }\left(\int_{\mathcal{O}}|\z(x)|^{r-1}|\y(x)-\z(x)|^2\d x\right)\nonumber\\&\quad+\frac{r-3}{r-1}\left(\frac{2}{\beta\mu (r-1)}\right)^{\frac{2}{r-3}}\left(\int_{\mathcal{O}}|\y(x)-\z(x)|^2\d x\right).
\end{align}
 Using \eqref{eqn-des} in \eqref{eqn-bdes}, we find 
\begin{align}\label{eqn-bes}
|	\langle\B(\y)-\B(\z),\y-\z\rangle|&\leq\frac{\mu }{2}\|\y-\z\|_{\V}^2+\frac{\beta}{2}\||\z|^{\frac{r-1}{2}}(\y-\z)\|_{\H}^2+\frac{\varrho}{2\mu}\|\y-\z\|_{\H}^2.
\end{align}
Combining \eqref{eqn-aes}, \eqref{eqn-ces} and \eqref{eqn-bes}, we finally get for $r>3$ 
\begin{align}\label{eqn-final-est}
\langle\mathcal{F}(\y)-\mathcal{F}(\z),\y-\z\rangle+\frac{\varrho}{2\mu}\|\y-\z\|_{\H}^2\geq\frac{\mu }{2}\|\y-\z\|_{\V}^2+\alpha\|\y-\z\|_{\H}^2\geq 0,
\end{align}
where $\varrho$ is defined in \eqref{eqn-rho} and the estimate \eqref{eqn-fes} follows.  
\end{proof}
\begin{remark}[{\cite[Remark 2.7]{SGMTM}}]\label{rem-mon}
1.	For the critical case $r=3$ ($d\in\{2,3\}$)  with $2\beta\mu \geq 1$, the operator $\F(\cdot):\V\cap\L^{4}\to \V^{\prime}+\L^{\frac{4}{3}}$ is globally monotone, that is, for all $\y,\z\in\V$, we have 
	\begin{align}\label{218}\langle\F(\y)-\F(\z),\y-\z\rangle\geq 0.\end{align}
	From \eqref{2.23}, we obtain
	\begin{align}\label{231}
		\beta\langle\mathfrak{C}(\y)-\mathfrak{C}(\z),\y-\z\rangle\geq\frac{\beta}{2}\|\y(\y-\z)\|_{\H}^2+\frac{\beta}{2}\|\z(\y-\z)\|_{\H}^2. 
	\end{align}
	We bound the term $|\langle\B(\y-\z,\y-\z),\z\rangle|$ by applying H\"older's and Young's inequalities as 
	\begin{align}\label{232}
		|\langle\B(\y-\z,\y-\z),\z\rangle|\leq\|\z(\y-\z)\|_{\H}\|\y-\z\|_{\V} \leq\frac{1}{2\beta} \|\y-\z\|_{\V}^2+\frac{\beta}{2 }\|\z(\y-\z)\|_{\H}^2.
	\end{align}
	Combining \eqref{eqn-aes}, \eqref{231} and \eqref{232}, we obtain 
	\begin{align}\label{eqn-critical}
		\langle\F(\y)-\F(\z),\y-\z\rangle\geq \left(\mu-\frac{1}{2\beta}\right)\|\y-\z\|_{\V}^2+\alpha\|\y-\z\|_{\H}^2+ \frac{\beta}{2}\|\y(\y-\z)\|_{\H}^2\geq 0,
	\end{align}
	provided $2\beta\mu \geq 1$. 
	
	2. For $d\in\{2,3\}$ and $r\in[1,3]$, one obtains the estimate $|\langle\B(\y-\z,\y-\z),\z\rangle|$ using Remark \ref{rem-equiv}, H\"older's, Ladyzhenskaya's and Young's inequalities  as 
	\begin{align}\label{2.21}
		|\langle\B(\y-\z,\y-\z),\z\rangle|&\leq \|\y-\z\|_{\V}\|\y-\z\|_{\L^4}\|\z\|_{\L^4}\nonumber\\&\leq C\|\y-\z\|_{\H}^{1-\frac{d}{4}}\|\nabla(\y-\z)\|_{\H}^{\frac{d}{4}}\|\y-\z\|_{\V}\|\z\|_{\L^4}\nonumber\\&\leq C\|\y-\z\|_{\H}^{1-\frac{d}{4}}\|\y-\z\|_{\V}^{1+\frac{d}{4}}\|\z\|_{\L^4}\nonumber\\&\leq  \frac{\mu }{2}\|\y-\z\|_{\V}^2+\frac{C}{\mu^{\frac{4+d}{4-d}}} \|\z\|_{\L^4}^{\frac{8}{4-d}}\|\y-\z\|_{\H}^2.
	\end{align}
	Combining \eqref{eqn-aes}, \eqref{eqn-ces} and \eqref{2.21}, we obtain 
		\begin{align}\label{eqn-diff}
		\langle\F(\y)-\F(\z),\y-\z\rangle+ \frac{C}{\mu^{\frac{4+d}{4-d}}} \|\z\|_{\L^4}^{\frac{8}{4-d}}\|\y-\z\|_{\H}^2\geq 0,
	\end{align}
	which implies 
	\begin{align}\label{fe2}
		\langle\F(\y)-\F(\z),\y-\z\rangle+ \frac{C}{\mu^{\frac{4+d}{4-d}}}N^{\frac{8}{4-d}}\|\y-\z\|_{\H}^2\geq 0,
	\end{align}
	for all $\z\in{\mathbb{B}}_N$, where 
	$
{\mathbb{B}}_N:=\big\{\s\in\L^4:\|\s\|_{\L^4}\leq N\big\}.
	$ Hence, $\F(\cdot)$ is locally monotone in the present setting. 
	
	3.  As discussed  in \cite{SGMTM, YZ}, for $r\geq 3$, one can estimate $|\langle\B(\y-\z,\y-\z),\z\rangle|$ as 
	\begin{align}\label{2.26}
		&	|\langle\B(\y-\z,\y-\z),\z\rangle|\nonumber\\&\leq \mu \|\y-\z\|_{\V}^2+\frac{1}{4\mu }\int_{\mathcal{O}}|\y(x)-\z(x)|^2\left(|\z(x)|^{r-1}+1\right)\frac{|\z(x)|^2}{|\z(x)|^{r-1}+1}\d x\nonumber\\&\leq \mu \|\y-\z\|_{\V}^2+\frac{1}{4\mu }\int_{\mathcal{O}}|\z(x)|^{r-1}|\y(x)-\z(x)|^2\d x+\frac{1}{4\mu }\int_{\mathcal{O}}|\y(x)-\z(x)|^2\d x,
	\end{align}
	where we have used the fact that $\left\|\frac{|\z|^2}{|\z|^{r-1}+1}\right\|_{\L^{\infty}}<1$, for $r\geq 3$. From the above estimate, we obtain a local monotonicity result, similar to \eqref{eqn-fes}, in which $\frac{\varrho}{2\mu}$ is replaced by $\frac{1}{4\mu}$, under the condition $2\beta\mu \geq 1$. 
	\end{remark}

	\begin{lemma}\label{lem-pseudo}
		For $1\leq r<\infty$, the operator $\mathcal{F}(\cdot)$ defined in \eqref{eqn-op-f} is coercive and pseudomonotone.
	\end{lemma}
	\begin{proof}
		The proof is divided into the following steps:
		\vskip 0.1cm
		\noindent 
		\textbf{Step 1.} 
Our first step is to prove that the operator $\mathcal{F}(\cdot):\V\cap\L^{r+1}\to\V^{\prime}+\L^{\frac{r+1}{r}}$  defined by $\mathcal{F}(\y):=\mu\A\y+\B(\y)+\alpha\y+\beta\mathfrak{C}(\y)$ is coercive. We infer from \eqref{eqn-def-A}, \eqref{eqn-b0} and \eqref{eqn-c} that 
		\begin{align}\label{eqn-coer}
			\langle\mathcal{F}(\y),\y\rangle =\mu\|\y\|_{\V}^2+\alpha\|\y\|_{\H}^2+\beta\|\y\|_{\L^{r+1}}^{r+1}. 
		\end{align}
		We defined an equivalent  norm in $\V\cap\L^{r+1}$ as $\sqrt{\|\y\|_{\V}^2+\|\y\|_{\L^{r+1}}^{2}}$. Therefore, using \eqref{eqn-equiv-1}, we have 
		\begin{align}\label{eqn-coercive-1}
			\frac{\langle\mathcal{F}(\y),\y\rangle}{\|\y\|_{\V\cap\L^{r+1}}}&\geq\frac{\min\{\mu,\beta\}\left(\|\y\|_{\V}^2+\|\y\|_{\L^{r+1}}^2-1\right)}{\sqrt{\|\y\|_{\V}^2+\|\y\|_{\L^{r+1}}^{2}}},
		\end{align}
		where we have used the fact that $x^2\leq x^{r+1} + 1,$ for all $x\geq 0$ and $r\geq 1$. Hence, we obtain 
		\begin{align*}
		\lim\limits_{\|\y\|_{\V\cap\L^{r+1}}\to\infty}	\frac{\langle\mathcal{F}(\y),\y\rangle}{\|\y\|_{\V\cap\L^{r+1}}}=\infty,
			\end{align*}
			so that the operator $\mathcal{F}(\cdot):\V\cap\L^{r+1}\to \V^{\prime}+\L^{\frac{r+1}{r}}$ is coercive. 
			
			\vskip 0.1cm
			\noindent 
			\textbf{Step 2.}
To establish the pseudomonotonicity of the operator $\mathcal{F}(\cdot):\V\cap\L^{r+1}\to \V^{\prime}+\L^{\frac{r+1}{r}}$, we first verify the boundedness of 
$\mathcal{F}$ (Definition \ref{def-pseudo}-(1)). From \eqref{eqn-abound}, it follows that for all $\y,\z\in\V,$ 
			\begin{align}
				|\langle\A\y,\z\rangle|&=|a(\y,\z)|\leq\|\y\|_{\V}\|\z\|_{\V}. 
			\end{align}
			The bound \eqref{eqn-bbound} implies  for all $\y,\z\in\V$  that 
			\begin{align}
				|\langle\B(\y),\z\rangle|\leq C\|\y\|_{\V}^2\|\z\|_{\V}. 
			\end{align}
			Moreover, an application of H\"older's inequality yields for all $\y,\z\in\L^{r+1}$  that 
			\begin{align}
					|\langle\mathfrak{C}(\y),\z\rangle|&\leq\|\y\|_{\L^{r+1}}^r\|\z\|_{\L^{r+1}}. 
			\end{align}
			Therefore, for all $\y\in\V\cap\L^{r+1}$, we have 
			\begin{align}
				\sup_{\|\z\|_{\V\cap\L^{r+1}}\leq 1}|\langle\mathcal{F}(\y),\z\rangle|\leq \left(\mu\|\y\|_{\V}+C\alpha\|\y\|_{\H}+C\|\y\|_{\V}^2+\beta\|\y\|_{\L^{r+1}}^r \right)<\infty, 
			\end{align}
		and 	the boundedness of the operator $\mathcal{F}(\cdot):\V\cap\L^{r+1}\to \V^{\prime}+\L^{\frac{r+1}{r}}$ follows.

			\vskip 0.1cm
		\noindent 
		\textbf{Step 3.} 
	Next, we establish that the operator $\wi{\mathcal{F}}(\cdot):=\mu\A+\alpha\mathrm{I}+\beta\mathfrak{C}(\cdot):\V\cap\L^{r+1}\to\V^{\prime}+\L^{\frac{r+1}{r}}$ is pseudomonotone. Suppose 
$\{\y_n\}_{n\in\N}$ is a sequence in $\V\cap\L^{r+1}$ such that 
		\begin{align}\label{eqn-pseudo}
		\y_n\xrightarrow{w}\y\ \text{ in }\ \V\cap\L^{r+1}\ \text{ and }\ \limsup\limits_{n\to\infty} \langle\wi{\mathcal{F}}(\y_n),\y_n-\y\rangle\leq 0, 
		\end{align}
		and let $\z\in \V\cap\L^{r+1}$. Since every weakly convergent sequence is bounded, the sequence $\{\y_n\}_{n\in\N}$ is bounded independent of $n$  in $\V\cap\L^{r+1},$ and 
		\begin{align}\label{eqn-weak}
			\y_n\xrightarrow{w}\y\ \text{ in }\ \V\ \text{ and }\ \y_n\xrightarrow{w}\y\ \text{ in }\ \L^{r+1}. 
		\end{align}
		Since the embedding of $\V\hookrightarrow\H$ is compact. Therefore, there exists a subsequence of $\{\y_n\}_{n\in\N}$ (still denoted by the same symbol) such that 
		\begin{align}\label{eqn-strong}
			\y_n\to \y\ \text{ in }\ \H. 
		\end{align}
		Moreover, along a further subsequence (still denoted by the same symbol), we have the following convergence also:  
		\begin{align}\label{eqn-ae}
			\y_n(x)\to\y(x) \ \text{ for a.e. }\ x\in\mathcal{O}. 
		\end{align}
From the relations \eqref{eqn-aes} and \eqref{eqn-ces}, we know that the operator $\wi{\mathcal{F}}$ is monotone. Therefore, we have 
		\begin{align*}
			\langle\wi{\mathcal{F}}(\y_n)-\wi{\mathcal{F}}(\y),\y_n-\y\rangle\geq 0,
		\end{align*}
		which implies 
		\begin{align*}
				\langle\wi{\mathcal{F}}(\y_n),\y_n-\y\rangle\geq 	\langle\wi{\mathcal{F}}(\y),\y_n-\y\rangle. 
		\end{align*}
		Taking the limit inferior on both sides, we deduce 
		\begin{align*}
		\liminf_{n\to\infty}	\langle\wi{\mathcal{F}}(\y_n),\y_n-\y\rangle\geq  	\liminf_{n\to\infty}	\langle\wi{\mathcal{F}}(\y),\y_n-\y\rangle=0, 
		\end{align*}
		where we have used the fact that  $	\y_n\xrightarrow{w}\y\ \text{ in }\ \V\cap\L^{r+1}$. From the assumption \eqref{eqn-pseudo}, we know that 
		\begin{align*}
			\limsup\limits_{n\to\infty} \langle\wi{\mathcal{F}}(\y_n),\y_n-\y\rangle\leq 0.
		\end{align*}
		Combining the above two results, we immediately have 
		\begin{align*}
			\lim_{n\to\infty}\langle\wi{\mathcal{F}}(\y_n),\y_n-\y\rangle=0. 
		\end{align*}
		Let us now show that $\wi{\mathcal{F}}(\y_n)\xrightarrow{w}\wi{\mathcal{F}}(\y)$ in $\V^{\prime}+\L^{\frac{r+1}{r}}$. For all $\z\in\V\cap\L^{r+1}$, we consider
		\begin{align}\label{eqn-lim}
			\langle \wi{\mathcal{F}}(\y_n)-\wi{\mathcal{F}}(\y),\z\rangle&=\mu \langle\A(\y_n-\y),\z\rangle+\alpha(\y_n-\y,\z)+\beta\langle\mathfrak{C}(\y_n)-\mathfrak{C}(\y),\z\rangle. 
		\end{align}
The first two terms on the right hand side of the above inequality converges to zero using \eqref{eqn-weak} and \eqref{eqn-strong}. It is left to show that $\langle\mathfrak{C}(\y_n)-\mathfrak{C}(\y),\z\rangle\to 0$  as $n\to\infty$ for all $\z\in\L^{r+1}$. From the convergence \eqref{eqn-ae}, we immediately get
		\begin{align}
			|\y_n(x)|^{r-1}\y_n(x)\to 	|\y(x)|^{r-1}\y(x)\ \text{ for a.e. }\ x\in\mathcal{O}.
		\end{align}
		Moreover, we know that $\|\mathfrak{C}(\y_n)\|_{\L^{\frac{r+1}{r}}}=\|\y_n\|_{\L^{r+1}}^r\leq C$ and $\y\in\L^{r+1}$. An application of Lemma \ref{Lem-Lions} yields 
		\begin{align*}
			\mathfrak{C}(\y_n)\xrightarrow{w}\mathfrak{C}(\y)\ \text{ in }\ \L^{\frac{r+1}{r}}\ \text{ as } \ n\to\infty,
		\end{align*}
		so that from \eqref{eqn-lim}, we  deuce 
			\begin{align}\label{eqn-lim-n}
			\langle \wi{\mathcal{F}}(\y_n)-\wi{\mathcal{F}}(\y),\z\rangle\to 0\ \text{ as }\ n\to\infty,
			\end{align}
			for all  $\z\in\V\cap\L^{r+1}$. Therefore, both conditions given in \eqref{eqn-con-pseudo} are satisfied, and hence the operator $\wi{\mathcal{F}}(\cdot):\V\cap\L^{r+1}\to\V^{\prime}+\L^{\frac{r+1}{r}}$ is pseudomonotone. 
				\vskip 0.1cm
			\noindent 
			\textbf{Step 4.} \emph{Pseudomonotonicity of the operator $\mathcal{F}(\y):=\wi{\mathcal{F}}(\y)+\B(\y).$} 
			To complete the proof of pseudomonotone property of $\mathcal{F}$, we need to verify Definition \ref{def-pseudo}-(2).  Let us consider a sequence $\{\y_n\}_{n\in\N}\in\V\cap\L^{r+1}$ such that 
			\begin{align}\label{eqn-pseudo-1}
				\y_n\xrightarrow{w}\y\ \text{ in }\ \V\cap\L^{r+1}\ \text{ and }\ \limsup\limits_{n\to\infty} \langle{\mathcal{F}}(\y_n),\y_n-\y\rangle\leq 0, 
			\end{align}
			and let $\z\in \V\cap\L^{r+1}$.	Using \eqref{eqn-b0}, we immediately get 
			\begin{align}\label{eqn-bcon}
				\langle\B(\y_n),\y_n-\z\rangle-	\langle\B(\y),\y-\z\rangle&=\langle\B(\y_n),\y_n\rangle-\langle\B(\y_n),\z\rangle-\langle\B(\y),\y\rangle+\langle\B(\y),\z\rangle\nonumber\\&=\langle\B(\y),\z\rangle-\langle\B(\y_n),\z\rangle\nonumber\\&=\langle\B(\y-\y_n,\y),\z\rangle+\langle\B(\y_n,\y-\y_n),\z\rangle. 
			\end{align}
Since $\B:\V\to\V^{\prime}$ is bounded, the first term $\langle\B(\y-\y_n,\y),\z\rangle\to 0 $ as $n\to\infty$ due to the weak convergence $	\y_n\xrightarrow{w}\y\ \text{ in }\ \V$. To estimate the second term on the right hand side of the equality \eqref{eqn-bcon}, we apply H\"older's and Ladyzhenskaya's inequalities, which can be written as follows:
			\begin{align*}
				|\langle\B(\y_n,\y-\y_n),\z\rangle|&\leq\|\y_n\|_{\V}\|\y-\y_n\|_{\L^{4}}\|\z\|_{\L^4}\nonumber\\&\leq C \|\y_n\|_{\V}\|\y-\y_n\|_{\V}^{\frac{d}{4}}\|\y-\y_n\|_{\H}^{1-\frac{d}{4}}\|\z\|_{\V}\nonumber\\&\leq C\|\y_n\|_{\V}\left(\|\y\|_{\V}^{\frac{d}{4}}+\|\y_n\|_{\V}^{\frac{d}{4}}\right)\|\y-\y_n\|_{\H}^{1-\frac{d}{4}}\|\z\|_{\V}\nonumber\\&\to 0\ \text{ as }\ n\to\infty, 
			\end{align*}
			where we have used the strong convergence \eqref{eqn-strong}. Consequently, from  \eqref{eqn-bcon}, we deduce that
			\begin{align}\label{eqn-bcon-1}
				\lim_{n\to\infty}	\langle\B(\y_n),\y_n-\z\rangle=	\langle\B(\y),\y-\z\rangle\ \text{ for all }\ \z\in\V. 
			\end{align}
			By taking $\z=\y$ in \eqref{eqn-bcon-1}, we have 
			\begin{align}\label{eqn-bcon-2}
				\lim_{n\to\infty}	\langle\B(\y_n),\y_n-\y\rangle=0. 
			\end{align}
			Let us now consider 
			\begin{align}
				\limsup_{n\to\infty}\langle\wi{\mathcal{F}}(\y_n),\y_n-\y\rangle&= \limsup_{n\to\infty}\langle\wi{\mathcal{F}}(\y_n),\y_n-\y\rangle+\lim_{n\to\infty}	\langle\B(\y_n),\y_n-\y\rangle\nonumber\\&=\limsup_{n\to\infty}\langle\wi{\mathcal{F}}(\y_n)+\B(\y_n),\y_n-\y\rangle\nonumber\\&=\limsup_{n\to\infty}\langle{\mathcal{F}}(\y_n),\y_n-\y\rangle\leq 0, 
			\end{align}
			by using \eqref{eqn-pseudo-1}. Furthermore, since the the operator $\wi{\mathcal{F}}$ is pseudomonotone, it follows that 
			\begin{align}\label{eqn-fcon-1}
				\langle\wi{\mathcal{F}}(\y),\y-\z\rangle\leq \liminf_{n\to\infty}\langle\wi{\mathcal{F}}(\y_n),\y_n-\z\rangle\ \text{ for all }\ \z\in\V\cap\L^{r+1}. 
			\end{align}
			Therefore, using \eqref{eqn-bcon-1} and \eqref{eqn-fcon-1}, we finally have 
			\begin{align}
				\langle{\mathcal{F}}(\y),\y-\z\rangle\leq \liminf_{n\to\infty}\langle{\mathcal{F}}(\y_n),\y_n-\z\rangle\ \text{ for all }\ \z\in\V\cap\L^{r+1}, 
			\end{align}
			which completes the proof. 
	\end{proof}

	\section{Stationary problem: A boundary hemivariational inequality}\label{sec3}\setcounter{equation}{0}
Our aim in this section is to show the existence and uniqueness of a weak solution to a boundary hemivariational inequality for stationary 2D and 3D CBF equations by using the abstract result Theorem \ref{thm-surjective}.	

\subsection{An abstract hemivariational inequality}
We first  formulate the problem under our consideration, that is, an abstract  hemivariational inequality for stationary CBF equations, as follows: 
	\begin{problem}\label{prob-station}
		Find $\y\in\V\cap\L^{r+1}$ such that 
		\begin{equation}
			\mathcal{F}(\y)+\ell^*\partial\psi(\ell\y)\ni \f,
		\end{equation}
		where $\mathcal{F}(\cdot)$ is defined in \eqref{eqn-op-f}, $\f\in\V^{\prime}$, $\psi:\mathbb{U}\to\R$, where $\mathbb{U}$ is a reflexive Banach space, $\partial\psi$ is the subdifferential of $\psi(\cdot)$ in the sense of Clarke, $\ell : \V \to \mathbb{U}$ is a linear, continuous operator and $\ell^*:\mathbb{U}^{\prime}\to\V^{\prime}$ is the adjoint operator to $\ell$.
		\end{problem}

		Problem \eqref{prob-station} can be equivalently reformulated as follows:
		\begin{problem}\label{prob-station-1}
			Find $(\y,\boldsymbol{\eta})\in\V\cap\L^{r+1}\times\mathbb{U}^{\prime}$ such that 
			\begin{equation}
				\left\{	\begin{aligned}
					&\mathcal{F}(\y)+\ell^*\boldsymbol{\eta}= \f,\\
					&\boldsymbol{\eta}\in\partial\psi(\ell\y). 
					\end{aligned}
						\right. 
			\end{equation}
		\end{problem}
	
	For our analysis, we refer to the following conditions:
	\begin{hypothesis}\label{hyp-psi-ell}
		The functional $\psi:\mathbb{U}\to\mathbb{R}$ satisfies the following assumptions: 
		\begin{enumerate}
			\item [(H1)] $\psi$ is locally Lipschitz;
			\item  [(H2)] $\partial\psi$ satisfies the growth condition $$\|\boldsymbol{\eta}\|_{\mathbb{U}^{\prime}}\leq C_{\psi}(1+\|\boldsymbol{\xi}\|_{\mathbb{U}}),$$ for all $\boldsymbol{\eta}\in\partial\psi(\boldsymbol{\xi})$, $\boldsymbol{\xi}\in\U$, with a constant $C_{\psi}>0$;
			\item [(H3)] $\psi$ satisfies the relaxed monotonicity condition (\cite[Definition 3.49]{SMAOMS}) $${}_{\U^{\prime}}\langle\boldsymbol{\eta}_1-\boldsymbol{\eta}_2,\boldsymbol{\xi}_1-\boldsymbol{\xi}_2\rangle_{\U}\geq -m_1\|\boldsymbol{\xi}_1-\boldsymbol{\xi}_2\|_{\mathbb{U}}^2,$$ for all $\boldsymbol{\eta}_i\in\partial\psi(\boldsymbol{\xi}_i),$ $\boldsymbol{\xi}_i\in\U$, $i=1,2$ with $m_1\geq 0$;
			\item [(H4)] the operator $\ell\in\mathcal{L}(\V;\U)$ is compact. 
		\end{enumerate}
	\end{hypothesis}
	Let us define a multi-valued operator $\mathscr{L}(\cdot):\V\cap\L^{r+1}\to 2^{\V^{\prime}+\L^{\frac{r+1}{r}}}$ by 
	\begin{align}\label{eqn-opl}
		\mathscr{L}(\y):=\mathcal{F}(\y)+\ell^{*}\partial\psi(\ell\y),\ \y\in\V\cap\L^{r+1}. 
	\end{align}
	Next, we prove that the operator $\mathscr{L}$ is pseudomonotone.
	\begin{lemma}\label{lem-pseudo-mon}
		Let $\mathcal{F}(\cdot)$ be the operator defined in \eqref{eqn-op-f},  and let $\psi$ and $\ell$ satisfy Hypothesis \ref{hyp-psi-ell}. Then 
		\begin{enumerate}
			\item for $\mu> C_{\psi}\|\ell\|_{\mathcal{L}(\V;\U)}^2$, the operator $\mathscr{L}$ is coercive;
			\item $\mathscr{L}$ is pseudomonotone. 
		\end{enumerate}
	\end{lemma}
	\begin{proof}
		Let us first show that $\mathscr{L}$ is coercive for $\mu> C_{\psi}\|\ell\|_{\mathcal{L}(\V;\U)}^2$. Let $\y\in\V\cap\L^{r+1}$ and $\y^*\in\mathscr{L}(\y)$. Then for some $\boldsymbol{\eta}\in \partial\psi(\ell\y)$, we have 
		\begin{align}
			\y^*=\mathcal{F}(\y)+\ell^*\boldsymbol{\eta}. 
		\end{align}
		Using \eqref{eqn-coer}, we find
		\begin{align}\label{eqn-coer-1}
			\langle\y^*,\y\rangle=\mu\|\y\|_{\V}^2+\alpha\|\y\|_{\H}^2+\beta\|\y\|_{\L^{r+1}}^{r+1}+{}_{\U^{\prime}}\langle\boldsymbol{\eta},\ell\y\rangle_{\U}.
		\end{align}
Applying Hypothesis \ref{hyp-psi-ell} (H2), the final term appearing on the right hand side in \eqref{eqn-coer-1} is bounded as follows:
		\begin{align}\label{eqn-coer-2}
			|{}_{\U^{\prime}}\langle\boldsymbol{\eta},\ell\y\rangle_{\U}|&\leq \|\boldsymbol{\eta}\|_{\U^{\prime}}\|\ell\y\|_{\U}\leq\|\boldsymbol{\eta}\|_{\U^{\prime}}\|\ell\|_{\mathcal{L}(\V;\U)}\|\y\|_{\V}\nonumber\\&\leq C_{\psi}\left(1+\|\ell\y\|_{\U}\right)\|\ell\|_{\mathcal{L}(\V;\U)}\|\y\|_{\V}\nonumber\\&\leq C_{\psi}\|\ell\|_{\mathcal{L}(\V;\U)}^2\|\y\|_{\V}^2+C_{\psi}\|\ell\|_{\mathcal{L}(\V;\U)}\|\y\|_{\V}. 
		\end{align}
		It follows from \eqref{eqn-coer-1} and \eqref{eqn-coer-2} that 
		\begin{align}\label{eqn-coer-3}
			\langle\y^*,\y\rangle\geq\left(\mu-C_{\psi}\|\ell\|_{\mathcal{L}(\V;\U)}^2\right)\|\y\|_{\V}^2+\alpha\|\y\|_{\H}^2+\beta\|\y\|_{\L^{r+1}}^{r+1}-C_{\psi}\|\ell\|_{\mathcal{L}(\V;\y)}\|\y\|_{\V}. 
		\end{align}
		A calculation similar to \eqref{eqn-coercive-1} yields 
		\begin{align}
			\frac{\langle\y^*,\y\rangle}{\|\y\|_{\V\cap\L^{r+1}}}\geq\frac{\min\left\{\left(\mu-C_{\psi}\|\ell\|_{\mathcal{L}(\V;\U)}^2\right),\beta\right\}\left(\|\y\|_{\V}^2+\|\y\|_{\L^{r+1}}^2-1\right)-C_{\psi}\|\ell\|_{\mathcal{L}(\V;\U)}\|\y\|_{\V}}{\sqrt{\|\y\|_{\V}^2+\|\y\|_{\L^{r+1}}^{2}}},
		\end{align} 
	Therefore, for $\mu> C_{\psi}\|\ell\|_{\mathcal{L}(\V;\U)}^2$, we immediately have 
		\begin{align*}
		\lim\limits_{\|\y\|_{\V\cap\L^{r+1}}\to\infty}	\frac{\langle\y^{*},\y\rangle}{\|\y\|_{\V\cap\L^{r+1}}}=\infty,
	\end{align*}
	so that the operator $\mathscr{L}(\cdot):\V\cap\L^{r+1}\to 2^{\V^{\prime}+\L^{\frac{r+1}{r}}}$ is coercive. 
		
		We aim to demonstrate that the operator $\mathscr{L}$ is pseudomonotone. According to Lemma \ref{lem-pseudo}, the operator $\mathcal{F}(\cdot)$ is already known to possess this property. To complete the argument for $\mathscr{L}$, we focus on verifying the pseudomonotonicity of the set-valued operator $\ell^* \partial\psi(\ell \cdot)$, using the sufficient conditions outlined in Proposition \ref{prop-suff-pseudo}. Our approach is inspired by techniques in \cite[Lemma 2]{PKa} and \cite[Proposition 5.6]{WHSM}.
	
	To begin, condition (1) of Proposition \ref{prop-suff-pseudo} requires the values of the operator to be nonempty, convex, and weakly closed. This follows from standard properties of the Clarke subdifferential of locally Lipschitz functionals on reflexive Banach spaces (see \cite[Proposition 2.1.2]{FHC}). Given the linearity of the mapping $\ell$, and the fact that  $\ell$ is a compact operator, it follows that $\ell^* \partial\psi(\ell \cdot)$ inherits these properties, mapping into nonempty, closed, and convex subsets of $\V^{\prime}$.
	
	Condition (2) of Proposition \ref{prop-suff-pseudo} concerns the growth of the operator. From Hypothesis \ref{hyp-psi-ell} (H2), we know that for each $\boldsymbol{\eta}_n \in \partial\psi(\ell \z_n)$, there exists a constant $C_\psi > 0$ such that
	\begin{align}\label{uniform est}
	\|\boldsymbol{\eta}_n\|_{\U^{\prime}} \leq C_\psi \left(1 + \|\ell \z_n\|_{\U} \right).
	\end{align}
		Using the continuity of $\ell \in \mathcal{L}(\V; \U)$ and boundedness of $\{ \z_n \} \subset \V$, this estimate implies uniform boundedness of $\{ \boldsymbol{\eta}_n \} \subset \U^{\prime}$.
	
	We now address condition (3) of Proposition \ref{prop-suff-pseudo}. Let $\z_n \text{ and } \boldsymbol{\zeta}_n $ be two sequences such that 
	\begin{align*}
	\z_n \xrightarrow{w} \z \text{ in } \V \text{ and } \boldsymbol{\zeta}_n \xrightarrow{w} \boldsymbol{\zeta} \text{ in } \V^{\prime},
	\end{align*}
	with $\boldsymbol{\zeta}_n \in \ell^* \partial\psi(\ell \z_n)$. Since $\ell$ is a compact operator from $\V$ to $\U$ (by Hypothesis \ref{hyp-psi-ell} (H4)), we have $\ell \z_n \to \ell \z$ in $\U$.	Define $\boldsymbol{\eta}_n \in \partial\psi(\ell \z_n)$ such that $\boldsymbol{\zeta}_n = \ell^* \boldsymbol{\eta}_n$. From the growth estimate above and the reflexivity of $\U^{\prime}$, there exists a weakly convergent subsequence $\boldsymbol{\eta}_n \xrightarrow{w} \boldsymbol{\eta}$ in $\U^{\prime}$. By the closedness of the graph of $\partial\psi$ in the weak topology (see \cite[Proposition 2.1.5]{FHC}), it follows that $\boldsymbol{\eta} \in \partial\psi(\ell \z)$, and hence $\boldsymbol{\zeta} = \ell^* \boldsymbol{\eta} \in \ell^* \partial\psi(\ell \z)$.	Furthermore, we verify the required convergence of duality pairings:
		$$
	\langle \boldsymbol{\zeta}_n, \z_n \rangle ={}_{\U^{\prime}}\langle\boldsymbol{\eta}_n, \ell \z_n \rangle_{\U} \to {}_{\U^{\prime}}\langle \boldsymbol{\eta}, \ell \z \rangle_{ \U} = \langle \boldsymbol{\zeta}, \z \rangle,
	$$
	which holds for the full sequence by the uniqueness of the weak limit and strong convergence of $\ell \z_n$.
	
	As both components of $\mathscr{L}$, namely $\mathcal{F}(\cdot)$ and $\ell^* \partial\psi(\ell \cdot)$, are pseudomonotone, and the sum of pseudomonotone operators remains pseudomonotone (cf. \cite[Proposition 1.3.68]{ZdSmP}), we conclude that $\mathscr{L}$ is pseudomonotone.
	\end{proof}
	
	The following theorem follows from the fundamental surjectivity result of nonlinear analysis for pseudomonotone and coercive operators (see Theorem \ref{thm-surjective}). 
	\begin{theorem}\label{thm-station-exis}
	For  $r\geq 1$ and $\mu> C_{\psi}\|\ell\|_{\mathcal{L}(\V;\U)}^2,$ under Hypothesis \ref{hyp-psi-ell} (H1), (H2) and (H4),	the operator $\mathscr{L}:\V\cap\L^{r+1}\to 2^{\V^{\prime}+\L^{\frac{r+1}{r}}}$ defined by \eqref{eqn-opl} is surjective, that is, 
	\begin{align*}
	\mbox{for every $\g\in\V^{\prime}+\L^{\frac{r+1}{r}}$, there is a $\y\in\V\cap\L^{r+1}$ such that $\mathscr{L}(\y)\ni\g$. }
	\end{align*}
	\end{theorem}
Since $\f\in\V^{\prime}\hookrightarrow\V^{\prime}+\L^{\frac{r+1}{r}}$, we obtain a solution to 	the  Problem \ref{prob-station} or Problem \ref{prob-station-1}. Under Hypothesis \ref{hyp-psi-ell} (H3) and further assumptions, we show that the solution is unique. 
\begin{theorem}\label{thm-station-unique}
	For  $r\geq 1$ and $\mu> C_{\psi}\|\ell\|_{\mathcal{L}(\V;\U)}^2,$ under Hypothesis \ref{hyp-psi-ell} (H1)-(H4), for $\f\in\V^{\prime}$, let $\y\in\V\cap\L^{r+1}$ be a solution to  the Problem \ref{prob-station}. Then the following bounds hold: 
	\begin{align}\label{eqn-ener-est}
		\|\y\|_{\V}\leq \frac{C_{\psi}\|\ell\|_{\mathcal{L}(\V;\U)}+\|\f\|_{\V^{\prime}}}{\mu-C_{\psi}\|\ell\|_{\mathcal{L}(\V;\U)}^2}\ \text{ and }\ \|\y\|_{\L^{r+1}}^{r+1}\leq \frac{\left(C_{\psi}\|\ell\|_{\mathcal{L}(\V;\U)}+\|\f\|_{\V^{\prime}}\right)^2}{\beta\left(\mu-C_{\psi}\|\ell\|_{\mathcal{L}(\V;\U)}^2\right)}. 
	\end{align}
	
Furthermore, for $d\in\{2,3\}$ and $3<r<\infty$, if 
	\begin{align}\label{eqn-unique}
		\mu>\min\left\{\frac{\tilde{\varrho}}{2\alpha}+m_1\|\ell\|_{\mathcal{L}(\V;\U)}^2,\max\left\{\frac{1}{2\beta},\frac{1}{4\alpha}+m_1\|\ell\|_{\mathcal{L}(\V;\U)}^2\right\} \right\},
		\end{align}
		where \begin{align}\label{eq-tilder}
		\tilde{\varrho}= \frac{r-3}{r-1}\left(\frac{2}{\beta(\mu-m_1\|\ell\|_{\mathcal{L}(\V;\U)}^2) (r-1)}\right)^{\frac{2}{r-3}},
	\end{align}
and for $d\in\{2,3\}$ and $r\in[1,3]$, and some constant $C >0$, if 
	\begin{align}\label{eqn-unique-1}
		\mu>\left(\frac{C}{\alpha}\right)^{\frac{4-d}{4+d}}\left(\frac{C_{\psi}\|\ell\|_{\mathcal{L}(\V;\U)}+\|\f\|_{\V^{\prime}}}{\mu-C_{\psi}\|\ell\|_{\mathcal{L}(\V;\U)}^2}\right)^{\frac{8}{4+d}}+m_1\|\ell\|_{\mathcal{L}(\V;\U)}^2,
	\end{align}
	 then the solution $\y$ to the Problem \ref{prob-station}  is unique. 
\end{theorem}
\begin{proof}
	Since $\y\in\V\cap\L^{r+1}$ is a solution to  the Problem \ref{prob-station}, we find
	\begin{align*}
		\langle\mathcal{F}(\y),\y\rangle+{}_{\U^{\prime}}\langle\boldsymbol{\eta},\ell\y\rangle_{\U}=\langle\f,\y\rangle,
	\end{align*}
	where $\boldsymbol{\eta}\in\partial\psi(\ell\y).$ The estimates \eqref{eqn-coer-1} and \eqref{eqn-coer-2} imply 
	\begin{align}\label{eqn-ener-1}
		\mu\|\y\|_{\V}^2+\alpha\|\y\|_{\H}^2+\beta\|\y\|_{\L^{r+1}}^{r+1}\leq C_{\psi}\|\ell\|_{\mathcal{L}(\V;\U)}^2\|\y\|_{\V}^2+C_{\psi}\|\ell\|_{\mathcal{L}(\V;\U)}\|\y\|_{\V}+\|\f\|_{\V^{\prime}}\|\y\|_{\V}.
	\end{align}
	For $\y=\boldsymbol{0}$, the above estimate is trivially satisfied. Therefore, for $\y\neq\boldsymbol{0}$, we further have 
	\begin{align}\label{eqn-ener-2}
		\left(\mu-C_{\psi}\|\ell\|_{\mathcal{L}(\V;\U)}^2\right)\|\y\|_{\V}\leq C_{\psi}\|\ell\|_{\mathcal{L}(\V;\U)}+\|\f\|_{\V^{\prime}}. 
	\end{align}
	Therefore, for $\mu>C_{\psi}\|\ell\|_{\mathcal{L}(\V;\U)}^2$, the first estimate in \eqref{eqn-ener-est} follows. Using \eqref{eqn-ener-2} in \eqref{eqn-ener-1}, we immediately get the second estimate in \eqref{eqn-ener-est} also. 
	
	Let us now prove the uniqueness. Let $\y_1,\y_2\in\V\cap\L^{r+1}$ be two solutions to  the Problem \ref{prob-station}. Then $\y_1-\y_2$ satisfies
	\begin{align}\label{eqn-differece}
		\langle\mathcal{F}(\y_1)-\mathcal{F}(\y_2),\z\rangle+{}_{\U^{\prime}}\langle\boldsymbol{\eta}_1-\boldsymbol{\eta}_2,\ell\z\rangle_{\U}=0,\text{ for all } \z\in\V\cap\L^{r+1}, 
	\end{align}
	where $\boldsymbol{\eta}_1\in\partial\psi(\ell\y_1)$ and $\boldsymbol{\eta}_2\in\partial\psi(\ell\y_2)$. Taking $\z=\y_1-\y_2$ in \eqref{eqn-differece}, we obtain 
	\begin{align}\label{eqn-differece-1}
			\langle\mathcal{F}(\y_1)-\mathcal{F}(\y_2),\y_1-\y_2\rangle&=-{}_{\U^{\prime}}\langle\boldsymbol{\eta}_1-\boldsymbol{\eta}_2,\ell(\y_1-\y_2)\rangle_{\U}.
	\end{align}
	Using Hypothesis \ref{hyp-psi-ell} (H3), we estimate right hand side of \eqref{eqn-differece-1} as 
	\begin{align}\label{eqn-differ}
		-{}_{\U^{\prime}}\langle\boldsymbol{\eta}_1-\boldsymbol{\eta}_2,\ell(\y_1-\y_2)\rangle_{\U}\leq m_1\|\ell(\y_1-\y_2)\|_{\U}^2\leq m_1\|\ell\|_{\mathcal{L}(\V;\U)}^2\|\y_1-\y_2\|_{\V}^2. 
	\end{align}
	Let us first consider the case $r>3$. A calculation similar to \eqref{eqn-bdes} and \eqref{eqn-des} yield 
	\begin{align}\label{eqn-differece-2}
	&	|	\langle\B(\y_1)-\B(\y_2),\y_1-\y_2\rangle|\nonumber\\&\leq\frac{(\mu-m_1\|\ell\|_{\mathcal{L}(\V;\U)}^2)}{2}\|\y_1-\y_2\|_{\V}^2+\frac{\beta}{2}\||\y_2|^{\frac{r-1}{2}}(\y_1-\y_2)\|_{\H}^2+\frac{\tilde{\varrho}}{2(\mu-m_1\|\ell\|_{\mathcal{L}(\V;\U)}^2)}\|\y_1-\y_2\|_{\H}^2.
	\end{align}
	 where $\varrho$ is defined in \eqref{eq-tilder}. Thus, by a calculation analogous to \eqref{eqn-final-est}, we obtain
	\begin{align}\label{eqn-differece-3}
		& \mu\|\y_1-\y_2\|_{\V}^2-\frac{(\mu-m_1\|\ell\|_{\mathcal{L}(\V;\U)}^2)}{2}\|\y_1-\y_2\|_{\V}^2+\alpha\|\y_1-\y_2\|_{\H}^2+\frac{\beta}{2}\||\y_1|^{\frac{r-1}{2}}(\y_1-\y_2)\|_{\H}^2\nonumber\\&\leq \langle\mathcal{F}(\y_1)-\mathcal{F}(\y_2),\y_1-\y_2\rangle+\frac{\tilde{\varrho}}{2(\mu-m_1\|\ell\|_{\mathcal{L}(\V;\U)}^2)}\|\y_1-\y_2\|_{\H}^2. 
		\end{align}
		Combining \eqref{eqn-differece-1}-\eqref{eqn-differece-3}, we deduce 
		\begin{align}\label{eqn-differece-4}
		\left(\frac{\mu-m_1\|\ell\|_{\mathcal{L}(\V;\U)}^2}{2}\right)\|\y_1-\y_2\|_{\V}^2+\left(\alpha-\frac{\tilde{\varrho}}{2(\mu-m_1\|\ell\|_{\mathcal{L}(\V;\U)}^2)}\right)\|\y_1-\y_2\|_{\H}^2\leq 0. 
		\end{align}
		The uniqueness follows by taking $\mu>\frac{\tilde{\varrho}}{2\alpha}+m_1\|\ell\|_{\mathcal{L}(\V;\U)}^2$. If one uses the estimate \eqref{2.26},  then the uniqueness follow for $\mu>\max\left\{\frac{1}{2\beta},\frac{1}{4\alpha}+m_1\|\ell\|_{\mathcal{L}(\V;\U)}^2\right\}$. 
		
		For $d=\{2,3\}$ and $r\in[1,3]$, a calculation similar to \eqref{eqn-diff} gives 
		\begin{align}\label{eqn-difference-5}
		&\frac{\mu}{2}\|\y_1-\y_2\|_{\V}^2-\frac{(\mu-m_1\|\ell\|_{\mathcal{L}(\V;\U)}^2)}{2}\|\y_1-\y_2\|_{\V}^2+\alpha\|\y_1-\y_2\|_{\H}^2+	\frac{\beta}{2}\||\y_1|^{\frac{r-1}{2}}(\y_1-\y_2)\|_{\H}^2
		\nonumber\\&\quad+\frac{\beta}{2}\||\y_2|^{\frac{r-1}{2}}(\y_1-\y_2)\|_{\H}^2 \nonumber\\&\quad\quad\leq 	\langle\F(\y_1)-\F(\y_2),\y_1-\y_2\rangle+ \frac{C}{(\mu-m_1\|\ell\|_{\mathcal{L}(\V;\U)}^2)^{\frac{4+d}{4-d}}} \|\y_2\|_{\L^4}^{\frac{8}{4-d}}\|\y_1-\y_2\|_{\H}^2. 
		\end{align}
Using \eqref{eqn-differece-1}, \eqref{eqn-differ}, \eqref{Eqn-mon-lip} and \eqref{eqn-ener-est} in \eqref{eqn-difference-5}, we arrive at 
		\begin{align}
		&\left(\frac{\mu-m_1\|\ell\|_{\mathcal{L}(\V;\U)}^2}{2}\right)\|\y_1-\y_2\|_{\V}^2+\frac{\beta}{2^{r-1}}\|\y_1-\y_2\|_{\L^{r+1}}^{r+1}\nonumber\\&\quad+\left(\alpha- \frac{C}{(\mu-m_1\|\ell\|_{\mathcal{L}(\V;\U)}^2)^{\frac{4+d}{4-d}}}\left(\frac{C_{\psi}\|\ell\|_{\mathcal{L}(\V;\U)}+\|\f\|_{\V^{\prime}}}{\mu-C_{\psi}\|\ell\|_{\mathcal{L}(\V;\U)}^2}\right)^{\frac{8}{4-d}}\right)\|\y_1-\y_2\|_{\H}^2\leq 0. 
		\end{align}
		Therefore, for $\mu>\left(\frac{C}{\alpha}\right)^{\frac{4-d}{4+d}}\left(\frac{C_{\psi}\|\ell\|_{\mathcal{L}(\V;\U)}+\|\f\|_{\V^{\prime}}}{\mu-C_{\psi}\|\ell\|_{\mathcal{L}(\V;\U)}^2}\right)^{\frac{8}{4+d}}+m_1\|\ell\|_{\mathcal{L}(\V;\U)}^2,$ and  the uniqueness follows. 
\end{proof}

\begin{remark}\label{rem-station-unique}
	For $d\in\{2,3\}$ and $r=3$, a calculation similar to \eqref{eqn-critical} provides 
	\begin{align*}
		 &\left(\mu-\frac{1}{2\beta}\right)\|\y_1-\y_2\|_{\V}^2+\alpha\|\y_1-\y_2\|_{\H}^2+ \frac{\beta}{2}\|\y_1(\y_1-\y_2)\|_{\H}^2 \leq 	\langle\F(\y_1)-\F(\y_2),\y_1-\y_2\rangle. 
	\end{align*}
	Combining the above inequality with \eqref{eqn-differece-1} and \eqref{eqn-differ}, we find 
	\begin{align*}
		&\left(\mu-\frac{1}{2\beta}-C_{\psi}\|\ell\|_{\mathcal{L}(\V;\U)}^2\right)\|\y_1-\y_2\|_{\V}^2+\alpha\|\y_1-\y_2\|_{\H}^2\leq 0,
	\end{align*}
	and the uniqueness follows provided $\mu>\frac{1}{2\beta}+C_{\psi}\|\ell\|_{\mathcal{L}(\V;\U)}^2$. 
\end{remark}

\subsection{A boundary hemivariational inequality}\label{sub-boundary}
Let us now provide a  boundary hemivariational inequality for stationary 2D and 3D CBF  equations as an example. We begin with the following stationary 2D and 3D CBF equations:
\begin{equation}\label{eqn-stationary}
	\left\{
	\begin{aligned}
		-\mu \Delta\y+(\y\cdot\nabla)\y+\alpha\y+\beta|\y|^{r-1}\y+\nabla p&=\boldsymbol{f}, \ \text{ in } \ \mathcal{O}, \\ \nabla\cdot\y&=0, \ \text{ in } \ \mathcal{O},
	\end{aligned}
	\right.
\end{equation}
where $\y(\cdot):\mathcal{O}\to\R^d$ denotes the flow velocity field,  $p(\cdot):\mathcal{O}\to\R$ is the pressure, and $\f (\cdot):\mathcal{O}\to\R^d$ is the density of external forces. 
 For $\alpha,\beta>0$, the system \eqref{eqn-stationary} can be considered as damped NSE. 
The second equation in \eqref{eqn-stationary} reflects the incompressibility constraint. We study a boundary hemivariational inequality for the system \eqref{eqn-stationary}. 

Let us first rewrite the first equation in \eqref{eqn-stationary} in terms of the curl operator (see \cite{VGPR} for its definition and properties). We have the following vector identities:
\begin{align}
	-\Delta\y&=\mathrm{curl\ }\mathrm{curl\ }\y-\nabla(\nabla\cdot\y)=\mathrm{curl\ }\mathrm{curl\ }\y,\label{eqn-vector-1}\\
	(\y\cdot\nabla)\y&=\mathrm{curl\ }\y\times\y +\frac12\nabla|\y|^2,\label{eqn-vector-2}
\end{align}
since $\nabla\cdot\y=0$. Using the above identities, one can re-write the problem \eqref{eqn-stationary} as 
\begin{equation}\label{eqn-stationary-1}
	\left\{
	\begin{aligned}
	\mu\  \mathrm{curl\ }\mathrm{curl\ }\y+\mathrm{curl\ }\y\times\y+\alpha\y+\beta|\y|^{r-1}\y+\nabla q&=\boldsymbol{f}, \ \text{ in } \ \mathcal{O}, \\ \nabla\cdot\y&=0, \ \text{ in } \ \mathcal{O},
	\end{aligned}
	\right.
\end{equation}
where $q(\cdot)=p(\cdot)+\frac12|\y(\cdot)|^2$ is the dynamic pressure. The CBF equations given in \eqref{eqn-stationary-1} are supplemented by  boundary conditions.  Let us now describe the associated boundary conditions. Let $\boldsymbol{n} = (n_1, \ldots, n_d)^{\top}$ denote the unit outward normal on the boundary $\Gamma$. For a vector field $\y$ defined on $\Gamma$, we write  
$$y_n = \y\cdot\boldsymbol{n},~ \y_{\tau}= \y - y_n\boldsymbol{n},$$ for its the normal and tangential components, respectively. We impose the following boundary conditions:
\begin{align}
	\y_{\tau}&=\boldsymbol{0}\ \text{ on }\ \Gamma,\label{eqn-bondary-1}\\
	q&\in\partial j(y_n)\ \text{ on }\ \Gamma,\label{eqn-bondary-2}\
\end{align}
where $j(y_n)$ is a short-hand notation for $j(\x,y_n(\x))$. The mapping  $j: \Gamma\times\mathbb{R}\to\mathbb{R}$ referred to as a superpotential and is assumed to be locally Lipschitz continuous with respect to its second argument. The operator $\partial j$ denotes the Clarke subdifferential of $j(\x,\cdot)$. The problem of fluid motion via a tube or channel gives rise to the boundary condition \eqref{eqn-bondary-2}. The fluid pumped into $\mathcal{O}$ can flow out through openings at the boundary, and a device can adjust the size of these openings. Here, we regulate the normal component of the fluid velocity along the boundary so as to minimize the total pressure on
$\Gamma$. Different physical phenomena are thus described by different boundary conditions. We point out here that if $ j(\x, y_n)$ is convex in its second argument,  then the problem \eqref{eqn-bondary-1}-\eqref{eqn-bondary-2} leads to a \emph{variational inequality}. In our context, we do not assume the convexity of $j$ with respect to its second variable; thus, our problem corresponds to a  \emph{hemivariational inequality}.
\subsubsection{Abstract formulation and main result} Multiplying the equation of motion \eqref{eqn-stationary-1} by $\z \in\V\cap\L^{r+1}$ and applying  Green's formula (see \eqref{eqn-green-1}), we obtain 
\begin{equation}\label{eqn-absract-1}
	\langle\mu\A\y+\B(\y)+\alpha\y+\beta\mathfrak{C}(\y),\z\rangle+\int_{\Gamma}qz_n\d\Gamma=\langle\f,\z\rangle,
\end{equation}
for all $\z\in\V\cap\L^{r+1}$. Using \eqref{eqn-bondary-2} together with the definition of the Clarke subdifferential, we deduce that
\begin{align}\label{eqn-clarke-sub}
	\int_{\Gamma}qz_n\d\Gamma\leq \int_{\Gamma}j^0(y_n;z_n)\d\Gamma,
\end{align}
where $j^0(\xi;\zeta)\equiv j^0(\x,\xi;\zeta)$ represents the generalized directional derivative of $j(\x,\cdot)$ at $\xi\in\R$ in the direction $\zeta\in\R$.

Let us take $\mathbb{U}:=\mathbb{L}^2(\Gamma)$, $\ell:\V\to\mathbb{U}$ and $\f\in\V^{\prime}$. The relations \eqref{eqn-absract-1} and \eqref{eqn-clarke-sub} yield the following variational formulation: 
\begin{problem}\label{prob-hemi-inequality}
	Find $\y\in\V\cap\L^{r+1}$ such that  
	\begin{equation}\label{eqn-absract-2}
			\langle\mu\A\y+\B(\y)+\alpha\y+\beta\mathfrak{C}(\y),\z\rangle+\int_{\Gamma}j^0(y_n;z_n)\d\Gamma\geq \langle\f,\z\rangle,
	\end{equation}
	for all $\z\in\V\cap\L^{r+1}$. 
\end{problem}
We make the following assumptions on the superpotential $j$:
\begin{hypothesis}\label{hyp-new-j}
	The superpotential $j:\Gamma\times\mathbb{R}\to\mathbb{R}$ satisfy 
	\begin{enumerate}
		\item [(H.1)] $j(\cdot,\xi)$ is measurable on $\Gamma$ for all $\xi\in\mathbb{R}$ and there exists an $e\in \mathbb{L}^2(\Gamma)$ such that $j(\cdot,e(\cdot))\in  \mathbb{L}^1(\Gamma)$;
		\item [(H.2)] for a.e. $\x \in \Gamma,~ j(\x,\cdot)$ is locally Lipschitz on $\mathbb{R}$;
		\item [(H.3)] for all $\zeta\in\partial j(\x,\xi) \text{ with }\xi\in\mathbb{R}$ for a.e. $\x\in\Gamma$ and $C_0>0$, it holds
		$$|\zeta|\leq C_0(1+|\xi|);$$
		\item [(H.4)] $(\zeta_1-\zeta_2)\cdot(\xi_1-\xi_2)\geq -m|\xi_1-\xi_2|^2$ for all $\zeta_i\in \partial j(\x,\xi_i)$, $\xi_i\in\mathbb{R}$, $i=1,2,$ for a.e. $\x\in\Gamma$ with $m\geq 0$. 
	\end{enumerate}
\end{hypothesis}
Let us define a functional $J:\mathbb{U}\to\mathbb{R}$ by 
\begin{align}\label{eqn-new-J}
	J(\y)=\int_{\Gamma}j(\x,y_n(\x))\d\Gamma,\ \y\in\mathbb{U}.
\end{align}

\begin{lemma}\label{lem-J-lem}
	Assume that $j:\Gamma\times\mathbb{R}\to\mathbb{R}$ satisfies Hypothesis \ref{hyp-new-j}. Then the functional $J$, as defined in \eqref{eqn-new-J}, satisfies the following properties:
	\begin{enumerate}
	\item [(i)] $J(\cdot)$ is locally Lipschitz on $\mathbb{U}$;
	\item [(ii)] $\|\boldsymbol{\eta}\|_{\mathbb{U}}\leq C_1(1+\|\y\|_{\mathbb{U}})$ for all $\boldsymbol{\eta}\in \partial J(\y)$, $\y\in\mathbb{U}$ with $C_1>0$, where $C_1=\sqrt{2}C_0\max\{\sqrt{\mathrm{meas}(\Gamma)},1\}$, $\mathrm{meas}(\Gamma)$ being the Lebesgue measure of $\Gamma$;
	\item [(iii)] $J^0(\y;\z)\leq\int_{\Gamma}j^0(y_n(\x);z_n(\x))\d\Gamma$, for all $\y,\z\in\mathbb{U}$;
	\item [(iv)] $(\u_1-\u_2,\y_1-\y_2)\geq -m\|\y_1-\y_2\|_{\mathbb{U}}^2$ for all $\u_i\in\partial J(\y_i)$, $\y_i\in\mathbb{U}$ for $i=1,2$. 
	\end{enumerate}
\end{lemma}
\begin{proof}
Let us define $\wi{j}:\Gamma\times\mathbb{R}^d\to\R$ by $\wi{j}(\x,\boldsymbol{\zeta})=j(\x,\zeta_n)$, for $(\x,\boldsymbol{\zeta})	\in\Gamma\times\R^d$. Then from \cite[Lemma 13]{SMAO}, we infer that $\wi{j}(\x,\boldsymbol{\zeta})=j(\x,L\boldsymbol{\zeta})$, where $L\in\mathcal{L}(\R^d;\R)$ is defined by $L\boldsymbol{\zeta}=\zeta_n=\boldsymbol{\zeta}\cdot\n$ and that $L^*\in\mathcal{L}(\R;\R^d)$ is given by $L^*r=r\n$ for $r\in\R$. Thus, based on \cite[Theorem 3.47]{SMAOMS}, we can deduce the conclusions (i)-(iii). By applying reasoning analogous to that in the proof of \cite[Theorem 4.20]{SMAOMS}, we obtain (iv).
\end{proof}

The Problem \ref{prob-hemi-inequality} can be  formulated as the following inclusion problem: 
\begin{problem}\label{prob-inclusion-neq}
	Find $\y\in\V\cap\L^{r+1}$ such that  in $\V^{\prime}+\L^{\frac{r+1}{r}},$ 
	\begin{equation}\label{eqn-abstract}
			\mu\A\y+\B(\y)+\alpha\y+\beta\mathfrak{C}(\y)+\ell^*\partial J(\ell\y)\ni\f, 
	\end{equation}
	where $\ell^*:\mathbb{U}'\to\V^{\prime}$ is the adjoint operator to $\ell.$ 
\end{problem}

\begin{remark}
If the functional $J$ defined as in \eqref{eqn-new-J} and Hypothesis \ref{hyp-new-j} is satisfied, it immediately follows that every solution to \eqref{eqn-abstract} is also a solution to the inequality \eqref{eqn-absract-2}. If in addition, $j$ or $-j$ is regular(see \cite[Definition 3.25]{SMAOMS}), then the converse is also true. Indeed, as stated in \cite[Theorem 3.47(vii)]{SMAOMS}, we have  for all $ \z\in\V\cap\L^{r+1}$ that 
	\begin{align*}
	\langle\f-[\mu\A\y+\B(\y)+\alpha\y+\beta\mathfrak{C}(\y)],\z \rangle\leq \int_{\Gamma}j^0(y_n;z_n)\d\Gamma=J^0(\ell\y;\ell\z).
	\end{align*}
	Using  \cite[Proposition 3.37(ii)]{SMAOMS}, we further have 
	\begin{align*}
		\f-[\mu\A\y+\B(\y)+\alpha\y+\beta\mathfrak{C}(\y)]\in\partial(J\circ\ell)(\y)=\ell^*\partial J(\ell\y),
	\end{align*}
	which implies \eqref{eqn-abstract}. 
\end{remark}
Let us now state and prove the existence of a weak solution to the Problem \ref{prob-inclusion-neq}. 
\begin{theorem}\label{thm-main-boundary}
	For  $r\geq 1$ and $\mu> C_{1}\|\ell\|_{\mathcal{L}(\V;\U)}^2,$ where $C_1=\sqrt{2}C_0\max\{\sqrt{\mathrm{meas}(\Gamma)},1\}$, under Hypothesis \ref{hyp-new-j} (H1)-(H3),  the Problem \ref{prob-inclusion-neq} has a weak solution. Furthermore, $\y$ satisfies the bounds in \eqref{eqn-ener-est} with $C_{\psi}$ replaced by $C_1$, and the values of $\mu$ given in \eqref{eqn-unique} and \eqref{eqn-unique-1} with $C_{\psi}$ replaced by $C_1$, the solution $\y$ to the Problem \ref{prob-inclusion-neq} is unique.


\end{theorem}
\begin{proof}
	Under Hypothesis \ref{hyp-new-j},  for $\psi(\cdot) := J(\cdot)$, we infer from  Lemma \ref{lem-J-lem}  that Hypothesis \ref{hyp-psi-ell} (H1)-(H3) is satisfied. 
	Note that the mapping $\ell:\V\to\U$ is linear, continuous and compact, so that Hypothesis \ref{hyp-psi-ell} (H4) is satisfied. Therefore, by applying Theorems \ref{thm-station-exis} and \ref{thm-station-unique}, we obtain the required result. 
\end{proof}

\section{Non-stationary problem: A boundary hemivariational inequality}\label{sec4}\setcounter{equation}{0}
Our aim in this section is to show the existence of a solution to a boundary hemivariational inequality for non-stationary 2D and 3D CBF equations.  

\subsection{Preliminaries}
For $T>0$ and a Banach space $\E$, the Bochner space $\mathrm{L}^p(0,T;\E),$ $1\leq p<\infty$ is defined as: 
\begin{align*}
	\mathrm{L}^p(0,T;\E):=\left\{\y:(0,T)\to\E: \y\ \text{ is strongly measurable and }\ \int_0^T\|\y(t)\|_{\E}^p\d t<\infty \right\}.
\end{align*}
For $p=\infty$, we define the set 
\begin{align*}
	\mathrm{L}^{\infty}(0,T;\E):=\left\{\y:(0,T)\to\E: \y\ \text{ is strongly measurable and }\ \esssup_{t\in(0,T)}\|\y(t)\|_{\E}<\infty \right\}.
\end{align*}
Let us define the set 
\begin{align}
	\mathcal{W}:=&\left\{\y\in\mathrm{L}^{\infty}(0,T;\H)\cap\mathrm{L}^2(0,T;\V)\cap\mathrm{L}^{r+1}(0,T;\L^{r+1})\big|\right.\nonumber\\&\qquad\left.\y'\in\mathrm{L}^2(0,T;\V^{\prime})+\mathrm{L}^{\frac{r+1}{r}}(0,T;\L^{\frac{r+1}{r}})\right\}.
\end{align}
From Theorem \ref{Thm-Abs-cont}, we infer that the space $\mathcal{W}$  is embedded continuously in $\C([0,T];\H)$, the space of all continuous functions $\y : [0,T] \to\H$  with the norm
\begin{align}
	\|\y\|_{\C([0,T];\H)}=\max_{0\leq t\leq T}\|\y(t)\|_{\H}. 
\end{align}
The following result is a generalization of the well-known Lions-Magenes Lemma \cite{JLEM}. Let $\mathcal{D}'(0,T;\E)$ denote the space of all distributions from $(0,T)$ to a Banach space $\E$. 
\begin{theorem}[{\cite[Theorem 1.8]{VVMI}}]\label{Thm-Abs-cont}
	Let $\H$ be a Hilbert space, and $\V, \X, \E$ be Banach spaces, satisfying the inclusions
	\[\V \hookrightarrow \H \hookrightarrow \V^\prime \hookrightarrow \E\  \mbox{ and } \ \X \hookrightarrow \H \hookrightarrow \X^\prime \hookrightarrow \E,\]
	where the spaces $\V^\prime$ and $\X^\prime$ are the duals of $\V$ and $\X$, respectively. Here the space $\H^\prime$ is identified with $\H$. Assume that $p>1$ and $\y\in \mathrm{L}^{2}(0,T;\V)\cap \mathrm{L}^p (0,T; \X)$,  $\y'\in \mathcal{D}^\prime(0,T; \E)$ and $\y^\prime = \y_1 +\y_2$, where $\y_1 \in \mathrm{L}^{2}(0,T;\V^\prime)$ and $\y_2 \in \mathrm{L}^{p^\prime}(0,T;\X^{\prime})$. Then,
	\begin{itemize}
		\item[(i)] $\y \in \C([0,T]; \H),$
		\item[(ii)] the function $[0,T] \ni t \mapsto\|\y(t)\|_{\H}^2 \in \mathbb{R}$ is absolutely continuous on $[0,T]$, and
		\begin{align}\label{eqn-absolute}
			\frac{\d}{\d t}\|\y(t)\|_{\H}^2= 2\,{}_{\E}\left\langle \y(t),\y'(t)\right\rangle_{\mathbb{E}^{\prime}}= 2\, {}_{\V}\left\langle \y(t),\y_1(t)\right\rangle_{\V^{\prime}} + 2{}_{\X}\left\langle \y(t),\y_2(t)\right\rangle_{\X^{\prime}},
		\end{align}
	\end{itemize}
	for a.e.  $t\in [0,T]$, that is,
	\begin{align*}
		\|\y(t)\|_{\H}^2
		& = \|\y(0)\|_{\H}^2 + 2\int_0^t [{}_{\V}\left\langle \y(s),\y_1(s)\right\rangle_{\V^{\prime}}+ {}_{\X}\left\langle \y(s),\y_2(s)\right\rangle_{\X^{\prime}}]\d s,
	\end{align*}
	for all $t\in[0,T]$. 
\end{theorem}

For a Banach space $\E$, let us denote by $\mathrm{BV}(0,T;\E)$, the space of functions of bounded total variation on $(0,T)$ defined as follows: Let $\uppi$ denote a finite partition of $[0,T]$: $ 0 =t_0 < t_1 <\cdots < t_n = T$, and let $\mathcal{P}$ be the family of all such partitions. Then we define the total variation of a function $\y: [0,T] \to\E$ as
\begin{align*}
	\|\y\|_{\mathrm{BV}(0,T;\E)}=\sup_{\uppi\in\mathcal{P}}\sum_{i=1}^n\|\y(t_i)-\y(t_{i-1})\|_{\E}.
\end{align*}
More generally, for $1\leq q<\infty$, we define analogously
\begin{align}\label{def-BV}
	\|\y\|_{\mathrm{BV}^q(0,T;\E)}^q=\sup_{\uppi\in\mathcal{P}}\sum_{i=1}^n\|\y(t_i)-\y(t_{i-1})\|_{\E}^q,
\end{align}
so that  the space $\mathrm{BV}^q(0,T;\E)$ consists of all the functions $\y:[0,T]\to\E$  such that $\|\y\|_{\mathrm{BV}^q(0,T;\E)}^q<\infty$. For Banach spaces $\E, \Z$ such that $\E\hookrightarrow\Z$, we introduce a vector space
\begin{align}\label{def-M}
	\mathcal{M}^{p,q}(0,T;\E,\Z)=\mathrm{L}^p(0,T;\E)\cap \mathrm{BV}^q(0,T;\Z),
\end{align}
which is a Banach space for $1\leq p,q <\infty$  with the norm given by $\|\cdot\|_{\mathrm{L}^p(0,T;\E)}+\|\cdot\|_{\mathrm{BV}^q(0,T;\Z)}$. The following result plays a key role in establishing the convergence of the Rothe method. The following theorem is a consequence of \cite[Theorem 1]{JSi}. 

\begin{theorem}[{\cite[Proposition 2]{PKa}, \cite[Theorem 5]{JSi}}]\label{thm-Simon}
	Let $1\leq p,q <\infty$  and $\E_1 \hookrightarrow \E_2 \hookrightarrow \E_3$ be real Banach spaces such that $\E_1$ is reflexive, the embedding $\E_1 \hookrightarrow \E_2$  is compact and the embedding $\E_2 \hookrightarrow \E_3$ is continuous. Then the following conditions hold: 
	\begin{enumerate}
		\item  If $\mathcal{H}$ is a bounded subset of $\mathcal{M}^{p,q}(0,T;\E_1,\E_3)$, then it is relatively compact in $\mathrm{L}^p(0,T;\E_2)$.
		\item  If $\mathcal{H}$ is a bounded subset of $\mathscr{W}^{p,q}(0,T;\E_1,\E_3):=\{\y\in\mathrm{L}^p(0,T;\E_1): \y^{\prime}\in\mathrm{L}^q(0,T;\E_3)\},$  then it is relatively compact in $\mathrm{L}^p(0,T;\E_2)$.
	\end{enumerate}
\end{theorem}

We will make use of the following version of the Aubin-Cellina convergence theorem in the sequel.
\begin{theorem}[{\cite[Theorem 1, Section 4, Chapter 1]{JPAAC}}]\label{thm-conv-inc}
	Let $\E$ and $\Y$ be Banach spaces. Assume $\mathcal{R}:\E\to 2^{\Y}$ is a multi-valued function such that
	\begin{enumerate}
		\item the values of $\mathcal{R}$  are nonempty, closed and convex subsets of $\Y$; \item  $\mathcal{R}$ is upper semicontinuous from $\E$ into $\Y_w$.
	\end{enumerate}
	Let $\y_n: (0,T)\to\E$, $\z_n: (0,T)\to\Y, n\in\N$, be measurable functions such that $\y_n$ converges a.e. on $(0,T)$ to a function $\y: (0,T) \to\E$ and $\z_n$ converges weakly in $\mathrm{L}^1(0,T;\Y)$ to $\z: (0,T)\to\Y$. If $\z_n(t)\in\mathcal{R}(\y_n(t))$ for all $n\in\mathbb{N}$ and a.a. $t\in (0,T)$, then $\z(t)\in\mathcal{R}(\y(t))$ for a.e. $t\in(0,T)$.
\end{theorem}

We denote by $\mathrm{C}_w([0,T];\E)$, the space of functions $\y : [0,T] \to\E$ which are weakly continuous, that is, for all $\boldsymbol{\phi}\in\E^{\prime}$, the scalar function
\begin{equation*}
	[0,T]\ni t\mapsto{}_{\E^{\prime}}\langle\boldsymbol{\phi},\y(t)\rangle_{\E}\in\R,
\end{equation*}
is continuous on $[0,T]$. The following result is originally due to Strauss (see \cite[Theorem 2.1]{WAS}, see \cite[Lemma 1.4, Chapter 3]{Te} also).
\begin{lemma}[Strauss lemma, {\cite[Proposition 1.7.1.]{PCAM}}]\label{lem-Strauss}
	Let $\E$ and $\Y$ be Banach spaces, with $\E$ reflexive and the embedding $\E\hookrightarrow \Y $ is continuous. Suppose that $\y\in \mathrm{L}^{\infty}(0,T;\E)\cap\mathrm{C}([0,T];\Y)$. Then, $\y\in\mathrm{C}_w([0,T];\E)$. 
\end{lemma}

\subsection{An abstract hemivariational inequality}\label{sub-abstract}We first consider the following inclusion problem:
\begin{problem}\label{prob-inclusion-3}
	Find $\y\in\mathcal{W}$ such that  in $\V^{\prime}+\L^{\frac{r+1}{r}},$ 
	\begin{equation}\label{eqn-abstract-3}
		\left\{
		\begin{aligned}
			&\y'(t)+\mu\A\y(t)+\B(\y(t))+\alpha\y(t)+\beta\mathfrak{C}(\y(t))+\ell^*\partial \psi(\ell\y(t))\ni\f(t), \  \mbox{ for  a.e. }\ t\in(0,T),\\
			&\y(0)=\y_0,
		\end{aligned}
		\right.
	\end{equation}
	where $\f\in\mathrm{L}^2(0,T;\V^{\prime})$, $\psi:\U\to\R$, $\partial\psi$ is the subdifferential of $\psi(\cdot)$ in the sense of Clarke and $\ell^{*}:\U^{\prime}\to\V^{\prime}$ is the adjoint operator to $\ell$.
\end{problem}
Let us now provide an equivalent formulation of the problem \eqref{prob-inclusion-3}.
\begin{problem}\label{prob-inclusion-4}
	Find $(\y,\boldsymbol{\eta})\in\mathcal{W}\times\mathrm{L}^2(0,T;\mathbb{U}^{\prime})$ such that   in $\V^{\prime}+\L^{\frac{r+1}{r}},$ 
	\begin{equation}\label{eqn-abstract-4}
		\left\{
		\begin{aligned}
			&\y'(t)+\mu\A\y(t)+\B(\y(t))+\alpha\y(t)+\beta\mathfrak{C}(\y(t))+\ell^*\boldsymbol{\eta}(t)=\f(t), \  \mbox{ for  a.e. }\ t\in(0,T),\\
			&\boldsymbol{\eta}(t)\in \partial \psi(\ell\y(t)), \  \mbox{ for  a.e. }\ t\in(0,T),\\
			&\y(0)=\y_0\in\H. 
		\end{aligned}
		\right.
	\end{equation}
\end{problem}
Note that the initial data given in \ref{eqn-abstract-4}  is satisfied in the weak sense, which is precisely explained in Definition \ref{defn-weak}. 
Let us now provide the definition of a weak solution to the Problem \ref{prob-inclusion-4}.
\begin{definition}\label{defn-weak}
	A function $\y\in\mathcal{W}$ is said to be \emph{weak solution} to the Problem  \ref{prob-inclusion-4}, if there exists $\boldsymbol{\eta}\in \mathrm{L}^2(0,T;\mathbb{L}^2(\Gamma))$ such that for all $\z\in \V\cap\L^{r+1},$
	\begin{align}
		\langle\y'(t)+\mu\A\y(t)+\B(\y(t))+\alpha\y(t)+\beta\mathfrak{C}(\y(t)),\z\rangle+{}_{\U^{\prime}}\langle\boldsymbol{\eta}(t),\ell\z\rangle_{\U}=\langle\f(t),\z\rangle, 
	\end{align}
for a.e. $t\in(0,T)$, $	\boldsymbol{\eta}(t)\in \partial \psi(\ell\y(t))$ for   a.e. $t\in(0,T)$ and 
\begin{align*}
\lim_{t\downarrow 0}(\y(t),\boldsymbol{\phi})=(\y_0,\boldsymbol{\phi}), \ \text{ for all }\ \boldsymbol{\phi}\in\H. 
\end{align*}
\end{definition}

For further  analysis, we need Hypothesis \ref{hyp-psi-ell} (H1)-(H3). We replace \ref{hyp-psi-ell} (H4) by the following condition:
\begin{hypothesis}\label{hyp-new-ell} The operator $\ell$ satisfies the following hypothesis: 
	\begin{enumerate}
		\item [(H4$'$)] The operator $\ell\in\mathcal{L}(\V;\U)$ is compact and
		\begin{itemize}
			\item  its Nemytskii operator $\overline{\ell}:\mathcal{M}^{2,2}(0,T;\V,\V^{\prime}) \to \mathcal{U}$ defined by $(\overline{\ell}\z)(t) = \ell\z(t)$ is compact (for $d\in\{2,3\}$ with $r\in[1,3]$),
			\item its Nemytskii operator $\overline{\ell}:\mathcal{W}\to \mathcal{U}$ defined by $(\overline{\ell}\z)(t) = \ell\z(t)$ is compact (for $d\in\{2,3\}$ with $r>3$).
		\end{itemize} 
		\end{enumerate}
\end{hypothesis}
The following result is crucial in establishing existence results for the Problem \ref{prob-inclusion-3}.
From Lemma \ref{lem-pseudo}, we know that the operator  $\mathcal{F}:\V\cap\L^{r+1}\to\V^{\prime}+\L^{\frac{r+1}{r}}$ defined in \eqref{eqn-op-f}  is pseudomonotone. 
Our aim is to show that its Nemytskii operator  is pseudomonotone.

Let us define the Nemytskii operators $$\mathscr{A},\mathscr{B},\mathscr{I}:\mathrm{L}^2(0,T;\V)\to\mathrm{L}^2(0,T;\V^{\prime})$$ $$\text{ and } \mathscr{C}:\mathrm{L}^{r+1}(0,T;\L^{r+1})\to\mathrm{L}^{\frac{r+1}{r}}(0,T;\L^{\frac{r+1}{r}})$$ by $(\mathscr{A}\z)(t)=\A\z(t)$, $(\mathscr{B}\z)(t)=\B(\z(t))$, $(\mathscr{I}\z)(t)=\I\z(t)=\z(t)$,  $(\mathscr{C}\z)(t)=\mathfrak{C}(\z(t))$ for $\z\in\mathrm{L}^{r+1}(0,T;\V\cap\L^{r+1})$, and $\overline{\ell}:\mathrm{L}^2(0,T;\V)\to \mathcal{U}$ by $(\overline{\ell}\z)(t)=\ell\z(t)$ for $\z\in\mathrm{L}^2(0,T;\V)$, where $\mathcal{U}=\mathrm{L}^2(0,T;\U)$.

The main ideas of the proof  given below has been borrowed from \cite[Theorem 2(b)]{JBVM},  \cite[Chapter I, Theorem 2.35]{Hu1},  \cite[Lemma 1]{PKa}, \cite[Proposition 1]{NSP},  \cite[Lemma 8.8]{TRo}. We consider the supercritical case only, that is,  the case $d\in\{2,3\}$ with $r>3$. 
\begin{proposition}\label{prop-pseudo-time}
	The Nemytskii operator 	$\mathscr{F}(\cdot):\mathrm{L}^2(0,T;\V)\cap\mathrm{L}^{r+1}(0,T;\L^{r+1})\to \mathrm{L}^2(0,T;\V^{\prime})+\mathrm{L}^{\frac{r+1}{r}}(0,T;{\L}^{\frac{r+1}{r}})$ defined by 
	\begin{align}\label{eqn-nem-f}
		\mathscr{F}(\y):=\mu\mathscr{A}\y+\mathscr{B}(\y)+\alpha\mathscr{I}\y+\beta\mathscr{C}(\y)
	\end{align}
	is pseudomonotone, that is, 
	\begin{enumerate}
		\item $\mathscr{F}$ is bounded,
		\item \begin{equation}\label{eqn-pseudo-time}
			\left\{
			\begin{aligned}
				&\y_m\xrightarrow{w^*}\y\ \text{ in }\ \mathcal{W},\\
				&	\limsup_{m\to\infty}\langle\mathscr{F}(\y_m),\y_m-\y\rangle \leq 0,
			\end{aligned}
			\right. \implies
			\left\{
			\begin{aligned}
				&\text{ for all }\ \z\in \mathcal{W},\\
				&\langle\mathscr{F}(\y),\y-\z\rangle\leq \liminf_{m\to\infty}\langle\mathscr{F}(\y_m),\y_m-\y\rangle.
			\end{aligned}
			\right.
		\end{equation}
	\end{enumerate}
\end{proposition}
\begin{proof}
	Let us first show the boundedness of the operator $\mathscr{F}$. Using \eqref{eqn-b-est-lr} and H\"older's inequality, we estimate 
	\begin{align}\label{eqn-bes-time}
	&	|\langle\mathscr{B}(\y),\z\rangle|\nonumber\\&\leq\int_0^T\|\y(t)\|_{\V}\|\y(t)\|_{\H}^{\frac{r-3}{r-1}}\|\y(t)\|_{\L^{r+1}}^{\frac{2}{r-1}}\|\z(t)\|_{\L^{r+1}}\d t
	\nonumber\\&\leq 
	T^{\frac{r-3}{2(r-1)}} \big(\sup_{t\in[0,T]}\|\y(t)\|_{\H}^2\big) \bigg(\int_0^T\|\y(t)\|_{\V}^2\d t\bigg)^{\frac{1}{2}}\bigg(\int_0^T\|\y(t)\|_{\L^{r+1}}^{r+1}\d t\bigg)^{\frac{2}{(r-1)(r+1)}}\bigg(\int_0^T\|\z(t)\|_{\L^{r+1}}^{r+1}\d t\bigg)^{\frac{1}{r+1}}. 
	\end{align}
	For all $\y,\z\in\mathrm{L}^2(0,T;\V)\cap\mathrm{L}^{r+1}(0,T;\L^{r+1})$, using H\"older's inequality and \eqref{eqn-bes-time},  we have 
	\begin{align}\label{eqn-bounded}
		|\langle\mathscr{F}(\y),\z\rangle|&=|\langle\mu\mathscr{A}\y+\mathscr{B}(\y)+\alpha\mathscr{I}\y+\beta\mathscr{C}(\y),\z\rangle|\nonumber\\&\leq\mu\|\y\|_{\mathrm{L}^2(0,T;\V)}\|\z\|_{\mathrm{L}^2(0,T;\V)}+\alpha\|\y\|_{\mathrm{L}^2(0,T;\H)}\|\z\|_{\mathrm{L}^2(0,T;\H)}\nonumber\\&\quad+T^{\frac{r-3}{2(r-1)}}\|\y\|_{\mathrm{L}^{\infty}(0,T;\H)}^2\|\y\|_{\mathrm{L}^2(0,T;\V)}\|\y\|_{\mathrm{L}^{r+1}(0,T;\L^{r+1})}^{\frac{2}{r-1}}\|\z\|_{\mathrm{L}^{r+1}(0,T;\L^{r+1})}\nonumber\\&\quad+\beta\|\y\|_{\mathrm{L}^{r+1}(0,T;\L^{r+1})}^r
		\|\z\|_{\mathrm{L}^{r+1}(0,T;\L^{r+1})},
	\end{align}
	and the boundedness of the map $\mathscr{F} $ follows. In fact, the estimate \eqref{eqn-bounded} holds true for all $\z\in\mathrm{L}^{r+1}(0,T;\V\cap\L^{r+1})$. Therefore, one can say that the map $\mathscr{F}:\mathrm{L}^2(0,T;\V)\cap\mathrm{L}^{r+1}(0,T;\L^{r+1})\to \mathrm{L}^{\frac{r+1}{r}}(0,T;\V^{\prime}+\L^{\frac{r+1}{r}})$  is also bounded. 
	
	In order to prove \eqref{eqn-pseudo-time}, we take $\y_m\xrightarrow{w^*}\y\ \text{ in }\ \mathcal{W}$.  By Helly’s selection principle for functions taking values in the separable reflexive Banach space $\V^{\prime}+\L^{\frac{r+1}{r}}$ (\cite[Theorem 1.126]{VBTP}), there is a subsequence (denoted by the same symbol) and $\wi\y :[0,T]\to\V^{\prime}+\L^{\frac{r+1}{r}}$ with a bounded variation such that $\y_m(t)\xrightarrow{w}\wi\y(t)$ in $\V^{\prime}+\L^{\frac{r+1}{r}}$ for all $t\in[0,T]$. It can be easily seen that $\y(t)=\wi\y(t)$ for  a.e. $t\in[0,T]$, since $\y_m\xrightarrow{w^*}\y\ \text{ in }\ \mathcal{W}$,  for any $\z\in\mathrm{L}^{\infty}(0,T;\V^{\prime}+\L^{\frac{r+1}{r}}),$we have $\langle\y_m,\z\rangle\to \langle\y,\z\rangle$ and, by the Lebesgue Dominated Convergence Theorem, we further infer $\langle\y_m,\z\rangle\to \langle\wi\y,\z\rangle$. Using  the boundedness of  the sequence $\{\y_m(t)\}_{m\in\N}$ in $\H$ for a.e. $t\in[0,T]$, we also have  $$\y_m(t)\xrightarrow{w}\y(t)\ \text{ in }\ \H\ \text{ for a.e. }\ t\in[0,T].$$ 	Let $\mathcal{N}$ be the set of points where the convergence does not hold. We define $$\xi_m(t)=\langle\mathcal{F}(\y_m(t)),\y_m(t)-\y(t)\rangle.$$ It can be easily seen that 
	\begin{align}\label{eqn-pm-1}
		\xi_m(t)&=\langle\mathcal{F}(\y_m(t)),\y_m(t)\rangle-\langle\mathcal{F}(\y_m(t)),\y(t)\rangle\nonumber\\&= \mu\|\y_m(t)\|_{\V}^2+\alpha\|\y_m(t)\|_{\H}^2+\beta\|\y_m(t)\|_{\L^{r+1}}^{r+1}-\langle\mathcal{F}(\y_m(t)),\y(t)\rangle.
	\end{align}
	Let us now consider the term $\langle\mathcal{F}(\y_m(t)),\y(t)\rangle$ and estimate it using the Cauchy-Schwarz, H\"older's  and Young's inequalities, the estimate \eqref{eqn-b-est-lr} as 
	\begin{align}\label{eqn-pm-2}
		&|\langle\mathcal{F}(\y_m(t)),\y(t)\rangle|
		\nonumber\\&\leq
		\mu\|\y_m(t)\|_{\V}\|\y(t)\|_{\V}+\|\y_m(t)\|_{\V}\|\y_m(t)\|_{\H}^{\frac{r-3}{r-1}}\|\y_m(t)\|_{\L^{r+1}}^{\frac{2}{r-1}}\|\y(t)\|_{\L^{r+1}}
		\nonumber\\&\quad+
		\alpha\|\y_m(t)\|_{\H}\|\y(t)\|_{\H}+\beta\|\y_m(t)\|_{\L^{r+1}}^r\|\y(t)\|_{\L^{r+1}}
		\nonumber\\&\leq
		\frac{\mu}{4}\|\y_m(t)\|_{\V}^2+\mu\|\y(t)\|_{\V}^2+\frac{\mu}{4}\|\y_m(t)\|_{\V}^2
		+\frac{1}{\mu}\underbrace{\|\y_m(t)\|_{\H}^{\frac{2(r-3)}{r-1}}\|\y_m(t)\|_{\L^{r+1}}^{\frac{4}{r-1}}\|\y(t)\|_{\L^{r+1}}^2}_{\text{Young's inequality with exponents $\frac{r-1}{r-3}$ and $\frac{r-1}{2}$}}
		\nonumber\\&\quad+
		\frac{\alpha}{4}\|\y_m(t)\|_{\H}^2+\alpha\|\y(t)\|_{\H}^2+\frac{\beta}{4}\|\y_m(t)\|_{\L^{r+1}}^{r+1}+
		\frac{\beta}{r+1}\left(\frac{4r}{r+1}\right)^{r}
		\|\y(t)\|_{\L^{r+1}}^{r+1}
		\nonumber\\&\leq \frac{\mu}{2}\|\y_m(t)\|_{\V}^2+\frac{\alpha}{2}\|\y_m(t)\|_{\H}^2+\frac{\beta}{4}\|\y_m(t)\|_{\L^{r+1}}^{r+1}	\nonumber\\&\quad+
		\left(\frac{4(r-3)}{\alpha\mu(r-1)}\right)^{\frac{r-3}{2}}
		\frac{2}{\mu(r-1)}\|\y_m(t)\|_{\L^{r+1}}^2
		\|\y(t)\|_{\L^{r+1}}^{r-1}
	+
		\mu\|\y(t)\|_{\V}^2+\alpha\|\y(t)\|_{\H}^2
		\nonumber\\&\leq
		\frac{\mu}{2}\|\y_m(t)\|_{\V}^2+\frac{\alpha}{2}\|\y_m(t)\|_{\H}^2+\frac{\beta}{2}\|\y_m(t)\|_{\L^{r+1}}^{r+1}+
		\mu\|\y(t)\|_{\V}^2+\alpha\|\y(t)\|_{\H}^2
		\nonumber\\&\quad+
		C(\beta,\mu,\alpha,r)\|\y(t)\|_{\L^{r+1}}^{r+1},
	\end{align}
	where  
	\begin{align*}
	C(\beta,\mu,\alpha,r)=
	\bigg[\frac{16}{\beta\mu(r^2-1)}\left(\frac{4(r-3)}{\alpha\mu(r-1)}\right)^{\frac{r-3}{2}}\bigg]^{\frac{2}{r-1}}\frac{r-1}{r+1},
	\end{align*}
 is a constant depending on $\alpha,\beta,\mu$ and $r$ only. Using \eqref{eqn-pm-2} in \eqref{eqn-pm-1}, we deduce 
	\begin{align}\label{eqn-pm-3}
	\xi_m(t)&\geq\frac{\mu}{2}\|\y_m(t)\|_{\V}^2+\frac{\alpha}{2}\|\y_m(t)\|_{\H}^2+\frac{\beta}{2}\|\y_m(t)\|_{\L^{r+1}}^{r+1}
	\nonumber\\&\quad
	-\left(\frac{1}{\mu}\|\y(t)\|_{\V}^2+\frac{1}{2\alpha}\|\y(t)\|_{\H}^2+C(\beta,\mu,\alpha,r)\|\y(t)\|_{\L^{r+1}}^{r+1}\right)
	\nonumber\\&\geq \frac{\mu}{2}\|\y_m(t)\|_{\V}^2+\frac{\beta}{2}\|\y_m(t)\|_{\L^{r+1}}^{r+1}-g(t),
	\end{align}
	where $g(t)=\left(\frac{1}{\mu}\|\y(t)\|_{\V}^2+\frac{1}{2\alpha}\|\y(t)\|_{\H}^2+C(\beta,\mu,\alpha,r)\|\y(t)\|_{\L^{r+1}}^{r+1}\right)$. 	Let $\mathcal{M}:=\Big\{t\in[0,T]:\liminf\limits_{m\to\infty}\xi_m(t)<0\Big\},$ which is a Lebesgue measurable subset of $[0,T]$. Let us assume that $\lambda(\mathcal{M})>0$, where $\lambda$ is the one-dimensional Lebesgue measure. For $t\in\mathcal{M}\backslash\mathcal{N}$, the sequence $\xi_m(t)$ has a subsequence (still denoted by the same symbol) such that $\lim\limits_{m\to\infty}\xi_m(t)<0$. 
	 Therefore, from \eqref{eqn-pm-3}, we infer that $\|\y_m(t)\|_{\V}^2+\|\y_m(t)\|_{\L^{r+1}}^{r+1}$ is bounded independent of $m$. Hence, an application of the Banach-Alaoglu theorem yields $\y_m(t)\xrightarrow{w}\y(t)$ in $\V\cap\L^{r+1}$, where the limit equals $\y(t)$ since we can consider only $t\not\in\mathcal{N}$.  Since the operator  $\mathcal{F}:\V\cap\L^{r+1}\to\V^{\prime}+\L^{\frac{r+1}{r}}$ defined in \eqref{eqn-op-f}  is pseudomonotone (Lemma \ref{lem-pseudo}), we immediately have $0\leq\liminf\limits_{m\to\infty}\langle\mathcal{F}(\y_m(t)),\y_m(t)-\y(t)\rangle,$ which is a contradiction. This  means $\lambda(\mathcal{M})=0$ so that $\liminf\limits_{m\to\infty}\xi_m(t)\geq 0$ for a.e. $t\in(0,T)$. We infer from the Fatou Lemma that 
	\begin{align}
		0\leq \int_0^T\liminf\limits_{m\to\infty}\xi_m(t)\d t\leq \liminf\limits_{m\to\infty}\int_0^T\xi_m(t)\d t\leq  \limsup\limits_{m\to\infty}\int_0^T\xi_m(t)\d t\leq 0, 
	\end{align}

	by using the assumption in \eqref{eqn-pseudo-time}.  Therefore, we deduce $\lim\limits_{m\to\infty}\int_0^T\xi_m(t)\d t=0$ as $m\to\infty$.  Since $\liminf\limits_{m\to\infty}\xi_m(t)\geq 0$ for a.e. $t\in(0,T)$, we have $\xi_m^{-}(t)\to 0$ for a.e. $t\in(0,T)$, where $\xi_m^{-}=\max\{-\xi_m,0\}\geq 0$. Note that $|\xi_m(t)|=\xi_m(t)+2\xi_m^{-}(t)$ and $\xi_m^{-}(t)\to 0$ for a.e. $t\in(0,T)$. From \eqref{eqn-pm-3}, we infer that $	\xi_m(t)\geq -g(t),$ with $g\in\mathrm{L}^1(0,T)$, so that $\xi_m^{-}(t)\leq g(t).$ Invoking the Vitali convergence theorem (\cite[Theorem 1.17]{TRo}), we obtain  $\int_0^T\xi_m^{-}(t)\d t\to 0$ as $m\to\infty$. Therefore, it is immediate that 
	\begin{align*}
		\lim_{m\to\infty}\int_0^T	|\xi_m(t)|\d t=\lim_{m\to\infty}\int_0^T\xi_m(t)\d t+2\lim_{m\to\infty}\int_0^T\xi_m^{-}(t)\d t=0. 
	\end{align*}
	Along a further subsequence (labeled by the same symbol), we have $\xi_m(t)\to 0$ as $m\to\infty$ for a.e. $t\in(0,T)$. Since,  for this subsequence, we have  $\y_m(t)\xrightarrow{w}\y(t)$ in $\V\cap\L^{r+1}$ and the pseudomonotonicity of the operator $\mathcal{F}$ implies 
	\begin{align*}
		\liminf_{m\to\infty}\langle\mathcal{F}(\y_m(t)),\y_m(t)-\z(t)\rangle\geq \langle\mathcal{F}(\y(t)),\y(t)-\z(t)\rangle\ \text{ for a.e. }\ t\in(0,T). 
	\end{align*}
	Once again an application  of the  Fatou Lemma yields 
	\begin{align*}
		\liminf_{m\to\infty}	\langle\mathscr{F}(\y_m),\y_m-\z\rangle&=	\liminf_{m\to\infty}\int_0^T	\langle\mathcal{F}(\y_m(t)),\y_m(t)-\z(t)\rangle\d t\nonumber\\&\geq \int_0^T	\liminf_{m\to\infty}	\langle\mathcal{F}(\y_m(t)),\y_m(t)-\z(t)\rangle\d t\nonumber\\&\geq \int_0^T\langle\mathcal{F}(\y(t)),\y(t)-\z(t)\rangle\d t=	\langle\mathscr{F}(\y),\y-\z\rangle,
	\end{align*}
	and the required result follows. 
\end{proof}

\begin{remark}
Establishing an estimate similar to \eqref{eqn-pm-2} is difficult in the case 	for $d\in\{2,3\}$ with $1\leq r\leq 3$, the results obtained in Proposition \ref{prop-pseudo-time} may not hold true for this case. 
\end{remark}

\subsubsection{Existence of approximate solution} 
Let us now prove the existence of a \emph{weak solution} for the Problem \ref{prob-inclusion-4}.  We follow the works \cite{Fang2016,PKa,TRo}, to obtain the existence results. Theorem \ref{thm-main-non-station} given below offers a constructive framework for establishing the existence of solutions to the Problem \ref{prob-inclusion-4}. The key idea is to approximate the time derivative using the backward difference scheme and to solve the resulting elliptic problem at each time step, thereby obtaining the solution at successive points in the time mesh. Notably, this method requires neither smoothing assumptions nor additional regularity conditions, provided that the underlying elliptic problems can be solved. The Rothe method has been effectively applied to various classes of partial differential equations, including nonlinear PDEs \cite{TRo}, variational inequalities \cite{Boc1981}, and hemivariational inequalities \cite{Fang2016}.

We begin by introducing a temporal semi-discrete approximation of Problem \ref{prob-inclusion-4} using the backward Euler scheme, a method commonly referred to as the \emph{Rothe method}. For a fixed $N\in\mathbb{N}$, define the time step-size $k = \frac{T}{N}$. Let us now introduce the piecewise constant interpolant of $\f$ by
\begin{align}
	\f_{k,i}=\frac{1}{k}\int_{(i-1)k}^{ik}\f(t)\d t, \ i=1,2,\ldots,N. 
\end{align}
We approximate the initial condition using elements from the space $\V\cap\L^{r+1}$. Specifically, let  $\{\y_{k,0}\}_{k\in\N}$ be a sequence such that  
\begin{align}\label{approx-con.}
\y_{k,0}\to\y_0 \text{ strongly in } \H,
\end{align}
and 
\begin{align}\label{approx-bound}
\|\y_{k,0}\|_{\V\cap\L^{r+1}}\leq\frac{C}{\sqrt{k}} \text{ for some constant } C>0.
\end{align}
  By the density of  $\V\cap\L^{r+1}$ in $\H$ (see \cite[Theorem 4.2.]{FHR} for such a construction of functions), the existence of such a sequence is guaranteed (\cite[Theorem 8.9]{TRo}). We now consider the following Rothe scheme as a temporal discretization for approximating Problem \ref{prob-inclusion-4}. 

\begin{problem}\label{prob-rothe}
	Find $\{\y_{k,i}\}_{i=0}^N\subset\V\cap\L^{r+1}$ and $\{\boldsymbol{\eta}_{k,i}\}_{i=0}^N\subset\U^{\prime}$ such that for $i=1,2,\ldots,N,$  and for all $\z\in\V\cap\L^{r+1}$
	\begin{equation}\label{eqn-rothe}
		\left\{
		\begin{aligned}
			&\frac{1}{k}(\y_{k,i}-\y_{k,i-1},\z)+\langle\mu\A\y_{k,i}+\B(\y_{k,i})+\alpha\y_{k,i}+\beta\mathfrak{C}(\y_{k,i}),\z\rangle+{}_{\U^{\prime}}\langle\boldsymbol{\eta}_{k,i},\ell\z\rangle_{\U}=\langle\f_{k,i},\z\rangle,\\
			&\boldsymbol{\eta}_{k,i}\in\partial\psi(\ell\y_{k,i}). 
		\end{aligned}
		\right.
	\end{equation}
\end{problem}
 Let us first show an existence result for the Problem \ref{prob-rothe}. 
\begin{theorem}\label{thm-rothe-existence}
	For  $r\geq 1$ and $\mu> C_{\psi}\|\ell\|_{\mathcal{L}(\V;\U)}^2,$ under Hypothesis \ref{hyp-psi-ell} (H1), (H2) and Hypothesis \ref{hyp-new-ell} (H4$'$),  there exists a solution to the Problem \ref{prob-rothe}. 
\end{theorem}
\begin{proof}
	It suffices to show that, for a given $\y_{k,i-1}\in\V\cap\L^{r+1}$, there exist $\y_{k,i}\in\V\cap\L^{r+1}$ and  $\boldsymbol{\eta}_{k,i}\in\U^{\prime}$ 	satisfying equation \eqref{eqn-rothe}. Let us define a multi-valued operator $\mathcal{H}:\V\cap\L^{r+1}\to 2^{\V^{\prime}+\L^{\frac{r+1}{r}}}$ by 
	\begin{align}\label{eqn-hop}
		\mathcal{H}(\z)=\frac{i^*i}{k}\z+\mu\A\z+\B(\z)+\alpha\z+\beta\mathfrak{C}(\z)+\ell^*\partial\psi(\ell\z),\ \z\in\V\cap\L^{r+1}. 
	\end{align}
	Then the problem \eqref{eqn-rothe} is equivalent to
	\begin{align}
		\mathcal{H}(\y_{k,i})\ni\f_{k,i}+\frac{i^*i}{k}\y_{k,i-1}. 
	\end{align}
	In order to prove the existence of a solution to \eqref{eqn-rothe}, it suffices to show the surjectivity of the operator defined in \eqref{eqn-hop}. We infer from Theorem \ref{thm-surjective} that the pseudomonotonicity and coercivity of $\mathcal{H}$ provide the required result. 
	\vskip 2mm
	\noindent
	\textbf{Step 1:} \emph{Coercivity of $\mathcal{H}$.} Let us first prove the coercivity of $\mathcal{H}$. For $\z\in\V\cap\L^{r+1}$ and $\z^*\in \mathcal{H}(\z)$, we have 
	\begin{align}
		\z^*=\frac{i^*i}{k}\z+\mu\A\z+\B(\z)+\alpha\z+\beta\mathfrak{C}(\z)+\ell^*\boldsymbol{\eta},
		\end{align}
		where $ \boldsymbol{\eta}\in \partial\psi(\ell\z)$. Then, by using a calculation similar to \eqref{eqn-coer-3}, we immediately get 
		\begin{align}
			\langle\z^*,\z\rangle&=\left(\frac{i^*i}{k}\z,\z\right)+\langle\mu\A\z+\B(\z)+\alpha\z+\beta\mathfrak{C}(\z),\z\rangle+{}_{\U^{\prime}}\langle\boldsymbol{\eta},\ell\z\rangle_{\U}\nonumber\\&\geq\wi\mu\|\z\|_{\V}^2+\left(\alpha+\frac{1}{k}\right)\|\z\|_{\H}^2+\beta\|\z\|_{\L^{r+1}}^{r+1}-C_{\psi}\|\ell\|_{\mathcal{L}(\V;\U)}\|\z\|_{\V},
		\end{align}
		where 
		\begin{align}\label{eqn-wi-mu}
			\wi\mu=\left(\mu-C_{\psi}\|\ell\|_{\mathcal{L}(\V;\U)}^2\right). 
		\end{align}
			A calculation similar to \eqref{eqn-coercive-1} provided 
		\begin{align}
			\frac{\langle\z^*,\z\rangle}{\|\z\|_{\V\cap\L^{r+1}}}\geq\frac{\min\left\{\wi\mu,\beta\right\}\left(\|\z\|_{\V}^2+\|\z\|_{\L^{r+1}}^2-1\right)-C_{\psi}\|\ell\|_{\mathcal{L}(\V;\U)}\|\z\|_{\V}}{\sqrt{\|\z\|_{\V}^2+\|\z\|_{\L^{r+1}}^{2}}},
		\end{align} 
		Therefore for $\mu> C_{\psi}\|\ell\|_{\mathcal{L}(\V;\U)}^2$, we it follows that 
		\begin{align*}
			\lim\limits_{\|\z\|_{\V\cap\L^{r+1}}\to\infty}	\frac{\langle\z^{*},\z\rangle}{\|\z\|_{\V\cap\L^{r+1}}}=\infty,
		\end{align*}
so that the operator $\mathcal{H}:\V\cap\L^{r+1}\to 2^{\V^{\prime}+\L^{\frac{r+1}{r}}}$ is coercive. 
		
			\vskip 2mm
		\noindent
		\textbf{Step 2:} \emph{Pseudomonotonicity of $\mathcal{H}$.} Let us now establish the pseudomonotonicity of the operator $\mathcal{H}$. In light of Proposition \ref{prop-suff-pseudo}, we conclude that the operator $\frac{i^*i}{k}$ is pseudomonotone from $\V\cap\L^{r+1}$ to $\V^{\prime}+\L^{\frac{r+1}{r}}$. Since the operator $\ell:\V\to\U$  is compact, a discussion similar to the proof of Lemma \ref{lem-pseudo-mon} yields $\ell^*\partial\psi(\ell\cdot)$ is  pseudomonotone. We infer from  Lemma \ref{lem-pseudo}  that the operator $\mathcal{F}(\cdot):=\mu\A+\B(\cdot)+\alpha\I+\beta\mathfrak{C}(\cdot)$ is pseudomonotone. Since the sum of two pseudomonotone operators is again pseudomonotone (see \cite[Proposition 1.3.68]{ZdSmP}), it follows that $\mathcal{H}$ is pseudomonotone. Finally an application of Theorem \ref{thm-surjective} yields the existence of a solution to the Problem \ref{prob-rothe}. 
\end{proof}
Our next aim is to establish uniform energy estimates  for the solutions of the Rothe problem \eqref{prob-rothe}. 
\begin{lemma}[Estimates for approximate solution]\label{lem-uniform-ener}
	Under the assumptions of Theorem \ref{thm-rothe-existence}, there exists a constant $M_1>0$, which is independent of $k$ such that 
	\begin{align}\label{eqn-uni-bound}
		\max\limits_{1\leq i\leq N}\|\y_{k,i}\|_{\H}+\sum_{i=1}^N\|\y_{k,i}-\y_{k,i-1}\|_{\H}^2+k\sum_{i=1}^N\|\y_{k,i}\|_{\V}^2+k\sum_{i=1}^N\|\y_{k,i}\|_{\L^{r+1}}^{r+1}\leq M_1. 
	\end{align}
\end{lemma}
\begin{proof}
	Let us take $\z=\y_{k,i}$ in \eqref{eqn-rothe}, we find 
	\begin{align*}
	&	\frac{1}{k}(\y_{k,i}-\y_{k,i-1},\y_{k,i})+\langle\mu\A\y_{k,i}+\B(\y_{k,i})+\alpha\y_{k,i}+\beta\mathfrak{C}(\y_{k,i}),\y_{k,i}\rangle\nonumber\\&+{}_{\U^{\prime}}\langle\boldsymbol{\eta}_{k,i},\ell\y_{k,i}\rangle_{\U}=\langle\f_{k,i},\y_{k,i}\rangle,
	\end{align*}
	so that
	\begin{align}\label{eqn-ener-rothe}
	&	\frac{1}{k}(\y_{k,i}-\y_{k,i-1},\y_{k,i})+\mu\|\y_{k,i}\|_{\V}^2+\alpha\|\y_{k,i}\|_{\H}^2+\beta\|\y_{k,i}\|_{\L^{r+1}}^{r+1}+{}_{\U^{\prime}}\langle\boldsymbol{\eta}_{k,i},\ell\y_{k,i}\rangle_{\U}=\langle\f_{k,i},\y_{k,i}\rangle. 
	\end{align}
	But we know that 
	\begin{align}\label{eqn-ener-rothe-1}
		(\y_{k,i}-\y_{k,i-1},\y_{k,i})=\frac{1}{2}\|\y_{k,i}\|_{\H}^2-\frac{1}{2}\|\y_{k,i-1}\|_{\H}^2+\frac{1}{2}\|\y_{k,i}-\y_{k,i-1}\|_{\H}^2
	\end{align}
	Using the Cauchy-Schwartz and Young's inequalities, we estimate $|\langle\f,\y_{k,i}\rangle|$ as 
	\begin{align}\label{eqn-ener-rothe-2}
		|\langle\f_{k,i},\y_{k,i}\rangle|\leq\|\f_{k,i}\|_{\V^{\prime}}\|\y_{k,i}\|_{\V}\leq\frac{\wi\mu}{4}\|\y_{k,i}\|_{\V}^2+\frac{1}{\wi\mu}\|\f_{k,i}\|_{\V^{\prime}}^2,
	\end{align}
where $\wi\mu$ is defined in \eqref{eqn-wi-mu}. 	A calculation similar to \eqref{eqn-coer-2} yields 
	\begin{align}\label{eqn-ener-rothe-3}
			|{}_{\U^{\prime}}\langle\boldsymbol{\eta},\ell\y_{k,i}\rangle_{\U}|&\leq C_{\psi}\|\ell\|_{\mathcal{L}(\V;\U)}^2\|\y_{k,i}\|_{\V}^2+C_{\psi}\|\ell\|_{\mathcal{L}(\V;\U)}\|\y_{k,i}\|_{\V}\nonumber\\&\leq C_{\psi}\|\ell\|_{\mathcal{L}(\V;\U)}^2\|\y_{k,i}\|_{\V}^2+\frac{\wi\mu}{4}\|\y_{k,i}\|_{\V}^2+\frac{1}{\wi\mu} C_{\psi}^2\|\ell\|_{\mathcal{L}(\V;\U)}^2,
	\end{align}
	where we have used Young's inequality also. Combining \eqref{eqn-ener-rothe-1}-\eqref{eqn-ener-rothe-3} and substituting it in \eqref{eqn-ener-rothe}, we deduce 
	\begin{align}\label{eqn-ener-rothe-4}
	&	\|\y_{k,i}\|_{\H}^2-\|\y_{k,i-1}\|_{\H}^2+\|\y_{k,i}-\y_{k,i-1}\|_{\H}^2+k\wi\mu\|\y_{k,i}\|_{\V}^2+2\alpha k\|\y_{k,i}\|_{\H}^2+2\beta k\|\y_{k,i}\|_{\L^{r+1}}^{r+1}\nonumber\\&\leq \frac{2k}{\wi\mu} \left[C_{\psi}^2\|\ell\|_{\mathcal{L}(\V;\U)}^2+\|\f_{k,i}\|_{\V^{\prime}}^2\right]. 
	\end{align}
	For $1\leq n\leq N$, let us sum the inequality \eqref{eqn-ener-rothe-4} for $i=1,2,\ldots, n$ to obtain 
	\begin{align}\label{eqn-ener-rothe-5}
		&	\|\y_{k,n}\|_{\H}^2+\sum_{i=1}^n\|\y_{k,i}-\y_{k,i-1}\|_{\H}^2+2\alpha k\sum_{i=1}^n\|\y_{k,i}\|_{\H}^2+2\beta k\sum_{i=1}^n\|\y_{k,i}\|_{\L^{r+1}}^{r+1}+k\wi\mu\sum_{i=1}^n\|\y_{k,i}\|_{\V}^2\nonumber\\&\leq \|\y_{k,0}\|_{\H}^2+ \frac{2}{\wi\mu} \left[C_{\psi}^2\|\ell\|_{\mathcal{L}(\V;\U)}^2T+k\sum_{i=1}^n\|\f_{k,i}\|_{\V^{\prime}}^2\right]\nonumber\\&\leq \|\y_{k,0}\|_{\H}^2+ \frac{2}{\wi\mu} \left[C_{\psi}^2\|\ell\|_{\mathcal{L}(\V;\U)}^2T+k\sum_{i=1}^{\frac{T}{k}}\|\f_{k,i}\|_{\V^{\prime}}^2\right]\nonumber\\&\leq \|\y_{k,0}\|_{\H}^2+ \frac{2}{\wi\mu} \left[C_{\psi}^2\|\ell\|_{\mathcal{L}(\V;\U)}^2T+k\sum_{i=1}^{\frac{T}{k}}\|\frac{1}{k}\int_{(i-1)k}^{ik}\f(t)dt\|_{\V^{\prime}}^2\right]\nonumber\\&\leq \|\y_{k,0}\|_{\H}^2+ \frac{2}{\wi\mu} \left[C_{\psi}^2\|\ell\|_{\mathcal{L}(\V;\U)}^2T+\sum_{i=1}^{\frac{T}{k}}\int_{(i-1)k}^{ik}\|\f(t)\|_{\V^{\prime}}^2dt\right]\nonumber\\&\leq \|\y_{k,0}\|_{\H}^2+ \frac{2}{\wi\mu} \left[C_{\psi}^2\|\ell\|_{\mathcal{L}(\V;\U)}^2T+\|\f\|_{\mathrm{L}^2(0,T;\V^{\prime})}^2\right],
	\end{align}
and the required uniform bound \eqref{eqn-uni-bound} follows, since every convergent sequence is bounded.
\end{proof}
\subsubsection{Construction of interpolants and convergence analysis}
We proceed by constructing sequences of time-dependent piecewise constant and piecewise linear functions derived from the solution of the Rothe problem, and establish the convergence of a subsequence to a solution of the Problem \ref{prob-rothe}. Let us now define piecewise linear and piecewise constant interpolants $\y_k\in\C([0,T];\V^{\prime})$ and $\overline{\y}_k\in\mathrm{L}^{\infty}(0,T;\V)$ by the following formulae: 
\begin{align}
	\y_k(t)&=\y_{k,i}+\left(\frac{t}{k}-i\right)(\y_{k,i}-\y_{k,i-1}),\ \text{ for }\ t\in((i-1)k,ik], \ 1\leq i\leq N, \label{eqn-uk}\\
	\overline{\y}_k(t)&=\left\{\begin{array}{ll}\y_{k,i},&t\in((i-1)k,ik], \ 1\leq i\leq N,\\ \y_{k,0},&t=0.
	\end{array}\right.\label{eqn-baruk}
\end{align}
The piecewise constant function $\overline{\boldsymbol{\eta}}_{k}:(0,T]\to\U^{\prime}$ is given by 
\begin{align}
	\overline{\boldsymbol{\eta}}_{k}(t)={\boldsymbol{\eta}}_{k,i}, \ \text{ for }\ t\in((i-1)k,ik], \ 1\leq i\leq N.
\end{align}
Furthermore, we define $\f_k : (0,T] \to\V^{\prime}$ as follows: 
\begin{align}
	\f_{k}(t)=\f_{k,i}, \ \text{ for }\ t\in((i-1)k,ik], \ 1\leq i\leq N.
\end{align}
From \cite[Lemma 3.3]{CCJG}, we infer that $\f_k\to\f$ in $\mathrm{L}^2(0,T;\V^{\prime})$ as $k\to 0$. We observe that the distributional derivative of $\y_k$
is given by the expression 
\begin{align}\label{eqn-ukprime}
	\y_k^{\prime}(t)=\frac{1}{k}(\y_{k,i}-\y_{k,i-1}), \ \text{ for }\ t\in((i-1)k,ik),\ 1\leq i\leq N.\end{align}
Note that 
\begin{align}\label{eq-dual}
	\langle \y_k^{\prime}, \overline{\y}_k \rangle &= \int_{0}^{t}\langle \y_k^{\prime}(t), \overline{\y}_k (t) \rangle \d t \nonumber\\&= \frac{1}{k}\sum_{i=1}^N\int_{(i-1)k}^{ik} (\y_{k,i}-\y_{k,i-1},\y_{k,i-1})\d t
	\nonumber\\&= \frac{1}{2} \bigg(\sum_{i=1}^N(\|\y_{k,i}\|_{\H}^2-\|\y_{k,i-1}\|_{\H}^2)+\sum_{i=1}^N \|\y_{k,i}-\y_{k,i-1}\|_{\H}^2\bigg)\nonumber\\&= \frac{1}{2} \bigg(\|\y_{k,N}\|_{\H}^2-\|\y_{k,0}\|_{\H}^2+\sum_{i=1}^N \|\y_{k,i}-\y_{k,i-1}\|_{\H}^2\bigg) \nonumber\\& \geq \frac{1}{2} (\|\y_{k,N}\|_{\H}^2-\|\y_{k,0}\|_{\H}^2),
\end{align}
where we are using the definitions of $\y_k^{\prime}$ given in \eqref{eqn-ukprime} and $\overline{\y}_k$ given in \eqref{eqn-baruk}, and using the equality \eqref{eqn-ener-rothe-1}.

Therefore the equation \eqref{eqn-rothe} can be re-written for all $\z\in\V\cap\L^{r+1}$ and a.e. $t\in(0,T)$ as 
\begin{equation}\label{eqn-inclusion-new}
	\left\{
	\begin{aligned}
	&	(\y_k^{\prime}(t),\z)+\langle\mu\A\overline{\y}_k(t)+\B(\overline{\y}_k(t))+\alpha\overline{\y}_k(t)+\beta\mathfrak{C}(\overline{\y}_k(t)),\z\rangle+{}_{\U^{\prime}}\langle\overline{\boldsymbol{\eta}}_k(t),\ell\z\rangle_{\U}=\langle\f_k(t),\z\rangle,\\&\overline{\boldsymbol{\eta}}_k(t)\in\partial\psi(\ell\overline{\y}_k(t)). 
	\end{aligned}
	\right.
\end{equation}
 The problem \eqref{eqn-inclusion-new} is equivalent to the following for all $\z\in\mathrm{L}^{r+1}(0,T;\V\cap\L^{r+1})$ and a.e. $t\in(0,T)$: 
\begin{equation}\label{eqn-inclusion-new-1}
	\left\{
	\begin{aligned}
		&	(\y_k^{\prime},\z)+\langle\mu\mathscr{A}\overline{\y}_k+\mathscr{B}(\overline{\y}_k)+\alpha\mathscr{I}\overline{\y}_k+\beta\mathscr{C}\overline{\y}_k,\z\rangle+{}_{\mathcal{U}^{\prime}}\langle\overline{\boldsymbol{\eta}}_k,\ell\z\rangle_{\mathcal{U}}=\langle\f_k,\z\rangle,\\&\overline{\boldsymbol{\eta}}_k(t)\in\partial\psi((\overline{\ell}\overline{\y}_k)(t)).
	\end{aligned}
	\right.
\end{equation}
  
\begin{lemma}[A-priori estimates]\label{lem-uniform-ener-1}
	Under the assumptions of Theorem \ref{thm-rothe-existence}, there exists a constant $M_2>0$, which is independent of $k$ such that for $d\in\{2,3\}$ and $r\in[1,3]$
	\begin{align}\label{eqn-uniform-bounds}
		&\|\overline{\y}_k\|_{\mathrm{L}^{\infty}(0,T;\H)}+	\|\overline{\y}_k\|_{\mathrm{L}^{2}(0,T;\V)}+ 	\|\overline{\y}_k\|_{\mathrm{L}^{r+1}(0,T;\L^{r+1})}+\|\overline{\y}_k\|_{\mathcal{M}^{2,2}(0,T;\V;\H^2(\mathcal{O})^{\prime})}\nonumber\\&\quad+\|\y_k\|_{\C([0,T];\H)}+\|\y_k\|_{\mathrm{L}^2(0,T;\V)}+\|\y_k\|_{\mathrm{L}^{r+1}(0,T;\L^{r+1})}+\|\y_k^{\prime}\|_{\mathrm{L}^{2}(0,T;\H^2(\mathcal{O})^{\prime})}\nonumber\\&\quad+\|\overline{\boldsymbol{\eta}}_k\|_{\mathrm{L}^2(0,T;\U^{\prime})}\leq M_2.
	\end{align}

		 For $d\in\{2,3\}$ and $r>3,$  there exists a constant $M_3>0$ such that 
		 	\begin{align}\label{eqn-uniform-bounds-1}
		 	&\|\overline{\y}_k\|_{\mathrm{L}^{\infty}(0,T;\H)}+	\|\overline{\y}_k\|_{\mathrm{L}^{2}(0,T;\V)}+ 	\|\overline{\y}_k\|_{\mathrm{L}^{r+1}(0,T;\L^{r+1})}+\|\overline{\boldsymbol{\eta}}_k\|_{\mathrm{L}^2(0,T;\U^{\prime})}\nonumber\\&\quad+\|\y_k\|_{\C([0,T];\H)}+\|\y_k\|_{\mathrm{L}^2(0,T;\V)}+\|\y_k\|_{\mathrm{L}^{r+1}(0,T;\L^{r+1})}+\|\y_k^{\prime}\|_{\mathrm{L}^{\frac{r+1}{r}}(0,T;\V^{\prime}+\L^{\frac{r+1}{r}})}\leq M_3.
		 \end{align}
\end{lemma}
\begin{proof}
	It is clear from the definition of $\y_k(\cdot)$ given in \eqref{eqn-uk}, $\y_k$ is a continuous function of $t$. Therefore, using the estimate \eqref{eqn-uni-bound}, we infer that $\|\overline{\y}_k\|_{\mathrm{L}^{\infty}(0,T;\H)}=\max\limits_{1\leq i\leq N}\|\y_{k,i}\|_{\H}$ and $\|\y_k\|_{\C([0,T];\H)}=\max\limits_{0\leq i\leq N}\|\y_{k,i}\|_{\H}$ are bounded uniformly independent of $k$. From the definition of $\overline{\y}_k(\cdot)$ given in \eqref{eqn-baruk}, it is immediate that 
	\begin{align*}
	\|\overline{\y}_k\|_{\mathrm{L}^2(0,T;\V)}^2=k\sum_{i=1}^N\|\y_{k,i}\|_{\V}^2,
\end{align*}
	 and 
	 \begin{align*}
	 	\|\overline{\y}_k\|_{\mathrm{L}^{r+1}(0,T;\L^{r+1})}^2=k\sum_{i=1}^N\|\y_{k,i}\|_{\L^{r+1}}^{r+1},
\end{align*}	 	
	 	 and the uniform bound follows from the estimate \eqref{eqn-uni-bound}. Note that 
	\begin{align*}
		\|\y_k\|_{\mathrm{L}^2(0,T;\V)}^2&=\sum_{i=1}^N\int_{(i-1)k}^{ik}\bigg\|\y_{k,i}+\left(\frac{t}{k}-i\right)(\y_{k,i}-\y_{k,i-1})\bigg\|_{\V}^2\d t\leq k\sum_{i=0}^N\|\y_{k,i}\|_{\V}^2,
	\end{align*}
	and 
	\begin{align*}
		\|\y_k\|_{\mathrm{L}^{r+1}(0,T;\L^{r+1})}^{r+1}&=\sum_{i=1}^N\int_{(i-1)k}^{ik}\bigg\|\y_{k,i}+\left(\frac{t}{k}-i\right)(\y_{k,i}-\y_{k,i-1})\bigg\|_{\L^{r+1}}^{r+1}\d t\leq k\sum_{i=0}^N\|\y_{k,i}\|_{\L^{r+1}}^{r+1}.
	\end{align*}
Using the estimates in \eqref{eqn-uni-bound} and \eqref{approx-bound}, we get uniform bounds on $\|\y_k\|_{\mathrm{L}^2(0,T;\V)}$  and $\|\y_k\|_{\mathrm{L}^{r+1}(0,T;\L^{r+1})}$. Now, using Hypothesis \ref{hyp-psi-ell} (H2), we find 
	\begin{align}
		\|\overline{\boldsymbol{\eta}}_k\|_{\mathrm{L}^2(0,T;\U^{\prime})}^2&=\int_0^T	\|\overline{\boldsymbol{\eta}}_k(t)\|_{\U^{\prime}}^2\d t=\sum_{i=1}^N\int_{(i-1)k}^{ik}\|\boldsymbol{\eta}_{k,i}\|_{\U^{\prime}}^2\d t=k\sum_{i=1}^N\|\boldsymbol{\eta}_{k,i}\|_{\U^{\prime}}^2\nonumber\\&\leq 2kC_{\psi}^2\sum_{i=1}^N\left(1+\|\ell\overline{\y}_k\|_{\U}^2\right)\leq 2C_{\psi}^2\bigg(T+\|\ell\|_{\mathcal{L}(\V;\U)}^2k\sum_{k=1}^N\|\overline{\y}_k\|_{\V}^2\bigg)\nonumber\\&\leq  2C_{\psi}^2\left(T+\|\ell\|_{\mathcal{L}(\V;\U)}^2M_1\right), 
	\end{align}
	where we have used \eqref{eqn-uni-bound} also. We divide the next part of the proof into two cases:
	\vskip 2mm
	\noindent
	\textbf{Case 1.} \emph{For $d\in\{2,3\}$ with $r\in[1,3]$}. Let $\boldsymbol{\psi}\in\H^2(\mathcal{O})$,  using Agmon's inequality, we know that 
	\begin{align}\label{eqn-agmon}
		\|\boldsymbol{\psi}\|_{\L^{\infty}}\leq C\|\boldsymbol{\psi}\|_{\H}^{1-\frac{d}{4}}\|\boldsymbol{\psi}\|_{\H^2(\mathcal{O})}^{\frac{d}{4}}\leq C\|\boldsymbol{\psi}\|_{\H^2(\mathcal{O})}. 
	\end{align}
	
	Let us now take $\boldsymbol{\phi}\in\mathrm{L}^2(0,T;\H^2(\mathcal{O}))$.	From the first equation in \eqref{eqn-inclusion-new-1}, we 
	have 
	\begin{align}\label{eqn-time}
	\langle\y_k^{\prime},\boldsymbol{\phi}\rangle&=	(\y_k^{\prime},\boldsymbol{\phi})=\langle\f_k,\boldsymbol{\phi}\rangle - \langle\mu\mathscr{A}\overline{\y}_k+\mathscr{B}(\overline{\y}_k)+\alpha\overline{\y}_k+\beta\mathscr{C}(\overline{\y}_k),\boldsymbol{\phi}\rangle-{}_{\mathcal{U}^{\prime}}\langle\overline{\boldsymbol{\eta}}_k,\ell\boldsymbol{\phi}\rangle_{\mathcal{U}}\nonumber\\&\leq\left(\|\f_k\|_{\mathrm{L}^2(0,T;\V^{\prime})}+\mu\|\mathscr{A}\overline{\y}_k\|_{\mathrm{L}^2(0,T;\V^{\prime})}+C\alpha\|\overline{\y}_k\|_{\mathrm{L}^2(0,T;\H)}\right)\|\boldsymbol{\phi}\|_{\mathrm{L}^2(0,T;\V)}\nonumber\\&\quad+|\langle\mathscr{B}(\overline{\y}_k),\boldsymbol{\phi}\rangle|+\beta|\langle \mathscr{C}(\overline{\y}_k),\boldsymbol{\phi}\rangle|+\|\overline{\boldsymbol{\eta}}_k\|_{\mathrm{L}^2(0,T;\U^{\prime})}\|\ell\boldsymbol{\phi}\|_{\mathrm{L}^2(0,T;\U)}.
	\end{align}
	Let us now estimate the term $|\langle\mathscr{B}(\overline{\y}_k),\boldsymbol{\phi}\rangle|$ using H\"older's inequality and  \eqref{eqn-agmon} as 
	\begin{align}\label{eqn-time-1}
		|\langle\mathscr{B}(\overline{\y}_k),\boldsymbol{\phi}\rangle|&\leq \|\overline{\y}_k\|_{\mathrm{L}^2(0,T;\V)}\|\overline{\y}_k\|_{\mathrm{L}^{\infty}(0,T;\H)}\|\boldsymbol{\phi}\|_{\mathrm{L}^2(0,T;\L^{\infty})}\nonumber\\&\leq C\|\overline{\y}_k\|_{\mathrm{L}^2(0,T;\V)}\|\overline{\y}_k\|_{\mathrm{L}^{\infty}(0,T;\H)}\|\boldsymbol{\phi}\|_{\mathrm{L}^2(0,T;\H^2(\mathcal{O}))}.
	\end{align}
	Using interpolation and H\"older's  inequalities, we estimate the term $|\langle \mathscr{C}(\overline{\y}_k),\boldsymbol{\phi}\rangle|$ as 
	\begin{align}\label{eqn-time-2}
|\langle \mathscr{C}(\overline{\y}_k),\boldsymbol{\phi}\rangle|&\leq	\|\overline{\y}_k\|_{\mathrm{L}^{2r}(0,T;{\L^r})}^r\|\boldsymbol{\phi}\|_{\mathrm{L}^2(0,T;\L^{\infty})}\nonumber\\&\leq 	CT^{\frac{3-r}{2(r-1)}}\|\overline{\y}_k\|_{\mathrm{L}^{\infty}(0,T;\H)}^{\frac{2}{r-1}}\|\overline{\y}_k\|_{\mathrm{L}^{r+1}(0,T;\L^{r+1})}^{\frac{(r+1)(r-2)}{r-1}}\|\boldsymbol{\phi}\|_{\mathrm{L}^2(0,T;\H^2(\mathcal{O}))}
	\end{align}
	for $2\leq r\leq 3$. For $1\leq r\leq 2$, the embedding of $\mathrm{L}^{\infty}(0,T;{\L^2})\hookrightarrow\mathrm{L}^{2r}(0,T;{\L^r})$ 
yields the result \eqref{eqn-time-2}. Combining \eqref{eqn-time}-\eqref{eqn-time-2} and 
	using the uniform bounds on $\|\overline{\y}_k\|_{\mathrm{L}^2(0,T;\V)}$, $\|\overline{\y}_k\|_{\mathrm{L}^{\infty}(0,T;\H)}$, $\|\overline{\y}_k\|_{\mathrm{L}^{r+1}(0,T;\L^{r+1})}$ and $\|\overline{\boldsymbol{\eta}}_k\|_{\mathrm{L}^2(0,T;\U^{\prime})}$, one can obtain 
	\begin{align}\label{bound-1}
	\|\y_k^{\prime}\|_{\mathrm{L}^{2}(0,T;\H^2(\mathcal{O})^{\prime})}\leq C.
\end{align}
	
	Let us now find a uniform bound for $\|\overline{\y}_k\|_{\mathcal{M}^{2,2}(0,T;\V;\H^2(\mathcal{O})^{\prime})}$.  Let us assume that the seminorm in  $\mathrm{BV}^{2}(0,T;\H^2(\mathcal{O})^{\prime}))$ of the piecewise constant function $\overline{\y}_k$ 	is realized  some partition $0=a_0<a_1<\cdots<a_n=T$, and each $a_i$ is in different interval $((m_i-1)k,m_ik]$ such that $\overline{\y}_k(a_i)=\y_{k,m_i}$ with $m_0=0$, $m_n=N$  and $m_{i+1}>m_i$ for $i=1,2,\ldots,N-1$. By telescopic sum and using the properties of norm, we get 
	\begin{align}\label{bound-2}
		\|\y_{k,m_i}-\y_{k,m_{i-1}}\|_{\H^2(\mathcal{O})^{\prime}}^{2}&=\bigg\|\sum_{j=m_{i-1}+1}^{m_i}(\y_{k,j}-\y_{k,j-1})\bigg\|_{\H^2(\mathcal{O})^{\prime}}^{2}\nonumber\\
		&\leq(m_i-m_{i-1})\bigg(\sum_{j=m_{i-1}+1}^{m_i}\|\y_{k,j}-\y_{k,j-1}\|_{\H^2(\mathcal{O})^{\prime}}^{2}\bigg).
	\end{align}
	Therefore, using the \eqref{def-BV}, \eqref{bound-1} and \eqref{bound-2}, we find 
	\begin{align}\label{eqn-bv-est}
		\|\overline{\y}_k\|_{\mathrm{BV}^{2}(0,T;\H^2(\mathcal{O})^{\prime})}^{2}&=\sum_{i=1}^n\|\y_{k,m_i}-\y_{k,m_{i-1}}\|_{\H^2(\mathcal{O})^{\prime}}^{2}\nonumber\\&\leq \sum_{i=1}^n(m_i-m_{i-1})\bigg(\sum_{j=m_{i-1}+1}^{m_i}\|\y_{k,j}-\y_{k,j-1}\|_{\H^2(\mathcal{O})^{\prime}}^{2}\bigg)\nonumber\\&\leq \sum_{i=1}^n(m_i-m_{i-1})\sum_{j=1}^{n}\|\y_{k,j}-\y_{k,j-1}\|_{\H^2(\mathcal{O})^{\prime}}^{2}\nonumber\\&\leq Nk^2\sum_{j=1}^{n}\bigg\|\frac{\y_{k,j}-\y_{k,j-1}}{k}\bigg\|_{\H^2(\mathcal{O})^{\prime}}^{2}\nonumber\\
		&=Nk\int_0^T\|\y_{k}^{\prime}(t)\|_{\H^2(\mathcal{O})^{\prime}}^2\d t\nonumber\\&= T \|\y_k^{\prime}\|_{\mathrm{L}^{2}(0,T;\H^2(\mathcal{O})^{\prime})}\leq C.
	\end{align}
	Consequently, from the uniform bound on $\|	\overline{\y}_k\|_{\mathrm{L}^2(0,T;\V)}$, we deduce that $\overline{\y}_k$  is bounded in $\mathcal{M}^{2,2}(0,T;\V,\H^2(\mathcal{O})^{\prime})$(see definition \eqref{def-M}), which completes the proof for the first case.

		\vskip 2mm
	\noindent
	\textbf{Case 2.} \emph{For $d\in\{2,3\}$ and $r\in(3,\infty)$}. Let us now take $\boldsymbol{\phi}\in\mathrm{L}^2(0,T;\V)\cap\mathrm{L}^{r+1}(0,T;\L^{r+1})$. Then a calculation similar to \eqref{eqn-time} yields 
		\begin{align}\label{eqn-time-3}
	|\langle\y_k^{\prime},\boldsymbol{\phi}\rangle|&\leq\left(\|\f_k\|_{\mathrm{L}^2(0,T;\V^{\prime})}+\mu\|\mathscr{A}\overline{\y}_k\|_{\mathrm{L}^2(0,T;\V^{\prime})}+C\alpha\|\overline{\y}_k\|_{\mathrm{L}^2(0,T;\H)}\right)\|\boldsymbol{\phi}\|_{\mathrm{L}^2(0,T;\V)}\nonumber\\&\quad+\Big(\|\mathscr{B}(\overline{\y}_k)\|_{\mathrm{L}^{\frac{r+1}{r}}(0,T;\L^{\frac{r+1}{r}})}+\beta\| \mathscr{C}(\overline{\y}_k)\|_{\mathrm{L}^{\frac{r+1}{r}}(0,T;\L^{\frac{r+1}{r}})}\Big)\|\boldsymbol{\phi}\|_{\mathrm{L}^{r+1}(0,T;\L^{r+1})}\nonumber\\&\quad+\|\overline{\boldsymbol{\eta}}_k\|_{\mathrm{L}^2(0,T;\U^{\prime})}\|\ell\boldsymbol{\phi}\|_{\mathrm{L}^2(0,T;\U)}\nonumber\\&\leq \left(\|\f_k\|_{\mathrm{L}^2(0,T;\V^{\prime})}+\mu\|\overline{\y}_k\|_{\mathrm{L}^2(0,T;\V)}+C\alpha\|\overline{\y}_k\|_{\mathrm{L}^2(0,T;\H)}\right)\|\boldsymbol{\phi}\|_{\mathrm{L}^2(0,T;\V)}\nonumber\\&\quad+\Big(T^{\frac{r-3}{2(r-1)}}\|\overline{\y}_k\|_{\mathrm{L}^{\infty}(0,T;\H)}^2\|\overline{\y}_k\|_{\mathrm{L}^2(0,T;\V)}\|\overline{\y}_k\|_{\mathrm{L}^{r+1}(0,T;\L^{r+1})}^{\frac{2}{r-1}}+\beta\| \overline{\y}_k\|_{\mathrm{L}^{r+1}(0,T;\L^{r+1})}^r\Big)\nonumber\\&\qquad\times\|\boldsymbol{\phi}\|_{\mathrm{L}^{r+1}(0,T;\L^{r+1})}+\|\overline{\boldsymbol{\eta}}_k\|_{\mathrm{L}^2(0,T;\U^{\prime})}\|\ell\|_{\mathcal{L}(\V;\U)}\|\boldsymbol{\phi}\|_{\mathrm{L}^2(0,T;\V)},
	\end{align}
	where we have used \eqref{eqn-b-est-lr}. Using the uniform bounds on $\|\overline{\y}_k\|_{\mathrm{L}^2(0,T;\V)}$, $\|\overline{\y}_k\|_{\mathrm{L}^{\infty}(0,T;\H)}$, $\|\overline{\y}_k\|_{\mathrm{L}^{r+1}(0,T;\L^{r+1})}$ and $\|\overline{\boldsymbol{\eta}}_k\|_{\mathrm{L}^2(0,T;\U^{\prime})}$, we have 
	\begin{align}
		\|\y_k^{\prime}\|_{\mathrm{L}^2(0,T;\V^{\prime})+\mathrm{L}^{\frac{r+1}{r}}(0,T;\L^{\frac{r+1}{r}})}\leq C, 
	\end{align}
	which completes the proof. Since the estimate \eqref{eqn-time-3} holds true for all $\boldsymbol{\phi} \in\mathrm{L}^{r+1}(0,T;\V\cap\L^{r+1})$, the estimate 	$\|\y_k^{\prime}\|_{\mathrm{L}^{\frac{r+1}{r}}(0,T;\V^{\prime}+\L^{\frac{r+1}{r}})}\leq C,$ holds as well.
\end{proof}
	
	 \begin{theorem}[Extracting convergent subsequences]\label{thm-main-non-station}
	 	For  $r\in[1,\infty)$ and $\mu> C_{\psi}\|\ell\|_{\mathcal{L}(\V;\U)}^2,$ let  Hypothesis \ref{hyp-psi-ell} (H1), (H2) and Hypothesis \ref{hyp-new-ell} (H4$'$) be satisfied. Assume that $\y_0\in\H$ and $\f\in\mathrm{L}^2(0,T;\V^{\prime})$. There exists a pair $(\y,\boldsymbol{\eta})\in\mathcal{W}\times \mathrm{L}^2(0,T;\U^{\prime})$ such that $(\y,\boldsymbol{\eta})$ is a  weak solution of the Problem \ref{eqn-abstract-4}. 
	 \end{theorem}
	
	\begin{proof}
	Let $\y_k,\overline{\y}_k$ denote the piecewise linear and piecewise constant time interpolants, respectively, constructed from the time-discrete solution of the Rothe Problem \eqref{eqn-rothe}. 	
	
		\vskip 2mm
	\noindent
	\textbf{Case 1.} For $d\in\{2,3\}$ with $r\in[1,3]$.
	Using the uniform bounds given in \eqref{eqn-uniform-bounds}, the Banach-Alaoglu theorem guarantees the existence of  subsequences of $\y_k$, $\overline{\y}_k$ and $\overline{\boldsymbol{\eta}}_k$ (still denoted by the same symbol)  and $\overline{\y},\y\in\mathrm{L}^{\infty}(0,T;\H)\cap \mathrm{L}^2(0,T;\V)\cap\mathrm{L}^{r+1}(0,T;\wi\L^{r+1}),$ $\y_1\in\mathrm{L}^{2}(0,T;\H^2(\mathcal{O})^{\prime}),$ and $\boldsymbol{\eta}\in\mathrm{L}^2(0,T;\U^{\prime})$   such that as $k\to 0$ 
	\begin{equation}\label{eqn-conv-eta}
		\left\{
		\begin{aligned}
&\overline{\y}_k\xrightarrow{w^*} \overline{\y}\ \text{ in }\ \mathrm{L}^{\infty}(0,T;\H) \ \text{ and }\ \overline{\y}_k\xrightarrow{w} \overline{\y}\ \text{ in }\ \mathrm{L}^{2}(0,T;\V)\cap\mathrm{L}^{r+1}(0,T;\L^{r+1}), \\
&{\y}_k\xrightarrow{w^*} {\y}\ \text{ in }\ \mathrm{L}^{\infty}(0,T;\H) \ \text{ and }\ {\y}_k\xrightarrow{w} {\y}\ \text{ in }\ \mathrm{L}^{2}(0,T;\V)\cap\mathrm{L}^{r+1}(0,T;\L^{r+1}),
\\
&
\overline{\boldsymbol{\eta}}_k\xrightarrow{w}\boldsymbol{\eta}\ \text{ in }\ \mathrm{L}^2(0,T;\U^{\prime}),
		\end{aligned}
			\right. 
		\end{equation}
		and 
		\begin{equation}\label{eqn-conv-time-der}
		{\y}_k'\xrightarrow{w} {\y}_1\ \text{ in }\ 
	\left\{\begin{array}{ll}
			\mathrm{L}^{2}(0,T;\H^2(\mathcal{O})^{\prime}),&\text{for }\ d\in\{2,3\}\ \text{ with }\ r\in[1,3],\\
			\mathrm{L}^{2}(0,T;\V^{\prime})+\mathrm{L}^{\frac{r+1}{r}}(0,T;\L^{\frac{r+1}{r}}),&\text{for }\ d\in\{2,3\}\ \text{ with }\ r\in(3,\infty). 
			\end{array}\right.
		\end{equation}
Since $\|\overline{\y}_k\|_{\mathcal{M}^{2,2}(0,T;\V;\H^2(\mathcal{O})^{\prime})}$ is bounded by $M_2$ (see \eqref{eqn-uniform-bounds}),	by taking  $\E_1=\V$, $\E_2=\H$, $\E_3=\H^2(\mathcal{O})^{\prime}$, $p=q=2$ in Theorem \ref{thm-Simon}, we obtain the existence of a subsequence $\overline{\y}_k$ (still denoted by the same symbol) such that
\begin{align}\label{eqn-conv-storng}
	\overline{\y}_k\to \overline{\y}\ \text{ in }\ \mathrm{L}^{2}(0,T;\H). 
\end{align}
Therefore, along a further subsequence (labeled with the same symbol), we have the following convergence: 
\begin{align}\label{eqn-conv-point}
		\overline{\y}_k(x,t)\to \overline{\y}(x,t)\ \text{ for a.e. }\ (x,t)\in\mathcal{O}\times(0,T). 
\end{align}
	Let us now show that $\overline{\y}=\y$. By the definition of $\y_k$ and $\overline{\y}_k$  given in \eqref{eqn-uk} and \eqref{eqn-baruk}, respectively, we calculate 
		\begin{align}\label{eqn-conv-unique}
		&	\|\overline{\y}_k-\y_k\|_{\mathrm{L}^2(0,T;\H^2(\mathcal{O})^{\prime})}^2\nonumber\\&=\int_0^T\|\overline{\y}_k(t)-\y_k(t)\|_{\H^2(\mathcal{O})^{\prime}}^2\d t=\sum_{i=1}^N\int_{(i-1)k}^{ik}\|\overline{\y}_k(t)-\y_k(t)\|_{\H^2(\mathcal{O})^{\prime}}^2\d t\nonumber\\&=\sum_{i=1}^N\int_{(i-1)k}^{ik}(t-ik)^2\bigg\|\frac{\y_{k,i}-\y_{k.i-1}}{k}\bigg\|_{\H^2(\mathcal{O})^{\prime}}^2\d t=\frac{k^3}{3}\|\y_k^{\prime}\|_{\mathrm{L}^2(0,T;\H^2(\mathcal{O})^{\prime})}^2\nonumber\\&\leq{Ck^3}\to 0\ \text{ as }\ k\to 0.
		\end{align}
		Therefore, we have 
	\begin{equation}\label{estimate-1}
\overline{\y}_k-\y_k\to 0 \text{ in } \mathrm{L}^2(0,T;\H^2(\mathcal{O})^{\prime}) \text{ as } k\to 0.
\end{equation}
	 Since the embedding $\mathrm{L}^2(0,T;\V)\hookrightarrow\mathrm{L}^2(0,T;\H^2(\mathcal{O})^{\prime})$ is continuous, and from convergence \eqref{eqn-conv-eta}, we also have
	\begin{equation}\label{estimate-2}
	 \overline{\y}_k-\y_k\xrightarrow{w}\overline{\y}-\y \text{ in } \mathrm{L}^2(0,T;\H^2(\mathcal{O})^{\prime}) \text{ as } k\to 0. 
	 \end{equation}
Thus, using \eqref{estimate-1}, \eqref{estimate-2} and the uniqueness of weak limits, we have $\overline{\y}=\y$. Since ${\y}_k\xrightarrow{w}\y$ in $\mathrm{L}^{2}(0,T;\V)$ and ${\y}_k'\xrightarrow{w} {\y}_1\ \text{ in }\ \mathrm{L}^{2}(0,T;\H^2(\mathcal{O})^{\prime}),$ one can conclude from \cite[Chapter I, Proposition 1.2]{JHMMPD} that $\y_1=\y'$. Hence, for all $\z\in\mathrm{C}([0,T];\H^2(\mathcal{O}))$, we obtain 
		\begin{align}\label{eqn-conv-time}
			(\y_k^{\prime},\z)=\langle\y_k^{\prime},\z\rangle\to \langle\y^{\prime},\z\rangle\ \text{ as }\ k\to 0. 
		\end{align}
	Let us define \begin{align}\label{eqn-q}q=\max\left\{r+1,\frac{4}{4-d}\right\}.\end{align} Since $\mathrm{C}([0,T];\H^2(\mathcal{O}))$ is dense in   $\mathrm{L}^q(0,T;\V)$, the above convergence holds true for all $\z\in\mathrm{L}^q(0,T;\V)$.
		
		Since $\mathscr{A}$ is a linear and continuous operator from $\mathrm{L}^2(0,T;\V)$ to $\mathrm{L}^2(0,T;\V^{\prime})$ and is also weakly continuous. Since $\overline{\y}_k\xrightarrow{w} {\y}\ \text{ in }\ \mathrm{L}^{2}(0,T;\V)$, it is immediate that for all $\z\in\mathrm{L}^2(0,T;\V)$
		\begin{align}\label{eqn-conv-a}
			\langle\mathscr{A}\overline{\y}_k,\z\rangle\to 	\langle\mathscr{A}{\y},\z\rangle\ \text{ as }\ k\to 0. 
		\end{align}
		\vskip 0.2cm
			\noindent
		\emph{Convergence of the bilinear term:} 
		Let us now consider for all $\z\in\mathrm{C}([0,T];\V),$
		\begin{align}\label{eqn-conv-b}
			|\langle\mathscr{B}(\overline{\y}_k)-\mathscr{B}({\y}),\z\rangle|\leq|	\langle\mathscr{B}(\overline{\y}_k,\overline{\y}_k-{\y}),\z\rangle|+|\langle\mathscr{B}(\overline{\y}_k-{\y},{\y}),\z\rangle|.
		\end{align}
		By using H\"older's inequality, we estimate $|\langle\mathscr{B}(\overline{\y}_k-{\y},{\y}),\z\rangle|$ as 
		\begin{align*}
		|\langle\mathscr{B}(\overline{\y}_k-{\y},{\y}),\z\rangle|&\leq\left(\|\overline{\y}_k\|_{\mathrm{L}^2(0,T;\V)}+\|{\y}_k\|_{\mathrm{L}^2(0,T;\V)}\right)\|\y\|_{\mathrm{L}^2(0,T;\L^4)}\|\z\|_{\mathrm{L}^{\infty}(0,T;\L^4)}
		\nonumber\\&\leq \left(\|\overline{\y}_k\|_{\mathrm{L}^2(0,T;\V)}+\|{\y}_k\|_{\mathrm{L}^2(0,T;\V)}\right)\|\y\|_{\mathrm{L}^2(0,T;\V)}\|\z\|_{\mathrm{L}^{\infty}(0,T;\V)}. 
		\end{align*}
The convergence \eqref{eqn-conv-point} and the Lebesgue's Dominated Convergence Theorem yield 
		\begin{align}\label{eqn-conv-b-1}
			|\langle\mathscr{B}(\overline{\y}_k-{\y},{\y}),\z\rangle|\to\ \text{ as }\ k\to 0. 
		\end{align}
		Using H\"older's and Gagliardo-Nirenberg interpolation inequalities, we estimate $	|	\langle\mathscr{B}(\overline{\y}_k,\overline{\y}_k-{\y}),\z\rangle|$ as 
		\begin{align}\label{eqn-conv-b-2}
	&	|	\langle\mathscr{B}(\overline{\y}_k,\overline{\y}_k-{\y}),\z\rangle|\nonumber\\&\leq \|\overline{\y}_k\|_{\mathrm{L}^2(0,T;\V)}\|\overline{\y}_k-\overline{\y}\|_{\mathrm{L}^2(0,T;\L^3)}\|\z\|_{\mathrm{L}^{\infty}(0,T;\L^6)}\nonumber\\&\leq C\|\overline{\y}_k\|_{\mathrm{L}^2(0,T;\V)}\left(\|\overline{\y}_k\|_{\mathrm{L}^2(0,T;\V)}^{\frac{d}{6}}+\|\overline{\y}\|_{\mathrm{L}^2(0,T;\V)}^{\frac{d}{6}}\right)\|\overline{\y}_k-\overline{\y}\|_{\mathrm{L}^2(0,T;\H)}^{\frac{6-d}{6}}\|\z\|_{\mathrm{L}^{\infty}(0,T;\V)}\nonumber\\&\to 0\ \text{ as }\ k\to 0,
		\end{align}
		where we have used the uniform bound \eqref{eqn-uniform-bounds} and the strong convergence \eqref{eqn-conv-storng} also. Combining the convergences \eqref{eqn-conv-b-1} and \eqref{eqn-conv-b-2} and substituting in \eqref{eqn-conv-b}, we obtain  for all $\z\in\mathrm{C}([0,T];\V),$
		\begin{align}\label{eqn-conv-b-3}
			\langle\mathscr{B}(\overline{\y}_k),\z\rangle\to \langle\mathscr{B}({\y}),\z\rangle \ \text{ as }\ k\to 0. 
		\end{align}
Since $\C([0,T];\V)$ is dense in $\mathrm{L}^{\frac{4}{4-d}}(0,T;\V)$, one can prove that above convergence holds true for all $\z\in \mathrm{L}^{\frac{4}{4-d}}(0,T;\V)$ also. For this, consider the case $d=3$ and assume $\z\in \mathrm{L}^{4}(0,T;\V)$. Then, for any given $\varepsilon>0$, there exists a sequence $\z_{\varepsilon}\in\C([0,T];\V)$ such that $$\|\z_{\varepsilon}-\z\|_{\mathrm{L}^{4}(0,T;\V)}\leq {\varepsilon}.$$ Therefore, by using the Gagliardo-Nirenberg interpolation inequality and Sobolev's embedding, we have 
		\begin{align}
			&	|\langle\mathscr{B}(\overline{\y}_k)-\mathscr{B}({\y}),\z\rangle|\nonumber\\&\leq 	{|\langle\mathscr{B}(\overline{\y}_k,\overline{\y}_k-\y),\z-\z_{\varepsilon}\rangle|+|\langle\mathscr{B}(\overline{\y}_k-\y,\y),\z-\z_{\varepsilon}\rangle|+|\langle\mathscr{B}(\overline{\y}_k)-\mathscr{B}({\y}),\z_{\varepsilon}\rangle|}\nonumber\\&\leq \|\overline{\y}_k\|_{\mathrm{L}^2(0,T;\V)}\left(\|\overline{\y}_k\|_{\mathrm{L}^{4}(0,T;\L^3)}+\|\y\|_{\mathrm{L}^{4}(0,T;\L^3)}\right)\|\z_{\varepsilon}-\z\|_{\mathrm{L}^{4}(0,T;\L^6)}\nonumber\\&\quad+\left(\|\overline{\y}_k\|_{\mathrm{L}^2(0,T;\V)}+\|\y\|_{\mathrm{L}^2(0,T;\V)}\right)\|\y\|_{\mathrm{L}^{4}(0,T;\L^3)}\|\z_{\varepsilon}-\z\|_{\mathrm{L}^{4}(0,T;\L^6)}
			+|\langle\mathscr{B}(\overline{\y}_k)-\mathscr{B}({\y}),\z_{\varepsilon}\rangle|\nonumber\\&\leq C\varepsilon  \|\overline{\y}_k\|_{\mathrm{L}^2(0,T;\V)}\left(\|\overline{\y}_k\|_{\mathrm{L}^{\infty}(0,T;\H)}^{1/2}\|\overline{\y}_k\|_{\mathrm{L}^{2}(0,T;\V)}^{1/2}+\|\y\|_{\mathrm{L}^{\infty}(0,T;\H)}^{1/2}\|\y\|_{\mathrm{L}^{2}(0,T;\V)}^{1/2}\right)\nonumber\\&\quad+C\varepsilon \left(\|\overline{\y}_k\|_{\mathrm{L}^2(0,T;\V)}+\|\y\|_{\mathrm{L}^2(0,T;\V)}\right)\|\y\|_{\mathrm{L}^{\infty}(0,T;\H)}^{1/2}\|\y\|_{\mathrm{L}^{2}(0,T;\V)}^{1/2}\nonumber\\&\quad+|\langle\mathscr{B}(\overline{\y}_k)-\mathscr{B}({\y}),\z_{\varepsilon}\rangle|
		\end{align}
		Passing $k\to 0$ and using \eqref{eqn-uniform-bounds} and \eqref{eqn-conv-b-3}, we arrive at 
		\begin{align}
			|\langle\mathscr{B}(\overline{\y}_k)-\mathscr{B}({\y}),\z\rangle|\leq CM_2\varepsilon.
		\end{align}
		Since $\varepsilon>0$ arbitrary, we infer that for any $\z\in \mathrm{L}^{4}(0,T;\V)$, the convergence \eqref{eqn-conv-b-3} holds. Similarily, for $d=2$, the convergence \eqref{eqn-conv-b-3} holds for all $\z\in \mathrm{L}^{2}(0,T;\V)$. 
		\vskip 0.2cm
		\noindent
		\emph{Convergence of the nonlinear term:} 	
	For any $\z\in\C([0,T];\H^2(\mathcal{O}))$, and by using Taylor's formula (\cite[Theorem 7.9.1]{PGC}), H\"older's and interpolation inequalities, we find 
			\begin{align}\label{eqn-conv-c-1}
			&	|\langle\mathscr{C}(\overline{\y}_k)-\mathscr{C}({\y}),\z\rangle|\nonumber\\&=\bigg|\int_0^T\langle\mathfrak{C}(\overline{\y}_k(t))-\mathfrak{C}(\y(t)),\z(t)\rangle\d t\bigg|&\nonumber\\&=\bigg|\int_0^T\bigg\langle\int_0^1\mathfrak{C}^{\prime}(\theta\overline{\y}_k(t)+(1-\theta)\y(t))(\overline{\y}_k(t)-\y(t))\d\theta,\z(t)\bigg\rangle\d t\bigg|\nonumber\\&\leq r\int_0^T\int_0^1\||\theta\overline{\y}_k(t)+(1-\theta)\y(t)|^{r-1}(\overline{\y}_k(t)-\y(t))\|_{\L^1}\|\z(t)\|_{\L^{\infty}}\d\theta\d t\nonumber\\&\leq C\sup_{t\in[0,T]}\|\z(t)\|_{\H^2(\mathcal{O})}\int_0^T\left(\|\overline{\y}_k(t)\|_{\L^{2(r-1)}}^{r-1}+\|\y(t)\|_{\L^{2(r-1)}}^{r-1}\right)\|\overline{\y}_k(t)-\y(t)\|_{\H}\d t\nonumber\\&\leq C\sup_{t\in[0,T]}\|\z(t)\|_{\H^2(\mathcal{O})}\|\overline{\y}_k-\y\|_{\mathrm{L}^2(0,T;\H)}\nonumber\\&\quad\times \left\{\begin{array}{l}\left(\|\overline{\y}_k\|_{\mathrm{L}^2(0,T;\H)}+\|\y\|_{\mathrm{L}^2(0,T;\H)}\right)T^{\frac12}, \hspace{1cm}\text{ for }1\leq r\leq 2,\\ \left(\|\overline{\y}_k\|_{\mathrm{L}^2(0,T;\H)}^{\frac{3-r}{r-1}}+\|\y\|_{\mathrm{L}^2(0,T;\H)}^{\frac{3-r}{r-1}}\right)\left(\|\overline{\y}_k\|_{\mathrm{L}^{r+1}(0,T;\L^{r+1})}^{\frac{(r-2)(r+1)}{r-1}}+\|\y\|_{\mathrm{L}^{r+1}(0,T;\L^{r+1})}^{\frac{(r-2)(r+1)}{r-1}}\right),\\ \hspace{7cm}
				\text{ for }2\leq r\leq 3,
			\end{array}\right.\nonumber\\&\to 0\ \text{ as } \ k\to 0, 
		\end{align}
		where we have used the uniform bound \eqref{eqn-uniform-bounds} and the strong convergence \eqref{eqn-conv-storng} also. Since $\C([0,T];\H^2(\mathcal{O}))$ is dense in $\mathrm{L}^{r+1}(0,T;\V)$, one can prove that above convergence holds true for all $\z\in \mathrm{L}^{r+1}(0,T;\V)$ also. For this, let us  assume that $\z\in \mathrm{L}^{r+1}(0,T;\V)$. Then, for any given $\varepsilon>0$, there exists a sequence $\z_{\varepsilon}\in\C([0,T];\H^2(\mathcal{O}))$ such that $$\|\z_{\varepsilon}-\z\|_{\mathrm{L}^{r+1}(0,T;\V)}\leq {\varepsilon}.$$ Therefore, by using Sobolev's embedding and H\"older's inequality, we deduce 
			\begin{align}
			&	|\langle\mathscr{C}(\overline{\y}_k)-\mathscr{C}({\y}),\z\rangle|\nonumber\\&\leq 	|\langle\mathscr{C}(\overline{\y}_k)-\mathscr{C}({\y}),\z-\z_{\varepsilon}\rangle|+	|\langle\mathscr{C}(\overline{\y}_k)-\mathscr{C}({\y}),\z_{\varepsilon}\rangle|\nonumber\\&\leq \left(\|\overline{\y}_k\|_{\mathrm{L}^{r+1}(0,T;\L^{r+1})}^r+\|\overline{\y}\|_{\mathrm{L}^{r+1}(0,T;\L^{r+1})}^r\right)\|\z-\z_{\varepsilon}\|_{\mathrm{L}^{r+1}(0,T;\L^{r+1})}\nonumber\\&\quad+	|\langle\mathscr{C}(\overline{\y}_k)-\mathscr{C}({\y}),\z_{\varepsilon}\rangle|\nonumber\\&\leq C\left(\|\overline{\y}_k\|_{\mathrm{L}^{r+1}(0,T;\L^{r+1})}^r+\|\overline{\y}\|_{\mathrm{L}^{r+1}(0,T;\L^{r+1})}^r\right)\|\z-\z_{\varepsilon}\|_{\mathrm{L}^{r+1}(0,T;\V)}\nonumber\\&\quad+|\langle\mathscr{C}(\overline{\y}_k)-\mathscr{C}({\y}),\z_{\varepsilon}\rangle|. 
			\end{align}
				Passing $k\to 0$ and using \eqref{eqn-uniform-bounds} and \eqref{eqn-conv-c-1} lead to
			\begin{align}
				|\langle\mathscr{C}(\overline{\y}_k)-\mathscr{C}({\y}),\z\rangle|\leq CM_2\varepsilon.
			\end{align}
			Since $\varepsilon>0$ arbitrary, we infer that for any $\z\in \mathrm{L}^{r+1}(0,T;\V)$, the convergence \eqref{eqn-conv-c-1} holds. 
			
			Using the final convergence given in \eqref{eqn-conv-eta}, we get for all $\z\in\mathrm{L}^2(0,T;\V)$
			\begin{align}\label{eqn-conv-w-eta}
				{}_{\mathscr{U}^{\prime}}\langle\boldsymbol{\overline{\eta}}_k,\overline{\ell}\z\rangle_{\mathscr{U}}\to {}_{\mathscr{U}^{\prime}}\langle \boldsymbol{\eta},\overline{\ell}\z\rangle_{\mathscr{U}}\ \text{ as }\ k\to 0. 
			\end{align}
			Since $\f_k\to\f$ in $\mathrm{L}^{2}(0,T;\V^{\prime})$, we immediately derive for all $\z\in\mathrm{L}^2(0,T;\V)$
			\begin{align}\label{eqn-conv-w-f}
				\langle\f_k,\z\rangle \to \langle\f,\z\rangle\ \text{ as }\ k\to 0. 
			\end{align}
			
			We use the convergences \eqref{eqn-conv-time}, \eqref{eqn-conv-a}, \eqref{eqn-conv-b-3},  \eqref{eqn-conv-c-1}, \eqref{eqn-conv-w-eta} and \eqref{eqn-conv-w-f} to pass to the limit as $k\to 0$ in \eqref{eqn-inclusion-new-1} to obtain 	for all $\z\in\mathrm{L}^q(0,T;\V)$
			\begin{align}
				\langle\y',\z\rangle+\langle\mu\mathscr{A}{\y}+\mathscr{B}({\y})+\alpha{\y}+\beta\mathscr{C}({\y}),\z\rangle+{}_{\mathcal{U}^{\prime}}\langle\boldsymbol{\eta},\overline{\ell}\z\rangle_{\mathcal{U}}=\langle\f,\z\rangle,
			\end{align}
		and $\y^{\prime}\in \mathrm{L}^{q'}(0,T;\V^{\prime})$, where $q$ is defined in \eqref{eqn-q} and $q^{\prime}$ is such that $\frac1q+\frac{1}{q^{\prime}}=1$. 
		
			
			Note that compact operator maps weakly convergent sequences into strongly convergent sequences. We infer from Hypothesis \ref{hyp-new-ell}(H4$'$) that the operator $\overline{\ell}$ is compact, and since $\overline{\y}_k\xrightarrow{w}\y$ in $\mathrm{L}^2(0,T;\V)$, therefore we have $\overline{\ell}\overline{\y}_k\to\overline{\ell}\y$ in $\mathrm{L}^2(0,T;\U)$. Consequently, along a subsequence 
		\begin{align}\label{conv-2}
		\overline{\ell}\overline{\y}_k(t)\to\overline{\ell}\y(t) \text{ in } \U \text{ for a.e. } t\in(0,T).		
		\end{align}
	  Also, the mapping $\partial\psi:\U\to 2^{\U^{\prime}}$ has nonempty, closed and convex values and is upper semicontinuous from $\U$ (endowed with strong topology) into $\U^{\prime}$ (endowed with weak topology)(see Proposition \ref{prop-ups}). Thus, using the final convergence given in \eqref{eqn-conv-eta}, \eqref{conv-2} and Theorem \ref{thm-conv-inc}, we obtain
			\begin{align}
				\boldsymbol{\eta}(t)\in\partial\psi(\ell\y(t)),\ \text{ for a.e. }\ t\in(0,T). 
			\end{align} 

\vskip 0.2cm
\noindent
\emph{Initial data:}
			On the other hand $\y\in\mathrm{L}^{\infty}(0,T;\H)$ implies $\y\in\mathrm{L}^{q'}(0,T;\V^{\prime}),$ so that $\y\in\mathrm{W}^{1,q^{\prime}}(0,T;\V^{\prime})$ and hence $\y\in\mathrm{C}([0,T];\V^{\prime})$. Therefore, $\y\in \mathrm{L}^{\infty}(0,T;\H)\cap\mathrm{C}([0,T];\V^{\prime})$ and application of Strauss Lemma (Lemma \ref{lem-Strauss}) provides $\y\in \mathrm{C}_w([0,T];\H)$. Since $\y_k\in\mathrm{C}([0,T];\H)$ (see \eqref{eqn-uniform-bounds}) and $\y\in \mathrm{C}_w([0,T];\H)$, the second convergence in \eqref{eqn-conv-eta} implies
			\begin{align*}
(\y_k(t),\boldsymbol{\phi})\to (\y(t),\boldsymbol{\phi})\ \text{ as }\ k\to 0 \ \text{ for all } \ t\in[0,T]\ \text{ and }\  \boldsymbol{\phi}\in\H. 
			\end{align*}
			Since $\y_{k}(0)=\y_{k,0}\to\y_0$ as $k\to 0$ in $\H$, we have  for all $ \boldsymbol{\phi}\in\H$
			\begin{align*}
				|(\y(t)-\y_0,\boldsymbol{\phi})|&\leq|(\y(t)-\y_k(t),\boldsymbol{\phi})|+|(\y_k(t)-\y_k(0),\boldsymbol{\phi})|+|(\y_k(0)-\y_0,\boldsymbol{\phi})|\nonumber\\&\leq  |(\y(t)-\y_k(t),\boldsymbol{\phi})|+\|\y_k(t)-\y_k(0)\|_{\H}\|\boldsymbol{\phi}\|_{\H}+\|\y_k(0)-\y_0\|_{\H}\|\boldsymbol{\phi}\|_{\H}
				\nonumber\\&\to 0
				\ \text{ as } k\to 0\ \text{ and then }t \to 0, 
			\end{align*}
so that 
\begin{align}\label{eqn-initial}
	\lim_{t\downarrow 0}(\y(t),\boldsymbol{\phi})=(\y(0),\boldsymbol{\phi})=(\y_0,\boldsymbol{\phi}),
\end{align}
	and hence the initial data is also satisfied in the weak sense.

				\vskip 2mm
			\noindent
		\textbf{Case 2.} For $d\in\{2,3\}$ with $r\in(3,\infty)$. In order to obtain the uniqueness of  limits given in \eqref{eqn-conv-eta}, we need to establish an estimate similar to \eqref{eqn-conv-unique}.  Since $\|\y_k^{\prime}\|_{\mathrm{L}^{\frac{r+1}{r}}(0,T;\V^{\prime}+\L^{\frac{r+1}{r}})}\leq C$ (see \eqref{eqn-uniform-bounds-1}), we find 
				\begin{align}\label{eqn-conv-unique-1}
				&	\|\overline{\y}_k-\y_k\|_{\mathrm{L}^{\frac{r+1}{r}}(0,T;\V^{\prime}+\L^{\frac{r+1}{r}})}^{\frac{r+1}{r}}\nonumber\\&=\int_0^T\|\overline{\y}_k(t)-\y_k(t)\|_{\V^{\prime}+\L^{\frac{r+1}{r}}}^{\frac{r+1}{r}}\d t=\sum_{i=1}^N\int_{(i-1)k}^{ik}\|\overline{\y}_k(t)-\y_k(t)\|_{\V^{\prime}+\L^{\frac{r+1}{r}}}^{\frac{r+1}{r}}\d t\nonumber\\&=\sum_{i=1}^N\int_{(i-1)k}^{ik}(t-ik)^{\frac{r+1}{r}}\bigg\|\frac{\y_{k,i}-\y_{k.i-1}}{k}\bigg\|_{\V^{\prime}+\L^{\frac{r+1}{r}}}^{\frac{r+1}{r}}\d t\leq\frac{rk^{2+\frac{1}{r}}}{2r+1}\|\y_k^{\prime}\|_{\mathrm{L}^{\frac{r+1}{r}}(0,T;\V^{\prime}+\L^{\frac{r+1}{r}})}^{\frac{r+1}{r}}\nonumber\\&\leq Ck^{2+\frac{1}{r}}\to 0\ \text{ as }\ k\to 0.
			\end{align}
			Therefore $\overline{\y}_k-\y_k\to 0$ in $\mathrm{L}^{\frac{r+1}{r}}(0,T;\V^{\prime}+\L^{\frac{r+1}{r}})$ as $k\to 0$. Since the embedding $\mathrm{L}^2(0,T;\V)\hookrightarrow\mathrm{L}^{\frac{r+1}{r}}(0,T;\V^{\prime}+\L^{\frac{r+1}{r}})$ is continuous, we also have $\overline{\y}_k-\y_k\xrightarrow{w}\overline{\y}-\y$ in $\mathrm{L}^{\frac{r+1}{r}}(0,T;\V^{\prime}+\L^{\frac{r+1}{r}})$. Thus, we deduce $\overline{\y}=\y$. Since $\y_k\xrightarrow{w} \y\ \text{ in }\ \mathrm{L}^{2}(0,T;\V)\cap\mathrm{L}^{r+1}(0,T;\L^{r+1})$ and ${\y}_k'\xrightarrow{w} {\y}_1\ \text{ in }\ \mathrm{L}^2(0,T;\V^{\prime})\cap\mathrm{L}^{\frac{r+1}{r}}(0,T;\L^{\frac{r+1}{r}}),$ we conclude from \cite[Chapter I, Proposition1.2]{JHMMPD} that $\y_1=\y'$. 
			
			Furthermore, since
			\begin{align*}
				\y_k\xrightarrow{w}\y&\in\mathrm{L}^2(0,T;\V)\cap\mathrm{L}^{r+1}(0,T;\L^{r+1}), \\  \text{ with }
				\y_k^{\prime}\xrightarrow{w}\y^{\prime}&\in\mathrm{L}^2(0,T;\V^{\prime})\cap\mathrm{L}^{\frac{r+1}{r}}(0,T;\L^{\frac{r+1}{r}})
			\end{align*} 
			and by the continuity (hence also weak continuity)  of the mapping $\y\mapsto\y(0):\mathcal{W}\to\H$ (see Theorem \ref{Thm-Abs-cont}), we have 
				\begin{align*}
					\y_{k}(0)&\xrightarrow{w}\y(0) \ \text{ in } \ \H,\\  \text{ and }
					\y_0\leftarrow \y_{k,0}&=\y_k(0)\xrightarrow{w}\y(0)
				\end{align*}
			Therefore, the initial condition 	\begin{align}\label{eqn-initial-data}\y(0)=\y_0\ \text{ in }\ \H 	\end{align} is satisfied. From \eqref{eqn-inclusion-new}, one can easily seen that 
			 \begin{equation}\label{eqn-inclus}
				\left\{
				\begin{aligned}
					&	(\y_k^{\prime}(t),\z)+\langle\mathcal{F}(\overline{\y}_k(t)),\z\rangle+{}_{\U^{\prime}}\langle\overline{\boldsymbol{\eta}}_k(t),\ell\z\rangle_{\U}=\langle\f_k(t),\z\rangle,\\&\overline{\boldsymbol{\eta}}_k(t)\in\partial\psi(\ell\overline{\y}_k(t)),
				\end{aligned}
				\right.
			\end{equation}
			where $\mathcal{F}$ is defined in \eqref{eqn-op-f}. Therefore, for any $\z\in\mathrm{L}^2(0,T;\V)\cap\mathrm{L}^{r+1}(0,T;\L^{r+1}),$ one has 
			\begin{align}
				\int_0^T\langle\y_k^{\prime}(t),\z(t)\rangle\d t+\int_0^T\langle\mathcal{F}(\overline{\y}_k(t)),\z(t)\rangle\d t+\int_0^T{}_{\U^{\prime}}\langle\overline{\boldsymbol{\eta}}_k(t),\ell\z(t)\rangle_{\U}\d t=\int_0^T\langle\f_k(t),\z(t)\rangle\d t.
			\end{align}
			If $\langle\cdot,\cdot\rangle$ denotes the duality pairing between $\mathrm{L}^2(0,T;\V)\cap\mathrm{L}^{r+1}(0,T;\L^{r+1})$ and $\mathrm{L}^2(0,T;\V^{\prime})+\mathrm{L}^{\frac{r+1}{r}}(0,T;\L^{\frac{r+1}{r}})$, the above equation can be re-written as 
			\begin{align}\label{eqn-approx-weak}
				\langle\y_k^{\prime},\z\rangle+\langle\mathscr{F}(\overline{\y}_k),\z\rangle+{}_{\mathcal{U}^{\prime}}\langle\overline{\boldsymbol{\eta}}_k,\ell\z\rangle_{\mathcal{U}}=\langle\f_k,\z\rangle,
				\end{align}
				where $\mathscr{F}(\cdot)$ is defined in \eqref{eqn-nem-f}. Taking $\z-\overline{\y}_k$ instead of $\z$ in \eqref{eqn-approx-weak}, we find 
				\begin{align}\label{eqn-final-conv}
					\langle\mathscr{F}(\overline{\y}_k),\z-\overline{\y}_k\rangle&=\langle\f_k,\z-\overline{\y}_k\rangle-{}_{\mathcal{U}^{\prime}}\langle\overline{\boldsymbol{\eta}}_k,\ell(\z-\overline{\y}_k)\rangle_{\mathcal{U}}-	\langle\y_k^{\prime},\z-\overline{\y}_k\rangle\nonumber\\&=:I_k^1-I_k^2-I_k^3.
				\end{align}
				Since $\f_k\to\f$ in $\mathrm{L}^2(0,T;\V^{\prime})$ and $\overline{\y}_k\xrightarrow{w}\y\in\mathrm{L}^2(0,T;\V)$ as $k\to 0$, we immediately have 
				\begin{align}\label{eqn-final-conv-1}
					\lim_{k\to 0}I_k^1=\langle\f,\z-{\y}\rangle. 
				\end{align}
				From Hypothesis \ref{hyp-new-ell} (H4$'$), we know that the Nemytskii operator $\overline{\ell}:\mathcal{W}\to\U$ defined by $(\overline{\ell}\z)(t) = \ell\z(t)$ is compact. Since  $\overline{\y}_k\xrightarrow{w}\y$ in $\mathrm{L}^2(0,T;\V)$ as $k\to 0$ and $\overline{\ell}\overline{\y}_k\to \overline{\ell}\y$ in $\mathrm{L}^2(0,T;\U)$ as $k\to 0$. Together with this fact and the convergence $\overline{\boldsymbol{\eta}}_k\xrightarrow{w}\boldsymbol{\eta}\ \text{ in }\ \mathrm{L}^2(0,T;\U^{\prime})$ as $k\to0$ given in \eqref{eqn-conv-eta}, we have 
					\begin{align}\label{eqn-final-conv-2}
					\lim_{k\to 0}I_k^2={}_{\mathcal{U}^{\prime}}\langle\boldsymbol{\eta},\ell(\z-{\y})\rangle_{\mathcal{U}}.
				\end{align}
			 Since $\y\in\mathcal{W}$, it follows from Theorem \ref{Thm-Abs-cont} that the absolute continuity of $\|\cdot\|_{\H}^2$ holds true (see \cite[Theorem 3.5]{SGMTM} for an independent proof of energy equality). The convergences given in \eqref{eqn-conv-eta}  and \eqref{eqn-conv-time-der} yield the weak continuity of the mapping $\y\mapsto \y(T) :\mathcal{W}\to\H$. Using \eqref{eq-dual} and the weak lower semicontinuity of $\|\cdot\|_{\H}^2$, we deduce
				\begin{align}\label{eqn-final-conv-3}
					\limsup\limits_{k\to 0}I_k^3&=\limsup\limits_{k\to 0}\left(	\langle\y_k^{\prime},\z\rangle-\langle\y_k^{\prime},\overline{\y}_k\rangle\right)\nonumber\\&\leq \limsup\limits_{k\to 0}\left(	\langle\y_k^{\prime},\z\rangle-\frac{1}{2} (\|\y_{k,N}\|_{\H}^2+\|\y_{k,0}\|_{\H}^2)\right)				
					\nonumber\\&\leq \limsup\limits_{k\to 0}\left(	\langle\y_k^{\prime},\z\rangle-\frac{1}{2}\|\y_k(T)\|_{\H}^2+\frac{1}{2}\|\y_{k,0}\|_{\H}^2\right)\nonumber\\&\leq \lim_{k\to0}\langle\y_k^{\prime},\z\rangle-\frac{1}{2}\liminf_{k\to 0}\|\y_k(T)\|_{\H}^2+\frac{1}{2}\lim_{k\to 0}\|\y_{k,0}\|_{\H}^2\nonumber\\&\leq \langle\y^{\prime},\z\rangle-\frac{1}{2}\|\y(T)\|_{\H}^2+\frac{1}{2}\|\y_{0}\|_{\H}^2\nonumber\\&=\langle\y',\z-\y\rangle,
				\end{align}
				where in the last identity, we have used \eqref{eqn-absolute} and \eqref{eqn-initial-data}. Combining \eqref{eqn-final-conv-1}, \eqref{eqn-final-conv-2} and \eqref{eqn-final-conv-3}, and substituting it in \eqref{eqn-final-conv}, we arrive at 
				\begin{align}\label{eqn-final-conv-4}
					\liminf_{k\to 0}	\langle\mathscr{F}(\overline{\y}_k),\z-\overline{\y}_k\rangle\geq \langle\f-\y^{\prime},\z-\y\rangle-{}_{\mathcal{U}^{\prime}}\langle\boldsymbol{\eta},\ell(\z-{\y})\rangle_{\mathcal{U}}.
				\end{align}
				In particular, for $\z=\y$, we further have 
					\begin{align*}
					\liminf_{k\to 0}	\langle\mathscr{F}(\overline{\y}_k),\z-\overline{\y}_k\rangle \geq 0. 
				\end{align*}
			By Proposition \ref{prop-pseudo-time}, that is, the pseudomonotonicity of $\mathscr{F}$, one can conclude that, for any $\z\in\mathrm{L}^2(0,T;\V)\cap\mathrm{L}^{r+1}(0,T;\L^{r+1})$
				\begin{align}\label{eqn-final-conv-5}
				\liminf_{k\to 0}	\langle\mathscr{F}(\overline{\y}_k),\overline{\y}_k-\z\rangle\geq \langle\mathscr{F}({\y}),{\y}-\z\rangle.
			\end{align}
			Combining \eqref{eqn-final-conv-4} and  \eqref{eqn-final-conv-5}, and the fact that limit infimum is less than limit supremum, we arrive at 
			\begin{align}
				\langle\mathscr{F}({\y}),{\y}-\z\rangle\leq \langle\f-\y^{\prime},\y-\z\rangle-{}_{\mathcal{U}^{\prime}}\langle\boldsymbol{\eta},\ell(\y-\z)\rangle_{\mathcal{U}}.
			\end{align}
			As  the above estimate holds true  for any $\z\in\mathrm{L}^2(0,T;\V)\cap\mathrm{L}^{r+1}(0,T;\L^{r+1})$, one can take $\z=\y-\theta\w$, where $\theta>0$ and $\w\in\mathrm{L}^2(0,T;\V)\cap\mathrm{L}^{r+1}(0,T;\L^{r+1})$ to find 
				\begin{align*}
				\langle\mathscr{F}({\y}),\w\rangle\leq \langle\f-\y^{\prime},\w\rangle-{}_{\mathcal{U}^{\prime}}\langle\boldsymbol{\eta},\ell\w\rangle_{\mathcal{U}}.
			\end{align*}
			Since $\w\in\mathrm{L}^2(0,T;\V)\cap\mathrm{L}^{r+1}(0,T;\L^{r+1})$ is arbitrary, we finally have 
			\begin{align*}
				\langle\y^{\prime},\w\rangle+	\langle\mathscr{F}({\y}),\w\rangle+{}_{\mathcal{U}^{\prime}}\langle\boldsymbol{\eta},\ell\w\rangle_{\mathcal{U}}=\langle\f,\w\rangle,
			\end{align*}
			and the required result follows.
	\end{proof}
	
	\begin{remark}
1. 	In the supercritical case, the main difficulty in using the method adapted for the subcritical case   lies in establishing an estimate similar  \eqref{eqn-bv-est}, since for $r>3$, $\y_k^{\prime}\in\mathrm{L}^{\frac{r+1}{r}}(0,T;\V^{\prime}+\L^{\frac{r+1}{r}})$.

		2. We point out here that the strong convergence given in \eqref{eqn-conv-storng} is not used in the case for $d\in\{2,3\}$ with $r\in(3,\infty)$.
		
	\end{remark}

	\subsection{Energy equality and uniqueness}
The aim of this subsection is to discuss the energy equality satisfied by $\y(\cdot)$ and uniqueness of weak solutions for the cases $d=2$ with $r\in[1,3]$, $d=r=3$ and $d\in\{2,3\}$ with $r\in(3,\infty)$. 

\begin{lemma}\label{lem-ener-eq}
	For the cases $d=2$ with $r\in[1,3],$ and $d=r=3$, and $\mu> C_{\psi}\|\ell\|_{\mathcal{L}(\V;\U)}^2,$ let  Hypothesis \ref{hyp-psi-ell} (H1), (H2) and Hypothesis \ref{hyp-new-ell} (H4$'$) be satisfied. Assume that $\y_0\in\H$ and $\f\in\mathrm{L}^2(0,T;\V^{\prime})$.  Then  the weak solution $(\y,\boldsymbol{\eta})$ of  the Problem \ref{eqn-abstract-4} obtained in Theorem \ref{thm-main-non-station} satisfies the following energy equality:
\begin{align}\label{ener-eq}
	&	\|\y(t)\|_{\H}^2+2\mu\int_0^t\|\y(s)\|_{\V}^2\d s+2\alpha\int_0^t\|\y(s)\|_{\H}^2\d s+2\beta\int_0^t\|\y(s)\|_{\L^{r+1}}^{r+1}\d s\nonumber\\&= \|\y_0\|_{\H}^2-2\int_0^t{}_{\U^{\prime}}\langle{\boldsymbol{\eta}}(s),\ell\y(s)\rangle_{\U}\d s+2\int_0^t\langle\f(s),\y(s)\rangle\d s,
\end{align}

for all $t\in[0,T]$. Moreover, the following energy estimate is satisfied: 
\begin{align}\label{ener-est}
	&\sup_{t\in[0,T]}	\|\y(t)\|_{\H}^2+2\mu\int_0^T\|\y(t)\|_{\V}^2\d t+2\beta\int_0^T\|\y(t)\|_{\L^{r+1}}^{r+1}\d t\nonumber\\&\leq \|\y_0\|_{\H}^2+\frac{2}{\wi\mu}\int_0^T\|\f(t)\|_{\V^{\prime}}^2\d t+\frac{2T}{\wi\mu} C_{\psi}^2\|\ell\|_{\mathcal{L}(\V;\U)}^2,
\end{align}
where $\wi\mu$ is given in \eqref{eqn-wi-mu}.
\end{lemma}
\begin{proof}
From Theorem  \ref{thm-main-non-station}, we infer that $\y\in\mathrm{L}^{\infty}(0,T;\H)\cap\mathrm{L}^2(0,T;\V)\cap\mathrm{L}^{r+1}(0,T;\L^{r+1})$. 	For $d=2$ and $r\in[1,3]$, an application of Ladyzhenskaya's inequality yields 	
	\begin{align}
	\int_0^T\|\y(t)\|_{\L^4}^4\d t \leq 2\sup_{t\in[0,T]}\|\y(t)\|_{\H}^2\int_0^T\|\y(t)\|_{\V}^2\d t<\infty,
	\end{align}
so that $\y\in\mathrm{L}^4(0,T;\L^4)$.  Similarly for $d=r=3$ also $\y\in\mathrm{L}^4(0,T;\L^4)$.   In \cite[Theorem 4.2]{FHR}, the authors established an approximation of $\y(\cdot)$ in bounded domains such that the approximations are bounded and converge in both Sobolev and Lebesgue spaces simultaneously.  Using this result they have proved that the critical CBF equations satisfy the energy equality provided $\y\in\mathrm{L}^4(0,T;\L^4)$  (see \cite[Theorem 5.4]{FHR}, see  \cite[Theorem 3.5]{SGMTM} for supercritical case, see \cite[Theorem 3.1]{GGP} for 2D NSE). For the case $d\in\{2,3\}$ with $r\in(3,\infty)$, since $\y\in\mathrm{L}^2(0,T;\V)\cap\mathrm{L}^{r+1}(0,T;\L^{r+1})$ with $\y^{\prime}\in\mathrm{L}^2(0,T;\V^{\prime})\cap\mathrm{L}^{\frac{r+1}{r}}(0,T;\L^{\frac{r+1}{r}}),$ one can use Theorem \ref{Thm-Abs-cont} for the energy equality. 

Proceeding similarly as in the above mentioned works, one can verify that $\y(\cdot)$ satisfies the following energy equality: 
\begin{align*}
	\|\y(t)\|_{\H}^2&=\|\y_0\|_{\H}^2-2\int_0^t\langle\mu\A\y(s)+\B(\y(s))+\alpha\y(s)+\beta\mathfrak{C}(\y(s)),\y(s)\rangle\d s\nonumber\\&\quad-2\int_0^t{}_{\y^{\prime}}\langle{\boldsymbol{\eta}}(s),\ell \y(s)\rangle_{\U}\d s+2\int_0^t\langle\f(s),\y(s)\rangle\d s,
\end{align*}
for all $t\in[0,T]$. Thus, \eqref{ener-eq} follows. 

Let us now prove \eqref{ener-est}. A calculations similar to \eqref{eqn-ener-rothe-2} and  \eqref{eqn-ener-rothe-3}  yield
\begin{align*}
	|\langle\f,\y\rangle|&\leq\|\f\|_{\V^{\prime}}\|\y\|_{\V}\leq\frac{\wi\mu}{4}\|\y\|_{\V}^2+\frac{1}{\wi\mu}\|\f\|_{\V^{\prime}}^2,\\
|{}_{\U^{\prime}}\langle\boldsymbol{\eta},\ell \y\rangle_{\U}|&\leq C_{\psi}\|\ell\|_{\mathcal{L}(\V;\U)}^2\|\y\|_{\V}^2+\frac{\wi\mu}{4}\|\y\|_{\V}^2+\frac{1}{\wi\mu} C_{\psi}^2\|\ell\|_{\mathcal{L}(\V;\U)}^2,
	\end{align*}
	where $\wi\mu$ is defined in \eqref{eqn-wi-mu}. Using the above estimates in \eqref{ener-eq}, we deduce
	\begin{align}\label{ener-eq-1}
		&	\|\y(t)\|_{\H}^2+2\wi\mu\int_0^t\|\y(s)\|_{\V}^2\d s+2\alpha\int_0^t\|\y(s)\|_{\H}^2\d s+2\beta\int_0^t\|\y(s)\|_{\L^{r+1}}^{r+1}\d s\nonumber\\&\leq  \|\y_0\|_{\H}^2+\frac{2}{\wi\mu}\int_0^T\|\f(t)\|_{\V^{\prime}}^2\d t+\frac{2T}{\wi\mu} C_{\psi}^2\|\ell\|_{\mathcal{L}(\V;\U)}^2,
	\end{align}
	for all $t\in[0,T]$, and thus the estimate \eqref{ener-est} follows. 
\end{proof}

\begin{remark}
	1. As $\y\in\C([0,T];\H)$, from \eqref{eqn-initial}, we immediately have 
	\begin{align*}
		\lim_{t\downarrow 0}\|\y(t)-\y_0\|_{\H}=0.
	\end{align*}
	
2. We cannot directly apply Lions-Magenes lemma (\cite{JLEM}, \cite[Lemma 1.2, Chapter 3]{Te}), since $\y^{\prime}\in \mathrm{L}^{q'}(0,T;\V^{\prime})$ and  $\frac{1}{2}+\frac{1}{q^{\prime}}>1$ (for $d=2$ and $r\in(2,3]$, and $d=3$ and $r\in[1,3]$). 

3. As in the case of 3D NSE, for $d=3$ and $r\in[1,3)$, the  energy equality satisfied by $\y$ is an open problem. 
\end{remark} 

\subsubsection{Continuous dependence and uniqueness of the solution $\y(\cdot)$} Next, we provide a continuous dependence result for $d=2$ with $r\in[1,3],$ $d\in \{2,3\} $ with $r\in(3,\infty)$, and $d=r=3$ with $\mu>\max\left\{\frac{1}{2\beta},C_{\psi}\|\ell\|_{\mathcal{L}(\V;\U)}^2\right\}$, from which the uniqueness of weak solutions is immediate.

	\begin{theorem}\label{thm-unique-1}
		For the cases $d=2$ with $r\in[1,3],$ $d\in \{2,3\}$ with $r\in(3,\infty)$, and $$\mu> \max\{C_{\psi},m_1\}\|\ell\|_{\mathcal{L}(\V;\U)}^2,$$ and for the case  $d=r=3$, with $$\mu>\max\left\{\frac{1}{2\beta}+m_1\|\ell\|_{\mathcal{L}(\V;\U)}^2, C_{\psi}\|\ell\|_{\mathcal{L}(\V;\U)}^2\right\},$$  let  Hypothesis \ref{hyp-psi-ell} (H1), (H2), (H3) and Hypothesis \ref{hyp-new-ell} (H4$'$) be satisfied. Assume that $\y_0\in\H$ and $\f\in\mathrm{L}^2(0,T;\V^{\prime})$, then the solution $\y$  to  the Problem \ref{eqn-abstract-4} obtained in Theorem \ref{thm-main-non-station} is unique. Moreover,  the mapping $$\mathrm{L}^2(0,T;\V^{\prime})\times\H \ni(\f,\y_0)\mapsto\y\in\mathrm{C}([0,T];\H)\cap\mathrm{L}^2(0,T;\V)\cap\mathrm{L}^{r+1}(0,T;\L^{r+1})$$ is Lipschitz continuous. 
	\end{theorem}
	\begin{proof}
		Let $(\y_1,\boldsymbol{\eta}_1),(\y_2,\boldsymbol{\eta}_1)\in\mathcal{W}\times\mathrm{L}^2(0,T;\U^{\prime})$ be two solutions of the Problem \ref{eqn-abstract-4} corresponding to two  external forces $\f_1,\f_2\in\mathrm{L}^2(0,T;\V^{\prime})$  and two initial conditions $\y_0^1,\y_0^2\in\H$, that is,  $(\y_i,\boldsymbol{\eta}_i)$ for $i=1,2$ satisfies the following: 
		 for all $\z\in \V\cap\mathrm{L}^{r+1}$
		\begin{align}
			\langle\y_i'(t)+\mu\A\y_i(t)+\B(\y_i(t))+\alpha\y(t)+\beta\mathfrak{C}(\y_i(t)),\z\rangle+{}_{\U^{\prime}}\langle\boldsymbol{\eta}_i(t),\ell\z\rangle_{\U}=\langle\f_i(t),\z\rangle, 
		\end{align}
	for a.e. $t\in(0,T)$, 	$	\boldsymbol{\eta}_i(t)\in \partial \psi(\ell\y_i(t)),$ for   a.e. $t\in(0,T)$. Then $(\w,\boldsymbol{\xi})=(\y_1-\y_2,\boldsymbol{\eta}_1-\boldsymbol{\eta}_2)$ satisfies the following: 
	 for all $ \z\in \V\cap\mathrm{L}^{r+1}$
	\begin{align}
		&\langle\w'(t)+\mu\A\w(t)+[\B(\y_1(t))-\B(\y_2(t))]+\alpha\w(t)+\beta[\mathfrak{C}(\y_1(t))-\mathfrak{C}(\y_2(t))],\z\rangle\nonumber\\&\quad+{}_{\U^{\prime}}\langle\boldsymbol{\xi}(t),\ell\z\rangle_{\U}=\langle\g(t),\z\rangle, 
	\end{align}
	for a.e. $t\in(0,T)$, $	\boldsymbol{\eta}_i(t)\in \partial \psi(\ell\y_i(t)),$ for   a.e. $t\in(0,T)$, where $\g=\f_1-\f_2\in\mathrm{L}^2(0,T;\V^{\prime})$ and $\w(0)=\y_0^1-\y_0^2\in\H$.

		\vskip 2mm
	\noindent
	\textbf{Case 1.} 
	For the case $d=2$ with $r\in[1,3],$ and $\mu> C_{\psi}\|\ell\|_{\mathcal{L}(\V;\U)}^2,$  from Lemma \ref{lem-ener-eq}, we infer that $\w$ satisfies the energy equality, that is
	\begin{align}\label{eqn-unique-non}
	&	\|\w(t)\|_{\H}^2+2\mu\int_0^t\|\w(s)\|_{\V}^2\d s+2\alpha\int_0^t\|\w(s)\|_{\H}^2\d s+2\beta\int_0^t\langle\mathfrak{C}(\y_1(s))-\mathfrak{C}(\y_2(s)),\w(s)\rangle\d s\nonumber\\&=\|\w(0)\|_{\H}^2-2\int_0^t\langle\B(\y_1(s))-\B(\y_2(s)),\w(s)\rangle\d s-2\int_0^t{}_{\U^{\prime}}\langle{\boldsymbol{\xi}}(s),\ell\w(s)\rangle_{\U}\d s\nonumber\\&\quad+2\int_0^t\langle\g(s),\w(s)\rangle\d s,
	\end{align}
	 for all $t\in[0,T]$. We infer from the estimate \eqref{Eqn-mon-lip} that 
	\begin{align}\label{eqn-unique-non-1}
		2\beta\langle\mathfrak{C}(\y_1)-\mathfrak{C}(\y_2),\w\rangle\geq\frac{\beta}{2^{r-2}}\|\w\|_{\L^{r+1}}^{r+1}.
	\end{align}
	Using a calculation similar to the estimate \eqref{2.21}, we arrive at 
	\begin{align}\label{eqn-unique-non-2}
	2|\langle\B(\y_1)-\B(\y_2),\w\rangle|\leq\frac{\widehat{\mu} }{2}\|\w\|_{\V}^2+\frac{C}{\widehat{\mu}^{3}} \|\y_2\|_{\L^4}^{4}\|\w\|_{\H}^2,
	\end{align}
where $\widehat{\mu}=\mu-m_1\|\ell\|_{\mathcal{L}(\V;\U)}^2>0$.  Using the Cauchy-Schwarz and Young's inequalities, we obtain 
\begin{align}\label{eqn-unique-non-3}
	2|\langle\g,\w\rangle|\leq 2\|\g\|_{\V^{\prime}}\|\w\|_{\V}\leq \frac{\widehat{\mu} }{2}\|\w\|_{\V}^2+\frac{2}{\widehat{\mu}}\|\g\|_{\V^{\prime}}^2. 
\end{align}
	Applying Hypothesis \ref{hyp-psi-ell} (H3), we further have 
	\begin{align}\label{eqn-unique-non-4}
-	2{}_{\U^{\prime}}\langle{\boldsymbol{\xi}},\ell\w\rangle_{\U}\leq 2m_1\|\ell\w\|_{\U}^2\leq 2m_1\|\ell\|_{\mathcal{L}(\V;\U)}^2\|\w\|_{\V}^2. 
\end{align}
Combining \eqref{eqn-unique-non-1}-\eqref{eqn-unique-non-4} and substituting it in \eqref{eqn-unique-non}, we deduce 
\begin{align}\label{eqn-unique-non-5}
	&	\|\w(t)\|_{\H}^2+\widehat{\mu}\int_0^t\|\w(s)\|_{\V}^2\d s+\frac{\beta}{2^{r-2}}\int_0^t\|\w(s)\|_{\L^{r+1}}^{r+1}\d s\nonumber\\&\leq\|\w(0)\|_{\H}^2+\frac{2}{\widehat{\mu}}\int_0^t\|\g(s)\|_{\V^{\prime}}^2\d s+\frac{C}{\widehat{\mu}^{3}} \int_0^t\|\y_2(s)\|_{\L^4}^{4}\|\w(s)\|_{\H}^2\d s.
\end{align}
An application of Gronwall's inequality in \eqref{eqn-unique-non-5} yields 
\begin{align}\label{eqn-unique-non-6}
	&	\|\y_1(t)-\y_2(t)\|_{\H}^2+\widehat{\mu}\int_0^t\|\y_1(s)-\y_2(s)\|_{\V}^2\d s+\frac{\beta}{2^{r-2}}\int_0^t\|\y_1(s)-\y_2(s)\|_{\L^{r+1}}^{r+1}\d s\nonumber\\&\leq\left\{\|\y_1(0)-\y_2(0)\|_{\H}^2+\frac{2}{\widehat{\mu}}\int_0^t\|\f_1(s)-\f_2(s)\|_{\V^{\prime}}^2\d s\right\}\exp\left(\frac{C}{\widehat{\mu}^{3}} \int_0^t\|\y_2(s)\|_{\L^4}^{4}\d s\right)\nonumber\\&\leq \left\{\|\y_1(0)-\y_2(0)\|_{\H}^2+\frac{2}{\widehat{\mu}}\int_0^T\|\f_1(t)-\f_2(t)\|_{\V^{\prime}}^2\d t\right\}\nonumber\\&\quad\times \exp\bigg\{\frac{C}{\widehat{\mu}^{3}}\left[ \|\y_0\|_{\H}^2+\frac{2}{\wi\mu}\int_0^T\|\f(t)\|_{\V^{\prime}}^2\d t+\frac{2T}{\wi\mu} C_{\psi}^2\|\ell\|_{\mathcal{L}(\V;\U)}^2\right]^2\bigg\},
\end{align}
for all $t\in[0,T]$, where we have used Ladyzhenskaya's inequality as well as the estimate \eqref{ener-est} also. The estimate \eqref{eqn-unique-non-6} provides both uniqueness and continuous dependence results.

	\vskip 2mm
\noindent
\textbf{Case 2.} For the case $d=r=3$, with $\mu>\max\left\{\frac{1}{2\beta}+m_1\|\ell\|_{\mathcal{L}(\V;\U)}^2, C_{\psi}\|\ell\|_{\mathcal{L}(\V;\U)}^2\right\}$ also, the energy equality \eqref{eqn-unique-non} is valid. We infer from \eqref{2.23} that 
\begin{align}\label{eqn-unique-non-7}
2\beta\langle\mathfrak{C}(\y_1)-\mathfrak{C}(\y_2),\w\rangle\geq\beta  \||\y_1|\w\|_{\H}^2+\beta\||\y_2|\w\|_{\H}^2
	\end{align}
	A calculation similar to \eqref{232} yields
	\begin{align}\label{eqn-unique-non-8}
	2	|\langle\B(\y_1)-\B(\y_2),\w\rangle|\leq 	\frac{1}{\beta} \|\w\|_{\V}^2+\beta\|\y_2\w\|_{\H}^2.
	\end{align}
	Using the Cauchy-Schwarz and Young's inequalities, we obtain 
	\begin{align}\label{eqn-unique-non-9}
		2|\langle\g,\w\rangle|\leq 2\|\g\|_{\V^{\prime}\|\w\|_{\V}}\leq \frac{\overline{\mu}}{2}\|\w\|_{\V}^2+\frac{2}{\overline{\mu}}\|\g\|_{\V^{\prime}}^2,
	\end{align}
	where $\overline{\mu}=\left(\mu-\frac{1}{2\beta}-m_1\|\ell\|_{\mathcal{L}(\V;\U)}^2\right)>0$. 
	Combining \eqref{eqn-unique-non-4}, \eqref{eqn-unique-non-7}, \eqref{eqn-unique-non-8} and \eqref{eqn-unique-non-9}, and  then substituting it in \eqref{eqn-unique-non}, we arrive at 
	\begin{align}
		&	\|\w(t)\|_{\H}^2+\left(\mu-\frac{1}{2\beta}-m_1\|\ell\|_{\mathcal{L}(\V;\U)}^2\right)\int_0^t\|\w(s)\|_{\V}^2\d s\leq\|\w(0)\|_{\H}^2+\frac{2}{\overline{\mu}}\int_0^t\|\g(s)\|_{\V^{\prime}}^2\d s,
	\end{align}
	for all $t\in[0,T]$ and the required result follows for $\mu>\frac{1}{2\beta}+m_1\|\ell\|_{\mathcal{L}(\V;\U)}^2$.
	
		\vskip 2mm
	\noindent
\textbf{Case 3.} For the case $d\in\{2,3\}$ with $r\in(3,\infty)$ and $\mu> C_{\psi}\|\ell\|_{\mathcal{L}(\V;\U)}^2,$ we first estimate the term $	2\beta\langle\mathfrak{C}(\y_1)-\mathfrak{C}(\y_2),\w\rangle$  using \eqref{eqn-ces} as 
	\begin{align}\label{eqn-unique-non-11}
			2\beta\langle\mathfrak{C}(\y_1)-\mathfrak{C}(\y_2),\w\rangle\geq\beta \||\y_1|^{\frac{r-1}{2}}\w\|_{\H}^2+\beta\||\y_2|^{\frac{r-1}{2}}\w\|_{\H}^2. 
	\end{align}
	Let us now estimate the term $2	|\langle\B(\y_1)-\B(\y_2),\w\rangle|$. An estimate similar to \eqref{eqn-bes} yields 
	\begin{align}\label{eqn-unique-non-12}
		2	|\langle\B(\y_1)-\B(\y_2),\w\rangle|\leq \frac{\widehat{\mu} }{2}\|\w\|_{\V}^2+\frac{\beta}{2}\||\y_2|^{\frac{r-1}{2}}\w\|_{\H}^2+\frac{2\wi\rho}{\widehat{\mu}}\|\w\|_{\H}^2,
	\end{align}
	where $\wi\rho=	\frac{r-3}{r-1}\left(\frac{8}{\beta\widehat{\mu} (r-1)}\right)^{\frac{2}{r-3}}.$ Combining \eqref{eqn-unique-non-11}, \eqref{eqn-unique-non-12}, \eqref{eqn-unique-non-3} and  \eqref{eqn-unique-non-4}  substitute it in \eqref{eqn-unique-non}, we deduce for all $t\in[0,T]$
	\begin{align}\label{eqn-unique-non-13}
		&	\|\w(t)\|_{\H}^2+\widehat{\mu}\int_0^t\|\w(s)\|_{\V}^2\d s+\frac{\beta}{2^{r}}\int_0^t\|\w(s)\|_{\L^{r+1}}^{r+1}\d s\nonumber\\&\leq\|\w(0)\|_{\H}^2+\frac{2}{\widehat{\mu}}\int_0^t\|\g(s)\|_{\V^{\prime}}^2\d s+\frac{2\wi\rho}{\widehat{\mu}}\int_0^t\|\w(s)\|_{\H}^2\d s . 
	\end{align}
	An application of Gronwall's inequality in \eqref{eqn-unique-non-13} yields  for all $t\in[0,T]$
	\begin{align}
		&	\|\y_1(t)-\y_2(t)\|_{\H}^2+\widehat{\mu}\int_0^t\|\y_1(s)-\y_2(s)\|_{\V}^2\d s+\frac{\beta}{2^{r}}\int_0^t\|\y_1(s)-\y_2(s)\|_{\L^{r+1}}^{r+1}\d s\nonumber\\&\leq \left\{\|\y_1(0)-\y_2(0)\|_{\H}^2+\frac{2}{\widehat{\mu}}\int_0^T\|\f_1(t)-\f_2(t)\|_{\V^{\prime}}^2\d t\right\}e^{\frac{2\wi\rho T}{\widehat{\mu}}}, 
	\end{align}
	so that the uniqueness and continuous dependence result follows. 
	\end{proof}

	Let us now discuss the uniqueness result for  $d=r=3$ with $\mu>\max\{C_{\psi},m_1\}\|\ell\|_{\mathcal{L}(\V;\U)}^2$. Here the condition on $\mu$ is independent of $\beta$. In order to do this, we adopt the ideas from \cite[Theorem 3.9]{SGMTM}. 
	\begin{theorem}\label{thm-unique-2}
		For $d=r=3$, with $\mu>\max\left\{m_1, C_{\psi}\right\}\|\ell\|_{\mathcal{L}(\V;\U)}^2$,  let  Hypothesis \ref{hyp-psi-ell} (H1), (H2), (H3) and Hypothesis \ref{hyp-new-ell} (H4$'$) be satisfied. Assume that $\y_0\in\H$ and $\f\in\mathrm{L}^2(0,T;\V^{\prime})$, then the solution $\y$  to  the Problem \ref{eqn-abstract-4} obtained in Theorem \ref{thm-main-non-station} is unique. 
	\end{theorem}
	
	\begin{proof}
		Let $(\y,\boldsymbol{\eta})\in \mathrm{L}^{\infty}(0,T;\H)\cap\mathrm{L}^2(0,T;\V)\cap\mathrm{L}^{4}(0,T;\L^4)\times\mathrm{L}^2(0,T;\U^{\prime})$ be a weak solution of the problem \eqref{eqn-abstract-4}. Let us now consider the following problem in $\V^{\prime}$ for a.e. $t\in(0,T)$: 
		\begin{equation}\label{eqn-abstract-5}
		\left\{
		\begin{aligned}
			&\z'(t)+\mu\A\z(t)+\B(\y(t),\z(t))+\alpha\z(t)+\beta\mathfrak{C}(\z(t))+\ell^*\boldsymbol{\zeta}(t)=\f(t),\\
			&\boldsymbol{\zeta}(t)\in \partial \psi(\ell\z(t)), \\
			&\z(0)=\y_0\in\H. 
		\end{aligned}
		\right.
	\end{equation}
		 Since $\langle\B(\y,\z),\z\rangle=0$, $\y_0\in\H$ and $\f\in\mathrm{L}^{2}(0,T;\V^{\prime})$, the existence of a unique weak solution $(\z,\boldsymbol{\zeta})\in \mathrm{L}^{\infty}(0,T;\H)\cap\mathrm{L}^2(0,T;\V)\cap\mathrm{L}^{4}(0,T;\L^4)\times\mathrm{L}^2(0,T;\U^{\prime})$ can be proved in a similar way as in the proof of Theorem \ref{thm-main-non-station}. Moreover, arguments similar to  the proof of Lemma \ref{lem-ener-eq} imply that $\z(\cdot)$ satisfies the following energy equality:  
		\begin{align}
		&	\|\z(t)\|_{\H}^2+2\mu\int_0^t\|\z(s)\|_{\V}^2\d s+2\alpha\int_0^t\|\z(s)\|_{\H}^2\d s+2\beta\int_0^t\|\z(s)\|_{\L^{r+1}}^{r+1}\d s\nonumber\\&= \|\y_0\|_{\H}^2-2\int_0^t{}_{\U^{\prime}}\langle{\boldsymbol{\zeta}}(s),\ell\z(s)\rangle_{\U}\d s+2\int_0^t\langle\f(s),\z(s)\rangle\d s,
		\end{align} 
		for all $t\in[0,T]$. Let us now show that the weak solution $(\z,\boldsymbol{\zeta})$ of the problem obtained above is unique. For any given data $\y_0\in\H$ and $\f\in\mathrm{L}^2(0,T;\V^{\prime})$, let us assume that $(\z_1,\boldsymbol{\zeta}_1)$ and $(\z_2,\boldsymbol{\zeta}_2)$ are two weak solutions of the problem \eqref{eqn-abstract-5}. Then $(\z,\boldsymbol{\zeta})=(\z_1-\z_2,\boldsymbol{\zeta}_1-\boldsymbol{\zeta}_2)$ and  satisfies the following for a.e. $t\in(0,T)$ in $\V^{\prime}$: 
	\begin{equation}\label{eqn-abstract-6}
		\left\{
		\begin{aligned}
			&\z'(t)+\mu\A\z(t)+\B(\y(t),\z(t))+\alpha\z(t)+\beta[\mathfrak{C}(\z_1(t))-\mathfrak{C}(\z_2(t))]+\ell^*\boldsymbol{\zeta}(t)=\boldsymbol{0},\\
			&\boldsymbol{\zeta}_i(t)\in \partial \psi(\ell\z_i(t)), \\
			&\z(0)=\boldsymbol{0},
		\end{aligned}
		\right.
	\end{equation} for $i=1,2$. 
	The solution $(\z,\boldsymbol{\zeta})$ to the problem \eqref{eqn-abstract-6} satisfies the following energy equality:   for all $t\in[0,T]$
		\begin{align}\label{eqn-unique-non-10}
		&	\|\z(t)\|_{\H}^2+2\mu\int_0^t\|\z(s)\|_{\z}^2\d s+2\alpha\int_0^t\|\z(s)\|_{\H}^2\d s+2\beta\int_0^t\langle\mathfrak{C}(\z_1(s))-\mathfrak{C}(\z_2(s)),\z(s)\rangle\d s\nonumber\\&=\|\z(0)\|_{\H}^2 -2\int_0^t{}_{\U^{\prime}}\langle{\boldsymbol{\zeta}}(s),\ell\z(s)\rangle_{\U}\d s,
	\end{align} 
		since $\langle\B(\y,\z),\z\rangle=0$.  Using the estimates \eqref{eqn-unique-non-1} and \eqref{eqn-unique-non-4} in \eqref{eqn-unique-non-10}, we deduce 
		\begin{align}\label{378}
			\|\z(t)\|_{\H}^2+\widehat{\mu}\int_0^t\|\z(s)\|_{\V}^2\d s+\frac{\beta}{2^{r-2}}\int_0^t\|\z(s)\|_{\L^{r+1}}^{r+1}\d s\leq\|\z(0)\|_{\H}^2=\boldsymbol{0},
		\end{align}
	where  $\widehat{\mu}=\mu-m_1\|\ell\|_{\mathcal{L}(\V;\U)}^2>0$. 	The above relation immediately gives $\z_1(t)=\z_2(t)$ for all $t\in[0,T]$ in $\H$. 
		
		Our next aim is to show that $\z(t,\x)=\y(t,\x)$ for all $t\in[0,T]$ and a.e. $\x\in\mathcal{O}$, where $\y$ is a weak solution of the problem \eqref{eqn-abstract-4} corresponding to the initial data $\y_0\in\H$ and forcing $\f \in \mathrm{L}^2(0,T;\V^{\prime})$, which is appearing in \eqref{eqn-abstract-5}. Let us define $\w=\y-\z$ and take the difference between \eqref{eqn-abstract-4} and \eqref{eqn-abstract-5} to find 
		\begin{equation}\label{eqn-abstract-7}
		\left\{
		\begin{aligned}
			&\w'(t)+\mu\A\w(t)+\B(\y(t),\w(t))+\alpha\w(t)+\beta[\mathfrak{C}(\y(t))-\mathfrak{C}(\z(t))]+\ell^*(\boldsymbol{\eta}(t)-\boldsymbol{\zeta}(t))=\boldsymbol{0},\\
			&\boldsymbol{\eta}(t)\in \partial \psi(\ell\y(t)),\ \boldsymbol{\zeta}(t)\in \partial \psi(\ell\z(t)), \\
			&\w(0)=\boldsymbol{0}.
		\end{aligned}
		\right.
	\end{equation}
		The problems \eqref{eqn-abstract-5} and \eqref{eqn-abstract-7} are the same and a calculation similar to \eqref{378} yields $\y(t)=\z(t)$ for all $t\in[0,T]$ in $\H$. The uniqueness of the weak solutions of the problem \eqref{eqn-abstract-5} yields the uniqueness of weak solutions of the problem \eqref{eqn-abstract-4} also. 
	\end{proof}

\subsection{A boundary hemivariational inequality}\label{sub-non-boundary}
We begin with the following evolutionary 2D and 3D CBF equations:
\begin{equation}\label{eqn-bon}
	\left\{
	\begin{aligned}
		\frac{\partial \y}{\partial t}-\mu \Delta\y+(\y\cdot\nabla)\y+\alpha\y+\beta|\y|^{r-1}\y+\nabla p&=\boldsymbol{f}, \ \text{ in } \ \mathcal{O}\times(0,T), \\ \nabla\cdot\y&=0, \ \text{ in } \ \mathcal{O}\times[0,T).
	\end{aligned}
	\right.
\end{equation}
We study a boundary hemivariational inequalities for the system \eqref{eqn-bon}. Let us first rewrite the first equation in \eqref{eqn-bon} in terms of the curl operator.  Using the vector identities \eqref{eqn-vector-1} and \eqref{eqn-vector-2}, one can re-write the problem \eqref{eqn-bon} as 
\begin{equation}\label{eqn-bon-1}
	\left\{
	\begin{aligned}
		\frac{\partial \y}{\partial t}-\mu\  \mathrm{curl\ }\mathrm{curl\ }\y+\mathrm{curl\ }\y\times\y+\alpha\y+\beta|\y|^{r-1}\y+\nabla q&=\boldsymbol{f}, \ \text{ in } \ \mathcal{O}\times(0,T), \\ \nabla\cdot\y&=0, \ \text{ in } \ \mathcal{O}\times[0,T),
	\end{aligned}
	\right.
\end{equation}
where, the dynamic pressure is $q(\x,t)=p(\x,t)+\frac12|\y(\x,t)|^2$. The CBF equations given in \eqref{eqn-bon-1} are supplemented by initial and boundary conditions. 

For the initial condition, we choose
\begin{align}
	\y(0)=\y_0, 
\end{align}
where $\y_0$ denotes a given initial value of $\y(t)=\y(\x,t)$. 

Let us now describe the associated boundary conditions. Let $\boldsymbol{n} = (n_1, \ldots, n_d)^{\top}$ denote the unit outward normal on the boundary $\Gamma$. For a vector field $\y$ defined on $\Gamma$, we write  
$$y_n = \y\cdot\boldsymbol{n},~ \y_{\tau}= \y - y_n\boldsymbol{n},$$ for its the normal and tangential components, respectively. We impose the following boundary conditions:
\begin{align}
	\y_{\tau}&=\boldsymbol{0}\ \text{ on }\ \Gamma\times(0,T),\label{eqn-bond-1}\\
	q(t)&\in\partial j(t,y_n(t))\ \text{ on }\ \Gamma\times(0,T),\label{eqn-bond-2}\
\end{align}
where $j(t,y_n(t))$ is a short-hand notation for $j(\x,t,y_n(\x,t))$. The mapping $j: \Gamma\times(0,T)\times\mathbb{R}^d\to\mathbb{R}$ is referred to as a superpotential and is assumed to be locally Lipschitz continuous with respect to its third argument. The operator $\partial j$ denotes the Clarke subdifferential of $j(\x,t,\cdot).$ As discussed in \cite{Fang2016}, the problem of fluid motion via a tube or channel gives rise to the boundary condition \eqref{eqn-bond-2}; the fluid pumped into $\mathcal{O}$ can flow out through openings at the boundary, and a device can adjust the size of these openings. Here, we regulate the normal component of the fluid velocity along the boundary so as to minimize the total pressure on
$\Gamma$. Different physical phenomena are thus described by different boundary conditions. We point out here that if $ j(\x, t, y_n)$ is convex in its third argument,  then the problem \eqref{eqn-bon-1}-\eqref{eqn-bond-2} leads to a \emph{variational inequality}. In our context, we do not assume the convexity of $j$ with respect to its third variable; thus, our problem corresponds to a  \emph{hemivariational inequality}.
\subsubsection{Abstract formulation} Multiplying the equation of motion \eqref{eqn-bon-1} by $\z \in \V\cap\L^{r+1}$ and applying  Green's formula (see \eqref{eqn-green-2}), we obtain  for  a.e. $t\in(0,T)$
\begin{equation}\label{eqn-abs-1}
	\left\{
\begin{aligned}
	&\langle\y'(t)+\mu\A\y(t)+\B(\y(t))+\alpha\y(t)+\beta\mathfrak{C}(\y(t)),\z\rangle+\int_{\Gamma}q(t)z_n\d\Gamma=\langle\f(t),\z\rangle,\\
	&\y(0)=\y_0,
\end{aligned}
\right.
\end{equation}
for all $\z\in\V\cap\L^{r+1}$. 

 From the relation \eqref{eqn-bond-2}, by using the definition of the Clarke subdifferential, we have
\begin{align}\label{eqn-clarke}
\int_{\Gamma}q(t)z_n\d\Gamma\leq \int_{\Gamma}j^0(t,y_n(t);z_n)\d\Gamma,
\end{align}
where $j^0(t,\xi;\zeta)\equiv j^0(\x,t,\xi;\zeta)$ denotes the generalized directional derivative of $j(\x,t,\cdot)$ at the point $\xi\in\R$ in the direction $\zeta\in\R$.

Let us denote $\mathbb{U}=\mathbb{L}^2(\Gamma)$ and $\ell:\V\to\mathbb{U}$. Moreover, we take $\f\in\mathrm{L}^2(0,T;\V^{\prime})$. 
The relations \eqref{eqn-abs-1} and \eqref{eqn-clarke} yield the following variational formulation: 
\begin{problem}\label{prob-inequality}
Find $\y\in\mathcal{W}$ such that  for  a.e. $t\in(0,T)$
\begin{equation}\label{eqn-abs-2}
	\left\{
	\begin{aligned}
		&\langle\y'(t)+\mu\A\y(t)+\B(\y(t))+\alpha\y(t)+\beta\mathfrak{C}(\y(t)),\z\rangle+\int_{\Gamma}j^0(t,y_n(t);z_n)\d\Gamma\geq \langle\f(t),\z\rangle,\\
		&\y(0)=\y_0,
	\end{aligned}
	\right.
\end{equation}
for all $\z\in\V\cap\L^{r+1}$. 
\end{problem}
We make the following assumptions on the superpotential $j$:
\begin{hypothesis}\label{hyp-j}
	The superpotential $j:\Gamma\times(0,T)\times\mathbb{R}\to\mathbb{R}$ satisfy 
	\begin{enumerate}
		\item [(H.1)] $j(\cdot,\cdot,\xi)$ is measurable on $\Gamma\times(0,T)$ for all $\xi\in\R$ and there exists an $e\in \mathbb{L}^2(\Gamma)$ such that $j(\cdot,\cdot,e(\cdot))\in  \mathrm{L}^1(\Gamma\times(0,T))$;
		\item [(H.2)] $j(\x,t,\cdot)$ is locally Lipschitz on $\R$ for a.e. $(\x,t) \in \Gamma\times(0,T)$;
		\item [(H.3)] $|\zeta|\leq C_0(1+|\xi|)$ for all $\zeta\in\partial j(\x,t,\xi)$, $\xi\in\mathbb{R}$ for a.e. $(\x,t)\in\Gamma\times(0,T)$ with $C_0>0$; 
		\item [(H.4)] $(\zeta_1-\zeta_2)\cdot(\xi_1-\xi_2)\geq -m|\xi_1-\xi_2|^2$ for all $\zeta_i\in \partial j(\x,t,\xi_i)$, $\xi_i\in\mathbb{R}$, $i=1,2,$ for a.e. $(\x,t)\in\Gamma\times(0,T)$ with $m\geq 0$. 
	\end{enumerate}
\end{hypothesis}
Let us define a functional $J:(0,T)\times\mathbb{U}\to\mathbb{R}$ by 
\begin{align}\label{eqn-J}
	J(t,\y)=\int_{\Gamma}j(\x,t,y_n(\x))\d\x,\ \y\in\mathbb{U}, \ \text{ a.e. }\ t\in(0,T).
\end{align}

\begin{lemma}
	Assume that $j:\Gamma\times(0,T)\times\mathbb{R}\to\mathbb{R}$ satisfies Hypothesis \ref{hyp-j}. Then the functional $J$ defined by \eqref{eqn-J} has the following properties: 
	\begin{enumerate}
		\item [(i)] $J(t,\cdot)$ is locally Lipschitz on $\mathbb{U}$ for a.e. $t \in (0,T)$;
		\item [(ii)] $\|\boldsymbol{\eta}\|_{\mathbb{U}}\leq C_1(1+\|\y\|_{\mathbb{U}})$ for all $\boldsymbol{\eta}\in \partial J(t,\y)$, $\y\in\mathbb{U}$ for a.e. $t\in(0,T)$ with $C_1>0$, where $C_1=\sqrt{2}C_0\max\{\sqrt{\mathrm{meas}(\Gamma)},1\}$, $\mathrm{meas}(\Gamma)$ being the Lebesgue measure of $\Gamma$;
		\item [(iii)] $J^0(t,\y;\z)\leq\int_{\mathcal{O}}j^0(t;y_n(\x);z_n(\x))\d\Gamma$, for all $\y,\z\in\mathbb{U}$ for a.e. $t \in (0,T)$;
		\item [(iv)] $(\u_1-\u_2,\y_1-\y_2)\geq -m\|\y_1-\y_2\|_{\mathbb{U}}^2$ for all $\u_i\in\partial J(t,\y_i)$, $\y_i\in\mathbb{U}$ for $i=1,2,$ for a.e. $t\in(0,T)$. 
	\end{enumerate}
\end{lemma}
\begin{proof}
Follows in the similar lines of the proof of Lemma \ref{lem-J-lem}.	
\end{proof}

The next result is motivated from \cite[Subsection 6.1]{Fang2016}. 
\begin{lemma}\label{lem-compact}
For $d\in\{2,3\}$ with $r\in[1,3]$, 	the operator $\ell\in\mathcal{ L}(\V,\mathbb{U}) $ is compact and its Nemytskii operator $\overline{\ell}: \mathcal{M}^{2,2}(0,T;\V,\V^{\prime}) \to\mathrm{L}^2(0,T;\mathbb{U})$ defined by $(\overline{\ell} \y)(t) = \ell\y(t)$ is compact.
\end{lemma}
\begin{proof}
Since $\U=\L^2(\Gamma)$, 	one can easily see that the operator $\ell:\V\to\U$  is linear, continuous and compact.
In order to prove its Nemytskii operator $\overline{\ell}: \mathcal{M}^{2,2}(0,T;\V,\V^{\prime}) \to\mathrm{L}^2(0,T;\mathbb{U})$ defined by $(\overline{\ell} \y)(t) = \ell\y(t)$ is compact, let us take $\{\y_n\}_{n\in\N}$ as a bounded sequence in $\mathcal{M}^{2,2}(0,T;\V,\V^{\prime})$. For $\delta \in \left(\frac{1}{2}, 1 \right)$, using the Aubin-Lions compactness lemma (Theorem \ref{thm-Simon}), we deduce that the embedding 
$
	\mathcal{M}^{2,2}(0,T;\V,\V^{\prime}) \subset \mathrm{L}^2(0,T;\mathbb{H}^{\delta}(\mathcal{O}))
$
	is compact. Therefore, there exists a subsequence $\{\y_{n_j}\}$ of $\{\y_n\}$ such that $\y_{n_j} \to \y$ in $\mathrm{L}^2(0,T;\mathbb{H}^{\delta}(\mathcal{O}))$ for some element $\y \in \mathrm{L}^2(0,T;\mathbb{H}^{\delta}(\mathcal{O}))$. We infer from the embedding 
$$
		\mathbb{H}^{\delta}(\mathcal{O}) \underbrace{\hookrightarrow}_{\text{continuous}} \mathbb{H}^{\delta-\frac{1}{2}}(\Gamma) \underbrace{\hookrightarrow}_{\text{compact}} \mathbb{L}^{2}(\Gamma),
$$
 that the embedding $	\mathbb{H}^{\delta}(\mathcal{O}) \hookrightarrow\U$  is compact, there exists a further subsequence of $\{\y_{n_j}\}$, still denoted by $\{\y_{n_j}\}$, such that $\y_{n_j} \to \y$ in $\mathrm{L}^2(0,T;\mathbb{U})$.
\end{proof}

\begin{lemma}\label{lem-compact-1}
	For $d\in\{2,3\}$ with $r\in(3,\infty)$, the operator $\ell\in\mathcal{ L}(\V,\mathbb{U}) $ is compact and its Nemytskii operator $\overline{\ell}: \mathcal{W} \to\mathrm{L}^2(0,T;\mathbb{U})$ defined by $(\overline{\ell} \y)(t) = \ell\y(t)$ is compact.
\end{lemma}
\begin{proof}
	Follows in the similar lines of the proof of Lemma \ref{lem-compact}.	
\end{proof}

%

The Problem \ref{prob-inequality} can be  formulated as the following inclusion problem: 
\begin{problem}\label{prob-inclusion}
	Find $\y\in\mathcal{W}$ such that  in $\V^{\prime}+\L^{\frac{r+1}{r}},$ 
	\begin{equation}\label{eqn-abs-3}
		\left\{
		\begin{aligned}
			&\y'(t)+\mu\A\y(t)+\B(\y(t))+\alpha\y(t)+\beta\mathfrak{C}(\y(t))+\ell^*\partial J(\ell\y(t))\ni\f(t), \  \mbox{ for  a.e. }\ t\in(0,T),\\
			&\y(0)=\y_0,
		\end{aligned}
		\right.
	\end{equation}
where $\partial J(\ell\y(t))=\partial J(t,\ell\y(t))$ and $\ell^*:\mathbb{U}'\to\V^{\prime}$ is the adjoint operator to $\ell.$ 
\end{problem}

\begin{remark}
	 If the functional $J$ is of the form given in  \eqref{eqn-J} and Hypothesis \ref{hyp-j} holds, then it is clear that every solution to \eqref{eqn-abs-3} is also a solution to the inequality \eqref{eqn-abs-2}. If either $j$ or $-j$ is regular, then the converse is also true.
\end{remark}

Let us now provide an  equivalent formulation of the problem \eqref{prob-inclusion}.
\begin{problem}\label{prob-inclusion-1}
	Find $(\y,\boldsymbol{\eta})\in\mathcal{W}\times\mathrm{L}^2(0,T;\mathbb{L}^2(\Gamma))$ such that   in $\V^{\prime}+\L^{\frac{r+1}{r}},$ 
	\begin{equation}\label{eqn-abs-4}
		\left\{
		\begin{aligned}
			&\y'(t)+\mu\A\y(t)+\B(\y(t))+\alpha\y(t)+\beta\mathfrak{C}(\y(t))+\ell^*\boldsymbol{\eta}(t)=\f(t), \  \mbox{ for  a.e. }\ t\in(0,T),\\
			&\boldsymbol{\eta}(t)\in \partial J(\ell\y(t)), \  \mbox{ for  a.e. }\ t\in(0,T),\\
			&\y(0)=\y_0\in\H. 
		\end{aligned}
		\right.
	\end{equation}
\end{problem}
Let us now state the main result on the existence, uniqueness and continuous dependence of solutions of the  Problem \ref{prob-inclusion}. 
\begin{theorem}\label{thm-main-1}
	Under Hypothesis \ref{hyp-j}, for $\y_0\in\H,$ $\f\in\mathrm{L}^2(0,T;\V^{\prime})$ and $\mu> C_1\|\ell\|_{\mathcal{L}(\V,\mathbb{U})}^2,$ for some positive constant $C_1,$ the Problem \ref{prob-inclusion-1} has a weak solution $(\y,\boldsymbol{\eta})\in\mathcal{W}\times\mathrm{L}^2(0,T;\mathbb{L}^2(\Gamma))$. 
	
	Moreover, for $d=2$ with $r\in[1,\infty)$ and $d=3$ with $r\in[3,\infty)$, if $\mu> \max\{C_{1},m\}\|\ell\|_{\mathcal{L}(\V;\U)}^2$ ($\mu>\max\left\{\frac{1}{2\beta}+m\|\ell\|_{\mathcal{L}(\V;\U)}^2, C_{1}\|\ell\|_{\mathcal{L}(\V;\U)}^2\right\}$ for $d=r=3$), then the weak solution to Problem \ref{prob-inclusion} is unique and the mapping $(\f,\y_0) \mapsto  \y$ is Lipschitz continuous from $\mathrm{L}^2(0,T;\V^{\prime}) \times\H$ to $\C([0,T];\H)$.
\end{theorem}

\medskip\noindent
\textbf{Acknowledgments:} The first author gratefully acknowledges the Ministry of Education, Government of India (Prime Minister Research Fellowship, PMRF ID: 2803609), for financial support to carry out her research work, and second author would like to thank Ministry of Education, Government of India - MHRD for financial assistance.  Support for M. T. Mohan's research received from the National Board of Higher Mathematics (NBHM), Department of Atomic Energy, Government of India (Project No. 02011/13/2025/NBHM(R.P)/R\&D II/1137).

\end{document}